\documentclass[1p]{elsarticle}
\usepackage[colorlinks=true,linkcolor=blue,pdfborder={0 0 0}]{hyperref}

\journal{Theoretical Computer Science}

\usepackage{graphicx}
\graphicspath{{./fig/}} 
\usepackage{subcaption} 
\usepackage{tabu}
\usepackage{booktabs}

\usepackage{amsthm}  
\newtheorem{theorem}{Theorem}
\newtheorem{proposition}[theorem]{Proposition}
\newtheorem{corollary}[theorem]{Corollary}
\newtheorem{lemma}[theorem]{Lemma}
\theoremstyle{definition}
\newtheorem{definition}[theorem]{Definition}

\usepackage{xcolor}
\usepackage{ifthen}
\newcommand{\markupdraft}[2]{
    \ifthenelse{\equal{#1}{display}}{#2}{}
    \ifthenelse{\equal{#1}{color}}{\color{#2}}{}
}
\newcommand{\notecolored}[3][]{\markupdraft{display}{{\color{#2}\noindent[Note (#1): #3]}}}
\newcommand{\newcolored}[3][]{{\markupdraft{color}{#2}#3}
    \ifthenelse{\equal{#1}{}}{}{\markupdraft{display}{{\color{yellow!70!black}[#1]}}}}
\newcommand{\del}[2][]{{\markupdraft{display}{{\color{orange}[removed: ``#2''[#1]]}}}} 
\newcommand{\new}[2][]{\newcolored[#1]{blue}{#2}}
\newcommand{\nnew}[2][]{\newcolored[#1]{red}{#2}}
\newcommand{\note}[2][]{\notecolored[#1]{green}{#2}}

\newcommand{\todo}[2][]{\markupdraft{display}{{\color{red}\noindent++TODO: #2 {\color{yellow}(#1)}++}}}
\renewcommand{\del}[1]{}  
\renewcommand{\markupdraft}[2]{}  

\newcommand{\niko}[1]{\note[Niko]{\color{brown} #1}}

\usepackage{latexsym}
\usepackage{amsmath,amssymb} 
\usepackage{mathtools}       
\usepackage{stmaryrd}        
\usepackage{bm,bbm}              

\providecommand{\argmax}{\operatornamewithlimits{argmax}} 
\DeclareMathOperator{\Tr}{Tr}     
\DeclareMathOperator{\Cov}{Cov}   
\DeclareMathOperator{\Cond}{Cond} 
\DeclareMathOperator{\diag}{diag} 
\providecommand{\R}{\mathbb{R}} 
\providecommand{\E}{\mathbb{E}} 
\providecommand{\T}{\mathrm{T}} 
\providecommand{\rmd}{\mathrm{d}} 
\providecommand{\ind}[1]{\mathbbm{1}_{#1}} 
\renewcommand{\geq}{\geqslant} 
\renewcommand{\leq}{\leqslant} 
\DeclarePairedDelimiterX{\inner}[2]{\langle}{\rangle}{#1, #2}
\DeclarePairedDelimiter{\norm}{\lVert}{\rVert}
\DeclarePairedDelimiter{\abs}{\lvert}{\rvert}

\let\oldsqrt\sqrt
\renewcommand{\sqrt}[2][]{\oldsqrt[\leftroot{-2}\uproot{2}#1]{#2}}

\newcommand{\eye}{\mathbf{I}}

\newcommand{\cm}{c_{m}}
\newcommand{\Hess}{\mathbf{A}}

\newcommand{\ns}{\bar{\sigma}}
\newcommand{\ee}{\mathbf{e}}
\newcommand{\xmean}{{\bm{m}}}
\newcommand{\NN}{\mathcal{N}}
\newcommand{\eig}{d}
\newcommand{\muw}{\mu_w}

\providecommand{\Hesst}{\Hess}
\providecommand{\eone}{\bm{n}}
\providecommand{\etwo}{\mathbf{N}}
\providecommand{\asynqg}{\varphi}
\renewcommand{\textstyle}{}

\providecommand{\rev}[1]{{#1}}
\usepackage{algpseudocode}
\usepackage{algorithm}

\begin{document}

\begin{frontmatter}

\title{Quality Gain Analysis of the Weighted Recombination Evolution Strategy on General Convex Quadratic Functions\tnoteref{mytitlenote}}
\tnotetext[mytitlenote]{This is the extension of our extended abstract presented at FOGA'2017 \cite{Akimoto2017foga}.}

\author[tsukuba]{Youhei Akimoto\corref{mycorrespondingauthor}}\ead{akimoto@cs.tsukuba.ac.jp}
\author[inria]{Anne Auger}\ead{anne.auger@inria.fr}
\author[inria]{Nikolaus Hansen}\ead{nikolaus.hansen@inria.fr}
\address[tsukuba]{Faculty of Engineering, Information and Systems, University of Tsukuba, Japan}
\address[inria]{Inria, RandOpt Team, CMAP, Ecole Polytechnique, France}
\cortext[mycorrespondingauthor]{Corresponding author}

\begin{abstract}
  Quality gain is the expected relative improvement of the function value in a single step of a search algorithm. Quality gain analysis reveals the dependencies of the quality gain on the parameters of a search algorithm, based on which one can derive the optimal values for the parameters. In this paper, we investigate evolution strategies with weighted recombination on general convex quadratic functions. We derive a bound for the quality gain and two limit expressions of the quality gain. From the limit expressions, we derive the optimal recombination weights and the optimal step-size, and find that the optimal recombination weights are independent of the Hessian of the objective function. Moreover, the dependencies of the optimal parameters on the dimension and the population size are revealed. Differently from previous works where the population size is implicitly assumed to be smaller than the dimension, our results cover the population size proportional to or greater than the dimension. \rev{Numerical simulation shows that the asymptotically optimal step-size well approximates the empirically optimal step-size for a finite dimensional convex quadratic function.}
\end{abstract}


\end{frontmatter}

\section{Introduction}

\paragraph{Background}

  Evolution Strategies (ES) are randomized search algorithms to minimize a black-box function $f:\R^N \to \R$ in continuous domain, where neither the gradient nor the Hessian matrix of the objective function is available. The most advanced and commonly used category of evolution strategies is covariance matrix adaptation evolution strategy (CMA-ES) \cite{Hansen2004ppsn,Hansen:2013vt}, which is recognized as the state-of-the-art black box continuous optimizer. It generates multiple \rev{candidate} solutions from a multivariate normal distribution. They are evaluated on the objective function. The distribution parameters such as the mean vector and the covariance matrix are updated by using the candidate solutions and their ranking information, where the objective function values are not directly used. Due to its population-based and comparison-based nature, the algorithm is invariant to any strictly increasing transformation of the objective function in addition to the invariance to scaling, translation, and rotation of the search space \cite{Hansen2000ppsn}. These invariance properties guarantee that the algorithm shows exactly the same behavior on a function $f$ and on its transformation $g \circ f \circ T$, where $g: \R \to \R$ is a strictly increasing function and $T: \R^N \to \R^N$ is a combination of scaling, translation and rotation defined as $T: \bm{x} \mapsto a \cdot \mathbf{U} (\bm{x} - \bm{b})$ with a positive real $a > 0$, an $N$ dimensional vector $\bm{b} \in \R^N$, and an $N$ dimensional orthogonal matrix $\mathbf{U}$. These invariance properties are the essence of the success of CMA-ES.

The performance evaluation of evolutionary algorithms is often based on empirical studies such as benchmarking on a test function suite \cite{Hansen2010geccobbobres,rios2013derivative} and well-considered performance assessment \cite{hansen2014assess,Krause2017foga}. It is easier to check the performance of an algorithm on a specific problem in simulation than to analyze it mathematically. The invariance properties of an algorithm then generalize the empirical result to a class of infinitely many functions defined by the invariance relation. On the other hand, theoretical studies often require simplification of algorithms and assumptions on the objective function, because of the difficulty of the analysis of advanced algorithms due to their comparison-based and population-based nature and the complex adaptation mechanisms. Nevertheless, theoretical studies lead us to a better understanding of algorithms and reveal the dependency of the performance on the interval parameter settings. For example, the recombination weights in CMA-ES are selected based on the mathematical analysis of an evolution strategy \cite{Arnold2005foga}\footnote{The weights of CMA-ES were set before the publication \cite{Arnold2005foga} because the theoretical result of optimal weights on the sphere was known before the publication.}. The theoretical result of the optimal step-size on the sphere function is used to design a box constraint handling technique \cite{Hansen2009tec} and to design a termination criterion for a restart strategy \cite{Yamaguchi2017BBOB}. A recent variant of CMA-ES \cite{Akimoto2016ppsn} exploits the theoretical result of the optimal rate of convergence of the step-size to estimate the condition number of the product of the covariance matrix and the Hessian matrix of the objective function.

\paragraph{Quality Gain Analysis}

\emph{Quality gain} and \emph{progress rate} analysis \cite{Rechenberg1994,Beyer1994ppsn,BeyerBOOK2001} measure the expected progress of the mean vector in one step. On one side, differently from convergence analysis (e.g., \cite{Auger2005tcs}), analyses based on these quantities do not guarantee the convergence and often take a limit to derive an explicit formula. Moreover, the step-size adaptation and the covariance matrix adaptation are not taken into account. On the other side, one can derive quantitative explicit estimates of these quantities, which is not the case in convergence analysis. The quantitative explicit formulas are particularly useful to know the dependency of the expected progress on the parameters of the algorithm such as the population size, number of parents, and recombination weights, which we may not recognize from empirical studies of algorithms. The above mentioned recombination weights in CMA-ES are derived from the quality gain analysis of evolution strategies \cite{Arnold2005foga}. 

Although the quality gain analysis is not meant to guarantee the convergence of the algorithm since it analyzes only a single step expected improvement, the progress rate is linked to the convergence rate of algorithms. It is directly related to the convergence rate of an ``artificial'' algorithm where the step-size is proportional to the distance to the optimum on the sphere function (see e.g., \cite{Auger2006gecco}). Moreover, the convergence rate of this artificial algorithm gives a bound on the convergence rate of algorithms that implement a proper step-size adaptation. For $(1+\lambda)$ or $(1,\lambda)$ ESs the bound holds on any function with a unique global optimum; that is, any step-size adaptive $(1 \stackrel{+}{,} \lambda)$-ES optimizing any function $f$ with a unique global optimum can not achieve a convergence rate faster than the convergence rate of the artificial algorithm on the sphere function where the step-size is the distance to the optimum times the optimal constant \cite{jebalia2008log,jebalia2010log,Auger2015}\footnote{More precisely, $(1 \stackrel{+}{,} \lambda)$-ES optimizing any function $f$ (that may have more than one global optimum) can not converge towards a given optimum $x^*$ faster in the search space than the artificial algorithm with step-size proportional to the distance to $x^*$.}. For algorithms implementing recombination, this bound still holds on spherical functions \cite{jebalia2010log,Auger2015}.

\paragraph{Related Work}

In this paper, we investigate ESs with weighted recombination on a general convex quadratic function. ESs with weighted recombination samples multiple candidate solutions at one time and compute the weighted average of the candidate solutions to update the distribution mean vector. Weighted recombination ESs are among the most important categories of ESs since the standard CMA-ES and most of the recent variants of CMA-ES \cite{Ros2008ppsn,Loshchilov2014,Akimoto2016gecco} employ weighted recombination. 

The first analysis of weighted recombination ESs was done in \cite{Arnold2005foga}, where the quality gain has been derived on the infinite dimensional sphere function $f: x \mapsto \norm{x}^2$. The optimal step-size and the optimal recombination weights are derived.
Reference~\cite{Arnold2007gecco} studied a variant of weighted recombination ESs called $(\mu/\mu_I, \lambda)$-ES\rev{, where $(\mu/\mu_I, \lambda)$ stands for intermediate recombination, where the recombination weights are equal for the best $\mu$ candidate solutions and zero for the other $\lambda - \mu$ candidate solutions. The analysis has been performed on the quadratic functions with the Hessian $\Hess = \frac12 \diag(\alpha, \dots, \alpha, 1, \dots, 1)$, where the number $\lfloor N \theta\rfloor$ of diagonal elements that are $\alpha > 1$ is controlled by the ratio $\theta$ of short axes}. Reference~\cite{Jagerskupper:2006cf} studied the $(1+1)$-ES with the one-fifth success rule on the same function and showed the convergence rate of $\Theta(1/(\alpha N))$. Reference~\cite{Finck2009foga} studied ES with weighted recombination on the same function. Their results, progress rate and quality gain, depend on the so-called localization parameter, the steady-state value of which is then analyzed to obtain the steady-state quality gain. References~\cite{beyer2014dynamics,Beyer2016ec} studied the progress rate and the quality gain of $(\mu/\mu_I, \lambda)$-ES on the general convex quadratic model. 

The quality gain analysis and the progress rate analysis in the above listed references rely on a geometric intuition of the algorithm in the infinite dimensional search space and on various approximations. On the other hand, the rigorous derivation of the progress rate (or convergence rate of the algorithm with step-size proportional to the \rev{distance to the} optimum) on the sphere function provided for instance in \cite{Auger2006gecco,Auger2015,abh2011b,jebalia:inria-00495401} only holds on spherical functions and provides solely a limit without a bound between the finite dimensional convergence rate and its asymptotic limit. \rev{The result of this paper is different in that we consider the general weighted recombination on the general convex quadratic objective and cover finite dimensional cases as well as the limit $N \to \infty$.}

\providecommand{\asy}{asymptotic normalized quality gain}
\providecommand{\nqglim}{normalized quality gain limit}
\paragraph{Contributions}

We study the weighted recombination ES on a general convex quadratic function $f(x) = \frac12(x - x^*)^\T \Hess (x - x^*)$ on the finite $N$ dimensional search space. We investigate the quality gain $\phi$, that is, the expectation of the relative function value decrease. We decompose $\phi$ as the product of two functions: $g$\new{, a function that} depends only on the mean vector of the sampling distribution and the Hessian $\Hess$, and $\bar{\phi}$, the so-called \emph{normalized quality gain} that depends essentially on all the \new{algorithm} parameters such as the recombination weights and the step-size. We approximate $\bar{\phi}$ by an analytically tractable function $\asynqg$. We call $\asynqg$ the \emph{\asy}. The main contributions are summarized as follows.

First, we derive the error bound between $\bar{\phi}$ and $\asynqg$ for finite dimension $N$. To the best of our knowledge, this is the first work that performs the quality gain analysis for finite $N$ and provides \new{an} error bound. The asymptotic normalized quality gain and the bounds in this paper are improved over the previous work \cite{Akimoto2017foga}. Thanks to the explicit error bound derived in the paper, we can treat the population size $\lambda$ increasing with $N$ and provide (for instance) a rigorous sufficient condition on the dependency between $\lambda$ and $N$ such that the per-iteration quality gain scal\new{es with} $O(\lambda/N)$ for algorithms with intermediate recombination \cite{BeyerBOOK2001}.

Second, we show that the error bound between $\bar{\phi}$ and $\asynqg$ converges to zero as the learning rate $\cm$ for the mean vector update tends to infinity. We derive the optimal step-size and the optimal recombination weights for $\asynqg$, revealing the dependencies of these optimal parameters on $\lambda$ and $N$. In contrast, the previous works of quality gain analysis mentioned above take the limit $N \to \infty$ while $\lambda$ is fixed, \new{hence assuming} $\lambda \ll N$. Therefore, they do not reveal the dependencies of $\bar{\phi}$ and the optimal parameters on $\lambda$ when $\lambda \not\ll N$. We validate in experiments that the optimal step-size derived for $\cm \to \infty$ provides a reasonable estimate of the optimal step-size even for $\cm = 1$.

Third, we prove that $\asynqg$ converges toward $\bar{\phi}_{\infty}$ as $N \to \infty$ under the condition $\lim_{N\to\infty}\Tr(\Hess^2) / \Tr(\Hess)^2 = 0$, where $\bar{\phi}_{\infty}$ is the limit of $\bar{\phi}$ on the sphere function for $N \to \infty$ derived in \cite{Arnold2005foga}. \rev{The condition $\lim_{N\to\infty}\Tr(\Hess^2) / \Tr(\Hess)^2$ holds, for example, for positive definite $\Hess$ with bounded eigenvalues. It also holds for some positive semi-definite $\Hess$ and for some positive definite $\Hess$ with unbounded eigenvalues, for example with eigenvalues in $[1, \sqrt N]$.}\todo{please verify} The result implies that the optimal recombination weights are independent of $\Hess$, whereas the optimal step-size heavily depends on $\Hess$ and the distribution mean. This part of the contribution is a generalization of the previous foundation in \cite{beyer2014dynamics,Beyer2016ec}, but the proof methodology is rather different. Furthermore, the error bound between $\bar{\phi}$ and $\asynqg$ derived in this paper allows us to further investigate how fast $\asynqg$ converges toward $\bar{\phi}_\infty$ as $N \to \infty$, depending on the eigenvalue distribution of $\Hess$. 

\paragraph{Organization} This paper is organized as follows. In Section~\ref{sec:form}, we formally define the evolution strategy with weighted recombination. The quality gain analysis on the infinite dimensional sphere function is revisited. In Section~\ref{sec:quad}, we derive the quality gain bound for a finite dimensional convex quadratic function. In Section~\ref{sec:cons}, important consequences of the quality gain bound are discussed. In Section~\ref{sec:conc}, we conclude our paper. Properties of the normal order statistics that are important to understand our results are summarized in \ref{apdx:nos} and the detailed proofs of lemmas are provided in \ref{apdx:proofs}.

\paragraph{Notation} We apply the following mathematical notations throughout the paper. For integers $n$, $m \in \mathbb{N}$ such that $n \leq m$, we denote the set of integers between $n$ and $m$ (including $n$ and $m$) by $\llbracket n, m \rrbracket$. Binomial coefficients are denoted as $\binom{m}{n} = \frac{m!}{(m-n)!n!}$. For real numbers $a$, $b \in \R$ such that $a \leq b$, the open and the closed intervals are denoted as $(a, b)$ and $[a, b]$, respectively. For an $N$-dimensional real vector $x \in \R^N$, let $[x]_i$ denote the $i$-th  coordinate of $x$.
\rev{A sequence of length $n$ is denoted as $(x_i)_{i=1}^{n} = (x_1, \cdots, x_n)$, or just as $(x_i)$, and an infinite sequence is denoted as $(x_i)_{i=1}^\infty$.}
For $x \in \R$, the absolute value of $x$ is denoted by $\abs{x}$. For $x \in \R^N$, the Euclidean norm is denoted by $\norm{x} = \big(\sum_{i=1}^{N}[x]_i^2\big)^\frac12$. Let $\ind{\texttt{condition}}$ be the indicator function which is $1$ if \texttt{condition} is true and $0$ otherwise. Let $\Phi$ be the cumulative density function (c.d.f.) deduced by the (one-dimensional) standard normal distribution $\mathcal{N}(0, 1)$.  Let $\NN_{i:\lambda}$ be the $i$-th smallest random variable among $\lambda$ independently and standard normally distributed random variables, i.e., $\NN_{1:\lambda} \leq \cdots \leq \NN_{\lambda:\lambda}$. The expectation of a random variable (or vector) $X$ is denoted as $\E[X]$. The conditional expectation of $X$ given $Y$ is denoted as $\E[X \mid Y]$. For a function $f$ of $n$ random variables $(X_i)_{i=1}^{n}$, the conditional expectation of $F = f(X_1,\dots,X_n)$ given $X_k$ for some $k \in \llbracket 1, n \rrbracket$ is denoted as $\E_{k}[F] = \E[F \mid X_k]$. Similarly, the conditional expectation of $F$ given $X_k$ and $X_l$ for different $k, l \in \llbracket 1, n \rrbracket$ is denoted as $\E_{k,l}[F] = \E[F \mid X_k, X_l]$. 

\section{Formulation}
\label{sec:form}

\subsection{Evolution Strategy with Weighted Recombination}

We consider an evolution strategy with weighted recombination. At each iteration $t \geq 0$, it draws $\lambda$ independent random vectors $Z_1, \dots, Z_\lambda$ from the $N$-dimensional standard normal distribution $\mathcal{N}(\bm{0}, \eye)$, where $\bm{0} \in \R^N$ is the zero vector and $\eye$ is the identity matrix of dimension $N$. The candidate solutions $X_1, \dots, X_\lambda \sim \mathcal{N}(\xmean^{(t)}, (\sigma^{(t)})^2\eye)$ are computed as $X_i = \xmean^{(t)} + \sigma^{(t)} Z_i$, where $\xmean^{(t)} \in \R^N$ is the mean vector and $\sigma^{(t)} > 0$ is the standard deviation, also called the step-size or the mutation strength. The \rev{candidate} solutions are evaluated on a given objective function $f:\R^N \to \R$. Without loss of generality (w.l.o.g.), we assume $f$ to be minimized. Let $i:\lambda$ be the index of the $i$-th best \rev{candidate} solution among $X_1, \dots, X_\lambda$, i.e., $f(X_{1:\lambda}) \leq \cdots \leq f(X_{\lambda:\lambda})$, and $w_1 \geq \cdots \geq w_\lambda$ be the real-valued recombination weights. W.l.o.g., we assume $\sum_{i=1}^{\lambda} \abs{w_i} = 1$. Let $\muw = 1 / \sum_{i=1}^{\lambda} w_i^2$ denote the so-called effective variance selection mass. The mean vector is updated according to 
\begin{equation}
\xmean^{(t+1)} = \xmean^{(t)} + \cm \sum_{i=1}^{\lambda} w_i (X_{i:\lambda} - \xmean^{(t)}) \enspace,
\label{eq:m}
\end{equation}%
where $\cm > 0$ is the learning rate of the mean vector update. 

\rev{In this paper we reformulate \eqref{eq:m} to investigate the algorithm with mathematical rigor. Hereunder, we write the candidate solutions, $X_1, \dots, X_\lambda$, and the corresponding random vectors, $Z_1, \dots, Z_\lambda$, as sequences $(X_i)_{i=1}^{\lambda}$ and $(Z_i)_{i=1}^{\lambda}$ for short.} First, we introduce the weight function
\begin{align}
  W(i; (X_k)_{k=1}^{\lambda}) &:= \sum_{k=1 + l_i}^{u_i} \frac{w_k}{u_i - l_i} \enspace, &&
  \text{where} 
  &\begin{cases}
    l_i = \sum_{j=1}^{\lambda} \ind{f(X_j) < f(X_i)}  \\
    u_i = \sum_{j=1}^{\lambda} \ind{f(X_j) \leq f(X_i)} 
  \end{cases},
\label{eq:w}
\end{align}%
i.e., $l_i$ and $u_i$ are the numbers of strictly and weakly better \rev{candidate} solutions than $X_i$, respectively. The weight value for $X_i$ is the arithmetic average of the weights $w_k$ for the tie \rev{candidate} solutions. In other words, all the tie \rev{candidate} solutions have the same weight values. If there is no tie, the weight value for the $i$-th best \rev{candidate} solution $X_{i:\lambda}$ is simply $w_i$. In the following, we drop the subscripts and the superscripts for sequences unless they are unclear from the context and write simply as $(X_k) = (X_k)_{k=1}^{\lambda}$. With the weight function, we rewrite the mean vector update \eqref{eq:m} as
\begin{equation}
\xmean^{(t+1)} = \xmean^{(t)} + \cm \sum_{i=1}^{\lambda} W(i; (X_k))(X_{i} - \xmean^{(t)}) \enspace\rev{.}
\label{eq:mmz}
\end{equation}%
The above update \eqref{eq:mmz} is equivalent with the original update \eqref{eq:m} if there is no tie among $\lambda$ \rev{candidate} solutions. If the objective function is a convex quadratic function, there will be no tie with probability one\del{ (w.p.1)}. Therefore, they are equivalent \rev{with probability one}. \rev{Algorithm~\ref{alg:es} summarizes the single step of the algorithm, where we rewrite \eqref{eq:mmz} by using $X_{i} - \xmean^{(t)} = \sigma^{(t)} Z_i$.}

\rev{The above formulation is motivated twofold.} One is to well define the update even when there is a tie. \rev{In our formulation, tie candidate solutions receive the equal recombination weights.} The other is a technical reason. In \eqref{eq:m} the already sorted \rev{candidate} solutions $X_{i:\lambda}$ are all correlated and they are not anymore normally distributed. However, they are assumed to be normally distributed in the previous work \cite{Arnold2005foga,beyer2014dynamics,Beyer2016ec}. To ensure that such an approximation leads to the asymptotically true quality gain limit, a mathematically involved analysis has to be done. See \cite{Auger2006gecco,abh2011b,jebalia:inria-00495401} for details. In \eqref{eq:mmz}, the weight function explicitly includes the ranking computation and $X_i$ are still independent and normally distributed. This allows us \rev{to derive the quality gain on a convex quadratic function rigorously.}

\begin{algorithm}[t]
  \caption{Single step of the weighted recombination ES solving $f$}
  \label{alg:es}
  \begin{algorithmic}[1] 
    \Procedure{ES}{$\xmean, \sigma, (w_k)_{k=1}^{\lambda}, \cm$}
    \For{$i=1,\dots,\lambda$} \Comment{Generate and evaluate $\lambda$ candidate solutions}
    \State $Z_i \sim \mathcal{N}(\bm{0}, \eye)$
    \State $X_i = \xmean + \sigma Z_i$
    \State Evaluate $f(X_i)$
    \EndFor 
    \State $W(i; (X_k)_{k=1}^{\lambda}) = \sum_{k=1 + l_i}^{u_i} w_k / (u_i - l_i)$ \Comment{Compute the weights with \eqref{eq:w}}
    \State $\xmean \gets \xmean + \cm \sigma \sum_{i=1}^{\lambda} W(i; (X_k)_{k=1}^{\lambda}) Z_{i}$ \Comment{Update the mean with \eqref{eq:mmz}}
    \State \textbf{return} $\xmean$
    \EndProcedure
  \end{algorithmic}
\end{algorithm}

\subsection{Quality Gain Analysis on the Spherical Function}

The quality gain is defined as the expectation of the relative decrease of the function value. Formally, it is the conditional expectation of the relative decrease of the function value conditioned on the mean vector $\xmean^{(t)} = \xmean$ and the step-size $\sigma^{(t)} = \sigma$\rev{, defined as follows.}
\rev{\begin{definition}
  The \emph{quality gain} of Algorithm~\ref{alg:es} given $\xmean^{(t)} = \xmean$ and $\sigma^{(t)} = \sigma$ is
\begin{equation}
\phi(\xmean, \sigma) = \frac{\E[f(\xmean) - f(\textsc{ES}(\xmean, \sigma, (w_k)_{k=1}^{\lambda}, \cm))]}{f(\xmean)  - f(x^*)} \enspace,
\label{eq:sph:qg}
\end{equation}%
where $x^* \in \R^N$ is (one of) the global minimum point of $f$. Note that the quality gain depends also on $(w_k)_{k=1}^{\lambda}$, $\cm$, and the dimension $N$.
\end{definition}}%

\paragraph{Results}
\rev{Algorithm~\ref{alg:es}} solving a spherical function $f(x) = \norm{x}^2$ is analyzed in \cite{Arnold2005foga}. For this purpose, the \emph{normalized step-size} and the \emph{normalized quality gain} are introduced
as 
  \begin{align}
  \ns &= \sigma \,\frac{\cm N}{\norm{\xmean}} &&\text{and} & \bar{\phi}(\xmean, \ns) &= \frac{N}{2} \phi\left(\xmean, \sigma = \frac{\ns\norm{\xmean}}{\cm N}\right) \enspace,
\label{eq:sph:ns}
  \end{align}%
  respectively. \rev{This normalization of the step-size suggests that $\sigma$ is proportional to $\norm{\xmean}$ and inverse proportional to $\cm$ and to $N$. The normalized step-size $\ns$ is proportional to the ratio between the actual step-size and the distance between the current mean and the optimal solution. This reflects the scale invariance of the algorithm on the sphere function, that is, the single step response is solely determined by the normalized step-size. The dimension $N$ in the numerator implies that the step-size $\sigma$ needs to be inversely proportional to $N$. The normalized quality gain $\bar\phi$ is simply the quality gain given $\ns$ scaled by $N / 2$. The scaling by $N$ reflects that the convergence speed can not exceed $O(1/N)$ for any comparison based algorithm \cite{Teytaud2006ppsn}.} By taking $N \to \infty$, the normalized quality gain converges \rev{pointwise (w.r.t.~$\ns$)} to 
\begin{align}
    \bar{\phi}_{\infty}(\ns, (w_k))
    &:= 
    \rev{\lim_{N\to\infty}\bar{\phi}(\xmean, \ns)} 
    = - \ns \sum_{i=1}^{\lambda} w_i \E[\NN_{i:\lambda}] - \frac{\ns^2}{2\muw} \nonumber\\
    &= \frac{\muw}{2}\left(\sum_{i=1}^{\lambda} w_i \E[\NN_{i:\lambda}]\right)^2\Bigg(1 -  \bigg(\frac{\ns}{\ns^*((w_k))} - 1\bigg)^2\Bigg)
\enspace,
\label{eq:nqlim}
\end{align}%
where $\ns^*((w_k))$ denotes the normalized step-size $\ns$ optimizing $\bar{\phi}_{\infty}$ given $(w_k)$ and is given by 
\begin{equation}
\ns^*((w_k)) = - \muw \sum_{i=1}^{\lambda} w_{i}\E[\mathcal{N}_{i:\lambda}] \enspace.
\label{eq:optns}
\end{equation}%
A formal proof of this result is presented in \rev{Theorem~2 of \cite{abh2011b} relying on the uniform integrability of some random variable proved in \cite{jebalia:inria-00495401}}. 

Consider the optimal recombination weights that maximize $\bar{\phi}_{\infty}$ in \eqref{eq:nqlim}. The optimal recombination weights are given independently of $\ns$ by
\begin{equation}
    w_k^* = - \frac{\E[\mathcal{N}_{k:\lambda}]}{\sum_{i=1}^{\lambda} \abs{\E[\mathcal{N}_{i:\lambda}]}}
\label{eq:optw}
\end{equation}%
and $\bar{\phi}_{\infty}$ is written as 
\begin{equation}
    \bar{\phi}_{\infty}(\ns, (w_k^*)) = \frac{\sum_{i=1}^{\lambda} \E[\mathcal{N}_{i:\lambda}]^2}{2}\left( 1  - \left( \frac{\ns}{\sum_{i=1}^{\lambda} \abs{\E[\mathcal{N}_{i:\lambda}]}} - 1 \right)^2\right) \enspace.
\label{eq:optw-nqg}
\end{equation}%
Note that $\ns^*((w_k^*)) = \sum_{i=1}^\lambda \abs{\E[\mathcal{N}_{i:\lambda}]}$. 
Given $\ns^*$ and $(w_k^*)$, we achieve the maximal value of $\bar{\phi}_{\infty}$ that is $\bar{\phi}_{\infty}(\ns^*((w_k^*)), (w_k^*)) = \sum_{i=1}^{\lambda} \E[\mathcal{N}_{i:\lambda}]^2 / 2$\del{ and is roughly estimated by $\lambda / 2$ if $\lambda$ is sufficiently large, say $\lambda \geq 10^2$. Therefore, $\phi$ is approximated by $\lambda / N$}.

\begin{figure}[t]
\centering
\includegraphics[width=0.65\linewidth]{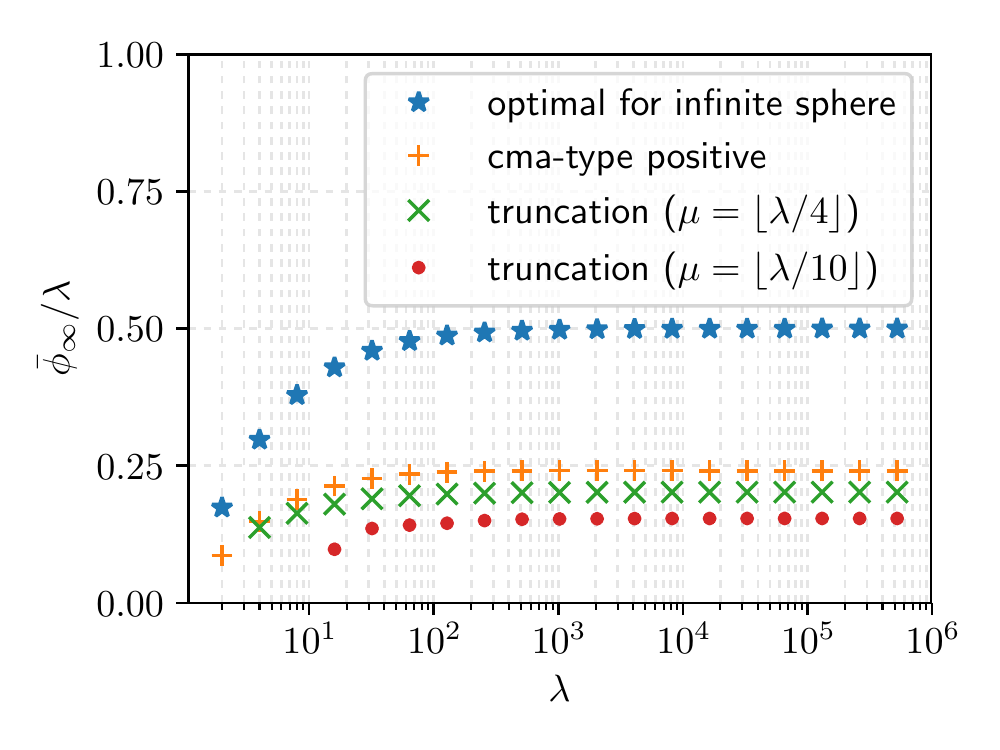}
\caption{The normalized quality gain limit $\bar{\phi}_{\infty}(\ns^*((w_k)), (w_k))$ divided by $\lambda$. Four different weight schemes are employed: the optimal weights ($w_k \propto - \E[\NN_{k:\lambda}]$), the weights used in the CMA-ES ($w_k \propto \max(\ln\big(\frac{\lambda+1}{2}\big) - \ln\big(k\big), 0)$), and the truncation weights ($w_k = 1/\mu$ for $k = 1, \dots, \mu$ and $w_k = 0$ for $k = \mu + 1,\dots,\lambda$) with $\mu = \lfloor\lambda / 4\rfloor$ and $\mu = \lfloor\lambda / 10\rfloor$. All the weights are scaled so that $\sum_{k=1}^{\lambda}\abs{w_k} = 1$. The value of $\E[\mathcal{N}_{i:\lambda}]$ is approximated by the Blom's formula (see \ref{apdx:nos}).}
\label{fig:inf-optsig-lam}
\end{figure}

\paragraph{Remarks}

The optimal normalized step-size \eqref{eq:optns} and the normalized quality gain \eqref{eq:nqlim} given $\ns^*$ depend on $(w_k)$. Particularly, they are proportional to $\muw$. For instance, \rev{%
  under the optimal weights \eqref{eq:optw}, we have $\rev{-}\sum_{i=1}^{\lambda} w_{i}\E[\mathcal{N}_{i:\lambda}] = \sum_{i=1}^{\lambda}\E[\mathcal{N}_{i:\lambda}]^2 / \sum_{i=1}^{\lambda}\abs{\E[\mathcal{N}_{i:\lambda}]} \approx (\pi / 2)^{1/2}$ for a sufficiently large $\lambda$\footnote{We used the facts $\lim_{\lambda\to\infty} \sum_{i=1}^{\lambda} \E[\mathcal{N}_{i:\lambda}]^2 / \lambda = 1$ and $\lim_{\lambda\to\infty} \sum_{i=1}^{\lambda} \abs{\E[\mathcal{N}_{i:\lambda}]} / \lambda = (2/\pi)^{1/2}$. See \ref{apdx:nos} for details.}. Then, from \eqref{eq:nqlim} and \eqref{eq:optns} we know $\ns \propto \muw$ and $\bar\phi_{\infty} \propto \muw$. Moreover, using the relation $\muw = \big(\sum_{i=1}^{\lambda}\abs{\E[\mathcal{N}_{i:\lambda}]}\big)^2 / \sum_{i=1}^{\lambda}\E[\mathcal{N}_{i:\lambda}]^2  \approx (2  / \pi) \lambda$, we can reword it as that the optimal step-size and the normalized quality gain given $\ns^*$ are proportional to $\lambda$.} %
Figure~\ref{fig:inf-optsig-lam} shows how $\bar{\phi}_\infty / \lambda$ scales with $\lambda$ when the optimal step-size $\sigma = \ns^*((w_k)) \norm{\xmean} / (\cm N)$ is set. This shows that the normalized quality gain, and hence the optimal normalized step-size, are proportional to $\lambda$ for standard weight schemes. When the optimal weights are used, $\bar{\phi}_\infty / \lambda$ goes up to $0.5$ as $\lambda$ increases. On the other hand, nonnegative weights can not achieve the value of $\bar{\phi}_\infty / \lambda$ above $0.25$. \rev{The CMA type weights are designed to approximate the optimal nonnegative weights, where the first half weights are proportional to the optimal setting and the last half are zero. The truncation weights result in a smaller normalized quality gain. It is shown in \cite{BeyerBOOK2001} that the truncation weights achieve $\bar{\phi}_\infty \in O(\mu \log(\lambda/\mu))$.}

The normalized quality gain limit $\bar{\phi}_\infty$ depends only on the normalized step-size $\ns$ and the weights $(w_k)$. Since the normalized step-size does not change if we multiply $\cm$ by some factor and divide $\sigma$ by the same factor, $\cm$ does not have any impact on $\bar{\phi}_\infty$, hence on the quality gain $\phi$. This is unintuitive and is not true in a finite dimensional space. The step-size $\sigma$ realizes the standard deviation of the sampling distribution and it has an impact on the ranking of the \rev{candidate} solutions. On the other hand, the product $\sigma \cm$ is the step-size of the $\xmean$-update that depends on the ranking of the \rev{candidate} solutions. The normalized quality gain limit provided above tells us that the ranking of the \rev{candidate} solutions is independent of $\ns$ in the infinite dimensional space. We will discuss \new{this further}\del{it} in Section~\ref{sec:cons}.

The quality gain is to measure the improvement in one iteration. If we generate and evaluate $\lambda$ \rev{candidate} solutions every iteration, the quality gain per evaluation ($f$-call) is $1/\lambda$ times smaller, i.e., the quality gain per evaluation is $1/N$, rather than $\lambda/N$. It implies that the number of iterations \rev{to achieve the same amount of the quality gain} is inversely proportional to $\lambda$. \new{This is the best we can hope for}\del{It ideal} when the algorithm is implemented on a parallel computer. However, since the above result is obtained in the limit $N\to\infty$ while $\lambda$ is fixed, it is implicitly assumed that $\lambda \ll N$. \new{The optimal down scaling of the number of iterations indeed only holds}\del{Therefore, the ideal scaling down of the number of iterations holds only} for $\lambda \ll N$. In practice, the quality gain per iteration tends to level out as $\lambda$ increases. We will revisit this point in Section~\ref{sec:cons} and see how the optimal values for $\ns$ and $\bar{\phi}$ depend on $N$ and $\lambda$ when both are finite.

\section{Quality Gain Analysis on General Quadratic Functions}
\label{sec:quad}

In this section we investigate the normalized quality gain of \rev{Algorithm~\ref{alg:es}} minimizing a quadratic function with its Hessian $\nabla \nabla f(x) = \Hess$ assumed to be nonnegative definite and symmetric, i.e.,
\begin{equation}
f(x) = \frac12(x - x^*)^\T \Hess (x - x^*) \enspace,
\label{eq:quad}
\end{equation}%
where $x^* \in \R^N$ is the global optimal solution\footnote{We use the following terminology in this paper. A nonnegative definite matrix $\mathbf{A}$ is a matrix having only nonnegative eigenvalues, i.e., $x^\T \Hess x \geq 0$ for all $x \in \R^N$. \rev{A nonnegative definite matrix $\mathbf{A}$ is called positive definite if $x^\T \Hess x > 0$ for all $x \in \R^N\setminus\{\bm{0}\}$, otherwise it is called positive semi-definite.} If $\Hess$ is positive semi-definite, the optimum $x^*$ is not unique.}. \nnew{W.l.o.g., we assume $\Tr(\Hesst) = 1$\footnote{None of the algorithmic components and the quality measures used in the paper are affected by multiplying a positive constant to $\Hess$, or equivalently to $f$. To consider a general $\Hess$, simply replace $\Hess$ with $\Hess / \Tr(\Hess)$ in the following of the paper.}.} For the sake of notation simplicity we denote the directional vector of the gradient of $f$ at $\xmean$ by $\ee = \frac{\nabla f(\xmean)}{\norm{\nabla f(\xmean)}} = \frac{\Hess (\xmean - x^*)}{\norm{\Hess (\xmean - x^*)}}$. To make the dependency of $\ee$ on $\xmean$ clear, we sometimes write it as $\ee_{\xmean}$. 

\subsection{Normalized Quality Gain and Normalized Step-Size}

We introduce the normalized step-size and the normalized quality gain. First of all, if the objective function is homogeneous around the optimal solution $x^*$, the optimal step-size must be a homogeneous function of degree $1$ with respect to $\xmean - x^*$. This is formally stated in the following proposition. \rev{The proof is found in \ref{apdx:prop:optsig}.}

\begin{proposition}\label{prop:optsig}
Let $f:\R^N\to\R$ be a homogeneous function of degree $n$, i.e., $f(\alpha \cdot x) = \alpha^n f(x)$ for a fixed integer $n > 0$ for any $\alpha > 0$ and any $x \in \R^N$. Consider \rev{Algorithm~\ref{alg:es}} minimizing a function $f^*: x \mapsto f(x - x^*)$. Then, the quality gain is scale-invariant, i.e., $\phi(x^* + (\xmean - x^*), \sigma) = \phi(x^* + \alpha (\xmean - x^*), \alpha \sigma)$ for any $\alpha > 0$. Moreover, the optimal step-size $\sigma^* = \argmax_{\sigma \geq 0} \phi(\xmean, \sigma)$, if it is well-defined, is a function of $\xmean - x^*$. For the sake of simplicity we write the optimal step-size as a map $\sigma^*: \xmean - x^* \mapsto \sigma^*(\xmean - x^*)$. It is a homogeneous function of degree $1$, i.e., $\sigma^*(\alpha\cdot(\xmean - x^*)) = \alpha \sigma^*(\xmean - x^*)$ for any $\alpha > 0$.\end{proposition}

Note that the quadratic function is homogeneous of degree $2$, and the function $\xmean \mapsto \norm{\nabla f(\xmean)} = \norm{\Hess (\xmean - x^*)}$ is homogeneous of degree $1$ around $x^*$. 
\new{The latter}\del{It} is our candidate for the optimal step-size. We define the normalized step-size, the scale-invariant step-size, and the normalized quality gain for a quadratic function as follows.\del{{\small\begin{equation}
\ns = \frac{\sigma \cm \Tr(\Hess)}{\norm{\nabla f(\xmean)}}
\enspace.
\label{eq:nsquad}
\end{equation}}%
In other words, the step-size is determined by
{\small\begin{equation}
\sigma = \frac{\ns \norm{\nabla f(\xmean)} }{ \cm \Tr(\Hess) } \enspace. 
\label{eq:optsig}
\end{equation}}}{}

\del{\begin{definition}\del{[Scale-Invariant Step-Size on a Quadratic Function]}
\label{def:ns}
For a quadratic function \eqref{eq:quad}, the \emph{normalized step-size} $\ns > 0$ is defined as
{\small\begin{equation}
\ns = \frac{\sigma \cm}{\norm{\nabla f(\xmean)}}
\enspace.
\label{eq:nsquad}
\end{equation}}%
In other words, \del{given the normalized step-size $\ns$, }the step-size is given by
{\small\begin{equation}
\sigma = \frac{\ns \norm{\nabla f(\xmean)} }{ \cm } \enspace. 
\label{eq:optsig}
\end{equation}}%
We call it the \emph{scale-invariant step-size} for a quadratic function \eqref{eq:quad}.
\end{definition}}{}

\begin{definition}
\label{def:ns}
For a convex quadratic function \eqref{eq:quad}, the \emph{normalized step-size} $\ns$ and the \emph{scale-invariant step-size} \rev{$\sigma$} given $\ns$ are defined as $\ns = (\sigma \cm) / \norm{\nabla f(\xmean)}$ and $\sigma = (\ns / \cm) \norm{\nabla f(\xmean)}$.
\del{
{\small\begin{equation}
    \ns = \frac{\sigma \cm }{\norm{\nabla f(\xmean)}}
    \enspace,
    \quad
\sigma = \frac{\ns \norm{\nabla f(\xmean)} }{ \cm } \enspace. 
\label{eq:optsig}
\end{equation}}%
}%
\end{definition}

\del{The normalized quality gain of weighted recombination ES with scale-invariant step-size on a quadratic function is then defined as follows.}
\begin{definition}
\label{def:nqg}
Let $g:\R^N \to \R$ be the $\xmean$-dependent scaling factor of the normalized quality gain defined as $g(\xmean) = \norm{\nabla f(\xmean)}^2 / f(\xmean)$. 
\del{by
{\small\begin{equation}
g(\xmean) = \frac{\norm{\nabla f(\xmean)}^2}{f(\xmean)} \enspace.
\label{eq:g}
\end{equation}}%
}%
The \emph{normalized quality gain} for a quadratic function is defined as $\bar{\phi}(\xmean, \ns) = 
\phi(\xmean, \sigma = \ns\norm{\nabla f(\xmean)} / \cm ) / g(\xmean)$.
\del{{\small\begin{equation}
\bar{\phi}(\xmean, \ns) = 
\frac{\phi(\xmean, \sigma = \ns\norm{\nabla f(\xmean)} / \cm )}{g(\xmean)} \enspace.
\label{eq:nqgquad}
\end{equation}}%
}%
\end{definition}

\del{In other words, $\phi\big(\xmean, \sigma = \ns\norm{\nabla f(\xmean)} / \cm \big) = g(\xmean)\bar{\phi}(\xmean, \ns)$. }%
Note that the normalized step-size and the normalized quality gain defined above coincide with \eqref{eq:sph:ns}\del{ and \eqref{eq:sph:nqg}, respectively,} if $f(x) = \norm{x}^2 / (2N)$, where $\Hess = \eye / N$, $\nabla f(\xmean) = \xmean / N$ and $g(\xmean) = 2 / N$. Moreover, they are equivalent to Eq.~(4.104) in \cite{BeyerBOOK2001} introduced to analyze the $(1+\lambda)$-ES and the $(1, \lambda)$-ES. The same normalized step-size has been used for $(\mu/\mu_I, \lambda)$-ES \cite{beyer2014dynamics,Beyer2016ec}. See Section~4.3.1 of \cite{BeyerBOOK2001} for the motivation of these normalization.

\paragraph*{Non-Isotropic Gaussian Sampling} Throughout the paper, we assume that the  multivariate normal sampling distributions have an isotropic covariance matri\rev{x}. We can generalize all the following results to an arbitrary positive definite symmetric covariance matrix $\mathbf{C}$ by considering a linear transformation of the search space. Indeed, let $f: x \mapsto \frac12 (x - x^*)^\T \Hess (x - x^*)$, and consider the coordinate transformation $x \mapsto y = \mathbf{C}^{-\frac12}x$. In the latter coordinate system the function $f$ can be written as $f(x) = \bar{f}(y) = \frac12 (y - \mathbf{C}^{-\frac12} x^*)^\T (\mathbf{C}^{\frac12}\Hess\mathbf{C}^{\frac12}) (y - \mathbf{C}^{-\frac12} x^*)$. The multivariate normal distribution $\mathcal{N}(\xmean, \sigma^2\mathbf{C})$ is transformed into $\mathcal{N}(\mathbf{C}^{-\frac12} \xmean, \sigma^2 \eye)$ by the same transformation. Then, it is easy to prove that the quality gain on the function $f$ given the parameter $(\xmean, \sigma, \mathbf{C})$ is equivalent to the quality gain on the function $\bar{f}$ given $(\mathbf{C}^{-\frac12} \xmean, \sigma, \eye)$. The normalization factor $g(\xmean)$ of the quality gain and the normalized step-size are then rewritten as
\begin{align*}
g(\xmean) &= \frac{\norm{\mathbf{C}^{\frac12}\Hess(\xmean - x^*)}^2}{f(\xmean)\Tr(\mathbf{C}^{\frac12}\Hess\mathbf{C}^{\frac12})} \enspace,
&
\ns &= \frac{\sigma \cm \Tr(\mathbf{C}^{\frac12}\Hess\mathbf{C}^{\frac12})}{\norm{\mathbf{C}^{\frac12}\Hess(\xmean - x^*)}}
\enspace.
\end{align*}%

\subsection{Conditional Expectation of the Weight Function}

\rev{The quadratic objective \eqref{eq:quad} can be written as
\begin{align}
f(\xmean + \Delta) 
  = f(\xmean) + \nabla f(\xmean)^\T \Delta + \frac12 \Delta^\T \Hess \Delta 
  \enspace. \label{eq:fdiff}
\end{align}}%
The normalized quality gain on a convex quadratic function can be written as (using \eqref{eq:fdiff} with $\Delta = \xmean^{(t+1)} - \xmean^{(t)}$ and substituting \eqref{eq:mmz}) 
\del{{\small\begin{equation}
    \begin{split}
      \bar{\phi}(\xmean, \ns) 
      =& \frac{\E[f(\xmean^{(t)}) - f(\xmean^{(t+1)}) \mid \xmean^{(t)} = \xmean, \sigma^{(t)} = \ns\norm{\nabla f(\xmean)} / \cm \Tr(\Hess)]}{\norm{\nabla f(\xmean^{(t)})}^2 / \Tr(\Hess)} 
      \\
      =& 
      \frac{\Tr(\Hess)}{\norm{\nabla f(\xmean)}} \ee^\T \E\left[\xmean^{(t+1)} - \xmean^{(t)}\right] 
      + \frac{\Tr(\Hess)}{2\norm{\nabla f(\xmean)}^2} \E\left[(\xmean^{(t+1)} - \xmean^{(t)})^\T\Hess(\xmean^{(t+1)} - \xmean^{(t)}) \right]
      \\
      =& - \ns \sum_{i=1}^{\lambda} \E[W(i; (X_k)_{k=1}^{\lambda}) \ee^\T Z_i] 
      - \frac{\ns^2}{2} \sum_{i=1}^{\lambda} \sum_{j=1}^{\lambda} \E\left[W(i; (X_k)_{k=1}^{\lambda})W(j; (X_k)_{k=1}^{\lambda}) Z_i^{\T} \Hesst Z_j \right]\enspace,
      \\
      =& - \ns \sum_{i=1}^{\lambda} \E[\E_{i}[W(i; (X_k)_{k=1}^{\lambda})] \ee^\T Z_i]
      \\
      &- \frac{\ns^2}{2} \sum_{i=1}^{\lambda} \sum_{j=1}^{\lambda} \E\left[\E_{i,j}[W(i; (X_k)_{k=1}^{\lambda})W(j; (X_k)_{k=1}^{\lambda})] Z_i^{\T} \Hesst Z_j \right]\enspace,
    \end{split}
  \end{equation}}}{}%
\begin{align*}
    \bar{\phi}(\xmean, \ns) 
    = &- \ns \sum_{i=1}^{\lambda} \E\big[\E_{i}[W(i; (X_k))] \ee^\T Z_i\big] \\
    &- \frac{\ns^2}{2} \sum_{i=1}^{\lambda} \sum_{j=1}^{\lambda} \E\big[\E_{i,j}[W(i; (X_k))W(j; (X_k))] Z_i^{\T} \Hesst Z_j \big]\enspace,
  \end{align*}%
where $X_k = \xmean^{(t)} + \sigma^{(t)} Z_k$, and $\E_{i}$ and $\E_{i,j}$ are the conditional expectations given $X_i$ and $(X_i, X_j)$, respectively.

\newcommand{\pb}{P_b}
\newcommand{\pt}{P_t}
The following lemma provides the expression of the conditional expectation of the weight function, which allows us to derive the bound for the difference between $\bar{\phi}$ and $\asynqg$.
  In the following, let
  \begin{align*}
    \pb(k; n, p) &= \textstyle \binom{n}{k} p^{k}(1 - p)^{n - k} \\
    \pt(k, l; n, p, q) &= \textstyle \rev{\binom{n}{l+k}\binom{l+k}{k}} p^{k}q^{l} (1 - (p + q))^{n-(k+l)}
  \end{align*}
  denote the probability mass functions of the binomial and trinomial distributions, respectively, where $0 \leq k \leq n$, $0 \leq l \leq n$, $k + l \leq n$, $0\leq p \leq 1$, $0 \leq q \leq 1$ and $p + q \leq 1$.
The proof of the lemma is provided in \ref{apdx:lem:0}.

\begin{lemma}\label{lem:0}
  Let $X \sim \mathcal{N}(\xmean, \sigma^2 \eye)$ and $(X_i)_{i=1}^{\lambda}$ be $\lambda$ i.i.d.\ copies of $X$. Let $F_{f}(t) = \Pr[f(X) < t]$ be the c.d.f.~of the function value $f(X)$. 
  Then, we have for any $i, j \in \llbracket 1, \lambda \rrbracket$, $i \neq j$,
  \begin{align*}
    \E_{i}[W(i; (X_k))] &= u_1(F_{f}(f(X_i))) \enspace,\\
    \E_{i}[W(i; (X_k))^2] &= u_2(F_{f}(f(X_i))) \enspace,\\
    \E_{i,j}[W(i; (X_k))W(j; (X_k)) ] &= u_3(F_{f}(f(X_i)), F_{f}(f(X_j))) \enspace,
  \end{align*}
  where    
\begin{align}
  u_1(p) &= \textstyle \sum_{k = 1}^{\lambda} w_k
           \pb(k-1; \lambda-1, p) \enspace,
                  \label{eq:u1}
         \\
  u_2(p) &= \textstyle  \sum_{k = 1}^{\lambda} w_k^2
           \pb(k-1; \lambda-1, p) \enspace,
                  \label{eq:u2}
         \\
  u_3(p, q) &= \textstyle \sum_{k = 1}^{\lambda-1}\sum_{l = k + 1}^{\lambda} w_k w_l
              \pt(k-1, l-k-1; \lambda-2, \min(p, q), \abs{q - p})
                     \enspace.
                     \label{eq:u3}
\end{align}%
\end{lemma}

Thanks to Lemma~\ref{lem:0} and the fact that $(X_k)_{k=1}^{\lambda}$ are i.i.d., we can further rewrite the normalized quality gain \del{\nnew{by taking the iterated expectation $\E = \E_{X_i}\E_{X_k,k\neq i}$ }}as
\del{{\small\begin{align}
\bar{\phi}(\xmean, \ns)
=& - \ns \lambda  \E\big[u_1(F_{f}(f(X))) \ee^\T Z\big] - \frac{\ns^2\lambda}{2} \E\big[u_2(F_{f}(f(X))) Z^{\T} \Hesst Z \big] \notag\\
& - \frac{\ns^2(\lambda - 1)\lambda}{2} \E\big[u_3(F_{f}(f(X)), F_{f}(f(\tilde{X}))) Z^{\T} \Hesst \tilde{Z} \big] \notag\\
=& - \ns \lambda  \E\big[u_1(F_{f}(f(X))) \ee^\T Z\big] - \frac{\ns^2\lambda}{2} \E[u_2(F_{f}(f(X)))]  - \frac{\ns^2\lambda}{2} \E\big[u_2(F_{f}(f(X))) \big(Z^{\T} \Hesst Z - 1\big) \big] \notag\\
&- \frac{\ns^2(\lambda - 1)\lambda}{2} \E\big[u_3(F_{f}(f(X)), F_{f}(f(\tilde X))) Z^{\T} \Hesst \tilde Z \big]\enspace.
  \label{eq:main}
\end{align}}}{}%
\begin{multline}
\bar{\phi}(\xmean, \ns)
= - \ns \lambda  \E\big[u_1(F_{f}(f(X))) \ee^\T Z\big] 
- \frac{\ns^2\lambda}{2} \E\big[u_2(F_{f}(f(X))) \big(Z^{\T} \Hesst Z - 1\big) \big] \\
- \frac{\ns^2\lambda}{2} \E[u_2(F_{f}(f(X)))] 
- \frac{\ns^2(\lambda - 1)\lambda}{2} \E\big[u_3(F_{f}(f(X)), F_{f}(f(\tilde X))) Z^{\T} \Hesst \tilde Z \big]\enspace.
  \label{eq:main}
\end{multline}%
Here $Z$ and $\tilde{Z}$ are independent and $\mathcal{N}(\bm{0}, \eye)$-distributed, and $X = \xmean + \sigma Z$ and $\tilde{X} = \xmean + \sigma \tilde{Z}$, where $\sigma = \ns \norm{\nabla f(\xmean)} / \cm$ is the scale-invariant step-size. Note that $X$ and $\tilde{X}$ are independent and $\NN(\xmean, \sigma^2 \eye)$-distributed.

The following Lemma shows the Lipschitz continuit\rev{y} of $u_1$, $u_2$, and $u_3$. The proof is provided in \ref{apdx:lem:u1lip}.
\begin{lemma}
\label{lem:u1lip}
The functions $u_1$, $u_2$, and $u_3$ are $\ell_1$-Lipschitz continuous, i.e., $\abs{u_1(p_1) - u_1(p_{\rev{2}})} \leq L_1 \abs{p_1 - p_2}$, $\abs{u_2(p_1) - u_2(p_{\rev{2}})} \leq L_2 \abs{p_1 - p_2}$, and $\abs{u_3(p_1, q_1) - u_3(p_{\rev{2}}, q_{\rev{2}})} \leq L_3 (\abs{p_1 - p_2} + \abs{q_1 - q_2})$, with the Lipschitz constants 
\begin{align*}
  L_1 &= \textstyle {\displaystyle \sup_{0 < p < 1}} \abs*{(\lambda - 1) \sum_{k = 1}^{\lambda-1} (w_{k+1} - w_{k}) \pb(k-1; \lambda-2, p)} \enspace,
  \\
  L_2 &= \textstyle {\displaystyle \sup_{0 < p < 1}}  \abs*{(\lambda - 1) \sum_{k = 1}^{\lambda-1} (w_{k+1}^2 - w_{k}^2) \pb(k-1; \lambda-2, p)}\enspace,
  \\
  L_3 &= \textstyle \max\Bigg[{\displaystyle \sup_{0 < p < q < 1}} \abs*{\sum_{k=1}^{\lambda-2}\sum_{l=k+2}^{\lambda} w_l (w_{k+1} - w_{k})  
        \pt(k-1, l-k-2; \lambda-3, p, q-p)
        } \enspace,
\\
	&\quad
          \textstyle{\displaystyle \sup_{0 < p < q < 1}} \abs*{\sum_{k=1}^{\lambda-2}\sum_{l=k+2}^{\lambda} w_k (w_{l} - w_{l-1})
          \pt(k-1, l-k-2; \lambda-3, p, q-p)
          }\Bigg] (\lambda-2) \enspace.
\end{align*}%
\end{lemma}
Upper bounds for the above Lipschitz constants are discussed in \ref{apdx:Lbound:deriv}.

\subsection{Theorem: Normalized Quality Gain on Convex Quadratic Functions}

The following main theorem provides the error bound between $\bar{\phi}$ and $\asynqg$.

\begin{theorem}\label{thm:main}
\rev{Consider Algorithm~\ref{alg:es} and let $f$ be a convex quadratic objective function \eqref{eq:quad}}. Let the normalized step-size $\ns$ and the normalized quality gain $\bar{\phi}$ defined in Definition~\ref{def:ns}\del{as \eqref{eq:optsig}} and Definition~\ref{def:nqg}\del{\eqref{eq:nqgquad}}, respectively. Let $\ee_{\xmean} = \nabla f(\xmean) / \norm{\nabla f(\xmean)}$ and $\alpha = \min\big( 1, (\ns / \cm) \Tr(\Hesst^2)^{1/2} \big)$. Define
\begin{equation}
G(\alpha) = \min\left[1,\ \alpha \left(2 + \frac{2^\frac12(\ln(1/\alpha))^\frac12}{\pi^\frac12} + \frac{\eig_1(\Hesst)\ln(1/\alpha)}{(2\pi)^\frac12 \Tr(\Hesst^2)^\frac12} \right)  \right] \label{eq:galpha}
\end{equation}%
and
\begin{multline}
\asynqg(\ns, (w_k), \ee_{\xmean}, \Hesst) 
= - \ns \sum_{i=1}^{\lambda} w_i\E[\mathcal{N}_{i:\lambda}] - \frac{\ns^2}{2}\sum_{i=1}^\lambda w_i^2\left(1 - \ee_{\xmean}^\T\Hesst\ee_{\xmean}\right) \\ - \frac{\ns^2}{2} \ee_{\xmean}^\T\Hesst\ee_{\xmean}\sum_{i=1}^\lambda\sum_{j=1}^\lambda w_i w_j \E[\NN_{i:\lambda}\NN_{j:\lambda}]
\enspace,
\label{eq:asynqg}
\end{multline}%
and let $L_1$, $L_2$, $L_3$ be the Lipschitz constants of $u_1$, $u_2$ and $u_3$ defined in Lemma~\ref{lem:0}, respectively. Then, 
\begin{multline}
\sup_{\xmean \in \R^N\setminus\{\bm{0}\}}\abs*{\bar{\phi}(\xmean, \ns) - \asynqg(\ns, (w_k), \ee_{\xmean}, \Hesst) }
\leq \ns \lambda L_1 \big((2 / \pi)^\frac12 G(\alpha) + (4\pi)^{-\frac12} \alpha \big)
	\\
	+ \ns\cm \lambda L_2 \big( 2^{-\frac12} G(\alpha) + (8\pi)^{-\frac12}\alpha \big) \alpha
	+ \ns\cm \lambda(\lambda - 1)L_3 \big((2 / \pi)^\frac12 G(\alpha) + (2 \pi^2)^{-\frac12} \alpha \big) \alpha
	 \enspace.
         \label{eq:RHS}
       \end{multline}%
\end{theorem}

The above theorem claims that if the right-hand side (RHS) of \eqref{eq:RHS} is sufficiently small, the normalized quality gain $\bar{\phi}$ is approximated by the asymptotic normalized quality gain $\asynqg$ defined in \eqref{eq:asynqg}. Compared to $\bar{\phi}_{\infty}$ in \eqref{eq:nqlim} derived for the infinite dimensional sphere function, $\asynqg$ is different even when $\Hess \propto \eye$. We investigate the properties of $\asynqg$ in Section~\ref{sec:asynqg}. The situations where the RHS of \eqref{eq:RHS} is sufficiently small
are discussed in Section~\ref{sec:cm} and Section~\ref{sec:infn}. \rev{We remark that Theorem~3.4 in \cite{Akimoto2017foga} provides a bound for the difference between $\bar{\phi}$ and $\bar{\phi}_\infty$, instead of the difference between $\bar{\phi}$ and $\asynqg$. Introducing $\asynqg$ allows us to consider a finite dimensional case and to derive a tighter bound.}

\subsection{Outline of the Proof of the Main Theorem}
\label{sec:proof}

\providecommand{\Ze}{Z_{\ee}}
\providecommand{\Zee}{Z_{\bot}}
\providecommand{\Zet}{\tilde{Z}_{\ee}}
\providecommand{\Zeet}{\tilde{Z}_{\bot}}

In the following of the section and in \ref{apdx:proofs}, let $\Ze = \ee^\T Z$, $\Zee = Z - \Ze \ee$, and $X = \xmean + \sigma Z$ for $Z \sim \mathcal{N}(\bm{0}, \eye)$. Then, $\Ze \sim \mathcal{N}(0, 1)$ and $\Zee \sim \mathcal{N}(\bm{0}, \eye-\ee\ee^\T)$ and they are independent. Define
\begin{align*}
    H_N &= \frac{f(\xmean + \sigma Z) - \E[f(\xmean + \sigma Z)] }{ \sigma \norm{\nabla f(\xmean)}}&&\text{and}
    &h(Z) 
	&= \frac12 \frac{\ns}{\cm}(Z^\T \Hesst Z - 1) \enspace, 
\end{align*}%
where $\E[f(\xmean + \sigma Z)] = f(\xmean) + \sigma^2 / 2$. It is easy to see that $H_N = \Ze + h(Z)$. Let $F_{f}$ and $F_{N}$ be the c.d.f.~induced by $f(X)$ and $H_N$, respectively. Then, $F_{f}(f(X)) = F_{N}(H_N)$. Let $\tilde{Z}$ be the i.i.d.\ copy of $Z$ and define $\Zet$, $\Zeet$, $\tilde{X}$, and $\tilde{H}_N$ analogously. 

The first lemma allows us to approximate $F_{N}$ (hence $F_{f}$) with the c.d.f.~$\Phi$ of the standard normal distribution. \rev{The proof is based on the Lipschitz continuity of $\Phi$ and the tail bound of $h(Z)$ proved in Lemma~1 of \cite{laurent2000}.} The detail is provided in \ref{apdx:lem:weak}. 
\begin{lemma}\label{lem:weak}
Let $\alpha$ and $G(\alpha)$ be defined in Theorem~\ref{thm:main}. 
Then, $\sup_{t \in \R} \abs{F_{N}(t) - \Phi(t)} \leq G(\alpha)$.
\end{lemma}

The following three lemmas are used to bound each term on the RHS of \eqref{eq:main}. \rev{The proofs are straight-forward from the Lipschitz continuity of $u_1$, $u_2$ and $u_3$ and Lemma~\ref{lem:weak}. The detailed proofs are found in \ref{apdx:lem:u1}, \ref{apdx:lem:u2}, and \ref{apdx:lem:u3}, respectively.}
\begin{lemma}\label{lem:u1}
  Let $L_1$, $\alpha$, and $G(\alpha)$ be the quantities appeared in Lemma~\ref{lem:u1lip} and Theorem~\ref{thm:main}. Then,
  $$\abs{\E[u_1(F_{f}(f(X))) \Ze] - \E[u_1(\Phi(\Ze)) \Ze]} \leq L_1 \big((2 / \pi)^\frac12 G(\alpha) + (4\pi)^{-\frac12} \alpha \big)\enspace.$$
\end{lemma}

\begin{lemma}\label{lem:u2}
  Let $L_2$, $\alpha$, and $G(\alpha)$ be the quantities appeared in Lemma~\ref{lem:u1lip} and Theorem~\ref{thm:main}. Then,
  \begin{multline*}
    \abs{\E[u_2(F_{f}(f(X))) (Z^\T \Hesst Z - 1)] - \E[u_2(\Phi(\Ze)) (Z^\T \Hesst Z - 1)]} \\
    \leq L_2\big( 2^\frac12 G(\alpha) + (2\pi)^{-\frac12}\alpha \big) \Tr(\Hesst^2)^\frac12\enspace.
  \end{multline*}
\end{lemma}

\begin{lemma}\label{lem:u3}
  Let $L_3$, $\alpha$, and $G(\alpha)$ be the quantities appeared in Lemma~\ref{lem:u1lip} and Theorem~\ref{thm:main}. Then,
  \begin{multline*}
    \abs{\E[u_3(F_{f}(f(X)), F_{f}(f(\tilde{X}))) Z^\T\Hesst\tilde{Z}] - \E[u_3(\Phi(\Ze), \Phi(\Zet)) Z^\T\Hesst\tilde{Z}]} \\
    \leq L_3 \big((8 / \pi)^\frac12 G(\alpha) + (2^\frac12 / \pi) \alpha \big) \Tr(\Hesst^2)^\frac12 \enspace.
  \end{multline*}
\end{lemma}
\rev{The asymptotic normalized quality gain $\asynqg$ in \eqref{eq:asynqg} is obtained by replacing the c.d.f.~$F_{f}$ of $f(X)$ in \eqref{eq:main} with the c.d.f.~$\Phi$ of the standard normal distribution. The above lemmas are used to bound the difference between $\bar\phi$ and $\asynqg$.} The following lemma provides the explicit form of each term of \eqref{eq:asynqg}. The proof is straight-forward in light of Lemma~\ref{lem:0} The detail can be found in \ref{apdx:lem:os}. 
\begin{lemma}\label{lem:os}
The functions $u_1$, $u_2$ and $u_3$ defined in Lemma~\ref{lem:0} satisfy the following properties:
\begin{align}
\lambda\E[u_1(\Phi(\Ze)) \Ze] &\textstyle= \sum_{i=1}^{\lambda} w_i\E[\mathcal{N}_{i:\lambda}] \enspace,\label{eq:os:1}\\
\lambda \E[ u_2(\Phi(\Ze)) ] &\textstyle= \sum_{i=1}^\lambda w_i^2  \enspace,\label{eq:os:2}\\
\lambda \E[ u_2(\Phi(\Ze)) (Z^\T \Hesst Z  - 1) ] &\textstyle= \ee^\T\Hesst\ee \sum_{i=1}^\lambda w_i^2 (\E[\NN_{i:\lambda}^2] - 1) \enspace,\label{eq:os:3}\\
\lambda (\lambda - 1)\E[ u_3(\Phi(\Ze), \Phi(\Zet)) Z^\T \Hesst \tilde{Z} ] &\textstyle= 2\ee^\T\Hesst\ee\sum_{k = 1}^{\lambda-1}\sum_{l = k + 1}^{\lambda} w_k w_l \E[\NN_{k:\lambda}\NN_{l:\lambda}] \enspace.\label{eq:os:4}
\end{align}%
\end{lemma}

Now we finalize the proof of the main theorem. Using Lemma~\ref{lem:os}, we can rewrite \eqref{eq:asynqg} as
\begin{multline*}
    \asynqg(\ns, (w_k)_{k=1}^\lambda, \ee_{\xmean}, \Hesst)
    = - \ns \lambda\E[u_1(\Phi(\Ze)) \Ze] - \frac{\ns^2}{2}\lambda \E[ u_2(\Phi(\Ze)) ]
    \\
    - \frac{\ns^2}{2}\lambda \E[ u_2(\Phi(\Ze)) (Z^\T \Hesst Z  - 1) ] 
    -  \frac{\ns^2}{2}\lambda (\lambda - 1)\E[ u_3(\Phi(\Ze), \Phi(\Zet)) Z^\T \Hesst \tilde{Z} ] \enspace.
  \end{multline*}%
From the equation \eqref{eq:main} and the above expression of $\asynqg$, we have
  \begin{align*}
    \MoveEqLeft[3]\bar{\phi}(\xmean, \ns) - \asynqg(\ns, (w_k)_{k=1}^\lambda, \ee_{\xmean}, \Hesst)\\
    =&\textstyle - \ns \lambda  \E\left[(u_1(F_{f}(f(X))) - u_1(\Phi(\Ze))) \Ze\right]\\
     &\textstyle- \frac{\ns^2\lambda}{2} \E[(u_2(F_{f}(f(X))) - u_2(\Phi(\Ze)))]\\
     &\textstyle- \frac{\ns^2\lambda}{2} \E\left[(u_2(F_{f}(f(X))) - u_2(\Phi(\Ze)))\left(Z^{\T} \Hesst Z - 1\right) \right]\\
     &\textstyle- \frac{\ns^2(\lambda - 1)\lambda}{2} \E\big[(u_3(F_{f}(f(X)), F_{f}(f(\tilde X))) - u_3(\Phi(\Ze), \Phi(\Zet))) Z^{\T} \Hesst \tilde Z \big] \enspace.
  \end{align*}%
From the well-known fact (e.g., Theorem~2.1 of \cite{Devroye1986}) that for a random variable $X$ with a continuous c.d.f.~$F_x$ the random variable $F_x(X)$ is uniformly distributed on $[0, 1]$, we can prove both $F_{f}(f(X))$ and $\Phi(\Ze)$ are uniformly distributed on $[0, 1]$. Therefore, we have $\E[u_2(F_{f}(f(X)))] =  \E[u_2(\Phi(\Ze))] = \E[u_2(\mathcal{U}[0, 1])]$, and the second term on the RHS of the above equality is zero. Applying the triangular inequality and Lemma~\ref{lem:u1}, Lemma~\ref{lem:u2}, and Lemma~\ref{lem:u3}, we obtain \eqref{eq:RHS}. 
It completes the proof of Theorem~\ref{thm:main}. 

\section{Consequences}\label{sec:cons}
\providecommand{\bbmw}{\bar{\bm{w}}}

Theorem~\ref{thm:main} tells that if the RHS of \eqref{eq:RHS} is sufficiently small, the normalized quality gain $\bar{\phi}(\bm{m}, \ns)$ is well approximated by $\asynqg(\ns, (w_k)_{k=1}^\lambda, \ee_{\xmean}, \Hesst)$ defined in \eqref{eq:asynqg}. First we investigate the parameter values that are optimal for $\asynqg$. Then, we consider the situations when the RHS of \eqref{eq:RHS} is sufficiently small.

Let $\eone_{(\lambda)}$ be the $\lambda$ dimensional column vector whose $i$-th component is $\E[\NN_{i:\lambda}]$ and $\etwo_{(\lambda)}$ be the $\lambda$ dimensional symmetric matrix whose $(i, j)$-th elements are $\E[\NN_{i:\lambda} \NN_{j:\lambda}]$. Let $\bm{w}$ and $\bbmw$ be the $\lambda$ dimensional column vector whose $i$-th element is $w_i$ and $\ns w_i$, respectively. 
Now \eqref{eq:asynqg} can be written as
\begin{equation}
\asynqg(\bbmw, \ee, \Hesst) 
= - \bbmw^\T \eone_{(\lambda)} - \frac{1}{2}\left(1 - \ee^\T\Hesst\ee\right)\bbmw^\T\bbmw - \frac{1}{2} (\ee^\T\Hesst\ee) \bbmw^\T \etwo_{(\lambda)} \bbmw
\enspace.
\label{eq:asynqg-mat}
\end{equation}%
In the following we use the following asymptotically true approximation for a sufficiently large $\lambda$ (see \eqref{eq:etwoapprox} in \ref{apdx:nos})
\begin{equation}
  \frac{\bbmw^\T \etwo_{(\lambda)}\bbmw}{\lambda \norm{\bbmw}^2}
  \approx \frac{ (\bbmw^\T\eone_{(\lambda)})^2 }{ \norm{\bbmw}^2 \norm{\eone_{(\lambda)}}^2 }
  \approx \frac{ (\bbmw^\T\eone_{(\lambda)})^2 }{ \lambda \norm{\bbmw}^2 } \enspace.
  \label{eq:etwocond}
\end{equation}
By \emph{``for a sufficiently large $\lambda$"}, we mean for a $\lambda$ large enough to approximate \rev{the left hand side (LHS) of \eqref{eq:etwocond} by the right-most side (RMS)}. For a sufficiently large $\lambda$, \eqref{eq:asynqg-mat} is approximated by
\begin{equation}
\asynqg(\bbmw, \ee, \Hesst) 
\approx - \bbmw^\T \eone_{(\lambda)} - \frac{1}{2}\left(1 - \ee^\T\Hesst\ee\right)\bbmw^\T\bbmw - \frac{1}{2} (\ee^\T\Hesst\ee) (\bbmw^\T \eone_{(\lambda)})^2
\enspace.
\label{eq:asynqg-mat2}
\end{equation}%

\subsection{Asymptotic Normalized Quality Gain and Optimal Parameters}
\label{sec:asynqg}

As we mentioned in the previous section, $\asynqg$ in \eqref{eq:asynqg} is different from the normalized quality gain limit $\bar{\phi}_\infty$ in \eqref{eq:nqlim}. Consider the sphere function $\Hesst = \eye / N$; then, since $\ee^\T \Hesst \ee = 1/N$ for any $\ee$ with $\norm{\ee} = 1$, we have
\begin{equation*}
\asynqg(\ns, (w_k), \ee, \Hesst) 
      \\= \bar{\phi}_{\infty}(\ns, (w_k)_{k=1}^\lambda) + \frac{\ns^2}{2N}\sum_{i=1}^\lambda w_i^2 - \frac{\ns^2}{2N} \sum_{i=1}^\lambda\sum_{j=1}^\lambda w_i w_j \E[\NN_{i:\lambda}\NN_{j:\lambda}]
\enspace.
\end{equation*}%
Note that the second and the third terms on the \rev{RHS} are proportional to $1/N$. By taking the limit for $N$ to infinity, we have $\asynqg = \bar{\phi}_{\infty}$. Therefore, the second and third terms describe how the finite dimensional cases are different from the infinite dimensional case. \rev{In particular, the last term prevents the quality gain from scaling up proportionally to $\lambda$ when $\lambda \not \ll N$.} 

\paragraph{Optimal Recombination Weights}
The recombination weights optimal for $\asynqg$ are provided in the following proposition.

\begin{proposition}\label{prop:optw-gen}
The asymptotic normalized quality gain $\asynqg$ \eqref{eq:asynqg} is optimized when $\bbmw$ is the solution to the following linear system of equation\rev{s}
\begin{equation}
  \label{eq:optw-gen}
(\eye + \ee^\T\Hesst\ee(\etwo_{(\lambda)} - \eye)) \bbmw = - \eone_{(\lambda)} \enspace,
\end{equation}%
where $\ns$ and $w_i$ are uniquely determined using the condition $\sum_{i=1}^\lambda \abs{w_i} = 1$. Then the optimal value of $\asynqg$ is $- \frac12 \eone_{(\lambda)}^\T\bbmw^*$ where $\bbmw^*$ \new{is the solution} to the linear system \eqref{eq:optw-gen}. 
\end{proposition}
\begin{proof}
  We obtain \eqref{eq:optw-gen} by taking the derivative of \eqref{eq:asynqg-mat} with respect to $\bbmw$, and \rev{requiring} $\partial \asynqg(\bbmw, \ee, \Hesst) / \partial [\bbmw]_i = 0$. This ends the proof.
\end{proof}

First, consider the limit for $N \to \infty$ while $\lambda$ is fixed. As long as the largest eigenvalue $d_1(\Hess)$ of $\Hesst$ converges to zero as $N \to \infty$, i.e., $\lim_{N\to \infty} d_1(\Hess) = 0$, we have $\ee^\T \Hesst \ee \to 0$ as $N \to 0$. Then, \eqref{eq:optw-gen} reads $\bbmw = - \eone_{(\lambda)}$. Therefore, we have the same optimal recombination weights as the ones derived for the infinite dimensional sphere function. 

Next, consider a finite dimensional case. If $\lambda$ is sufficiently large, the optimality condition \eqref{eq:optw-gen} is approximated by 
\begin{equation*}
\left(\frac{1}{\lambda} \left(1 - \frac{1}{\ee^\T \Hesst \ee}\right)\eye + \frac{1}{\ee^\T \Hesst \ee}\frac{\eone_{(\lambda)}\eone_{(\lambda)}^\T}{ \norm{\eone_{(\lambda)}}^2 }\right)\bbmw = -\frac{\eone_{(\lambda)}}{\lambda} \enspace.
\end{equation*}%
The solution to the above approximated condition is given by $\bbmw \propto - \eone_{(\lambda)}$ independently of $\Hesst$ and $\ee$. It means, for a sufficiently large $\lambda$, the optimal recombination weights are approximated by the weights \eqref{eq:optw} optimal for the infinite dimensional sphere function.

\paragraph{Optimal Normalized Step-Size}
The optimal $\ns$ under a given $(w_k)_{k=1}^\lambda$ is provided in the following proposition.
\begin{proposition}
Given $\bm{w} = (w_1, \dots, w_\lambda)$, the asymptotic normalized quality gain \eqref{eq:asynqg} is maximized when the normalized step-size $\ns$ is
\begin{equation}
    \ns^* = \frac{- \sum_{i=1}^{\lambda} w_i\E[\mathcal{N}_{i:\lambda}]}{\sum_{i=1}^\lambda w_i^2\left(1 - \ee_{\xmean}^\T\Hesst\ee_{\xmean}\right) + \ee_{\xmean}^\T\Hesst\ee_{\xmean}\sum_{i=1}^\lambda\sum_{j=1}^\lambda w_i w_j \E[\NN_{i:\lambda}\NN_{j:\lambda}]}
    \enspace,
\label{eq:optns-gen}
\end{equation}%
then $\asynqg(\ns^*, (w_k), \ee, \Hesst) = \frac{\ns^*}{2} \big(- \sum_{i=1}^{\lambda} w_i\E[\mathcal{N}_{i:\lambda}]\big)$.
\end{proposition}
\begin{proof}
It is a straight-forward consequence from differentiating \eqref{eq:asynqg} with respect to $\ns$ and solving $\partial \asynqg / \partial \ns = 0$.
\end{proof}

For a sufficiently large $\lambda$ (see \eqref{eq:etwocond}), one can rewrite and approximate \eqref{eq:optns-gen} as
\begin{align}
    \ns^*
    &= \frac{- \bm{w}^\T \eone_{(\lambda)}}{\left(1 - \ee_{\xmean}^\T\Hesst\ee_{\xmean}\right) \norm{\bm{w}}^2 + \ee_{\xmean}^\T\Hesst\ee_{\xmean} \bm{w}^\T\etwo_{(\lambda)}\bm{w}} \notag\\
  &\approx \frac{(\ee_\xmean^\T \Hesst \ee_\xmean)^{-1} \muw (- \bm{w}^\T \eone_{(\lambda)})}{(\ee_\xmean^\T \Hesst \ee_\xmean)^{-1} - 1  + \muw (- \bm{w}^\T \eone_{(\lambda)})^2 }  \enspace.
  \label{eq:optsig-approx}
\end{align}%
%
Note again that $-\bm{w}^\T \eone_{(\lambda)} = -\sum_{i=1}^{\lambda} w_i \E[\mathcal{N}_{i:\lambda}] \in O(1)$  for the optimal weights, CMA-type non-negative weights, and truncation weights with fixed truncation ratio. To provide a better insight, consider the case of the sphere function ($\Hess = \eye/N$). Then, the RMS of \eqref{eq:optsig-approx} reads
\begin{equation}
    \ns^*
    \approx \frac{N \muw (- \bm{w}^\T \eone_{(\lambda)})}{N - 1  + \muw (- \bm{w}^\T \eone_{(\lambda)})^2 }
  \enspace.
  \label{eq:optsig-largelambda}
\end{equation}%
Then, we find the following: (\new{i}\del{1}) if $N \gg \muw$, then $\ns^* \approx \muw (-\bm{w}^\T\eone_{(\lambda)})$ and $\asynqg \approx \muw(-\bm{w}^\T\eone_{(\lambda)})^2 / 2$; (\new{ii}\del{2}) if $\muw \gg N$, then $\ns^* \approx N / (-\bm{w}^\T\eone_{(\lambda)})$ and $\asynqg \approx N/ 2$.
Figure~\ref{fig:cmlargeopts} visualizes the optimal normalized step-size \new{\eqref{eq:optns-gen}} for various $\bm{w}$ on the sphere function. The optimal normalized step-size \eqref{eq:optns-gen} scales linearly for $\lambda \leq N$ and it tends to level out for $\lambda > N$. 

\begin{figure}[!t]
  \centering
  \begin{subfigure}{0.5\linewidth}
      \centering
    \includegraphics[width=0.95\hsize]{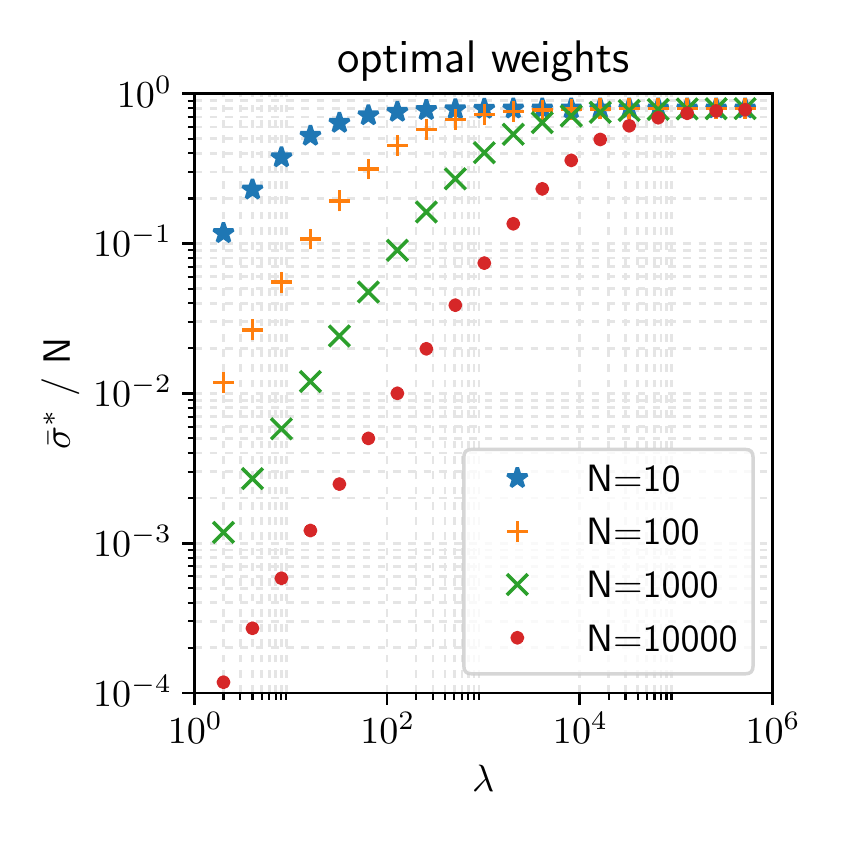}
  \end{subfigure}%
  \begin{subfigure}{0.5\linewidth}
      \centering
    \includegraphics[width=0.95\hsize]{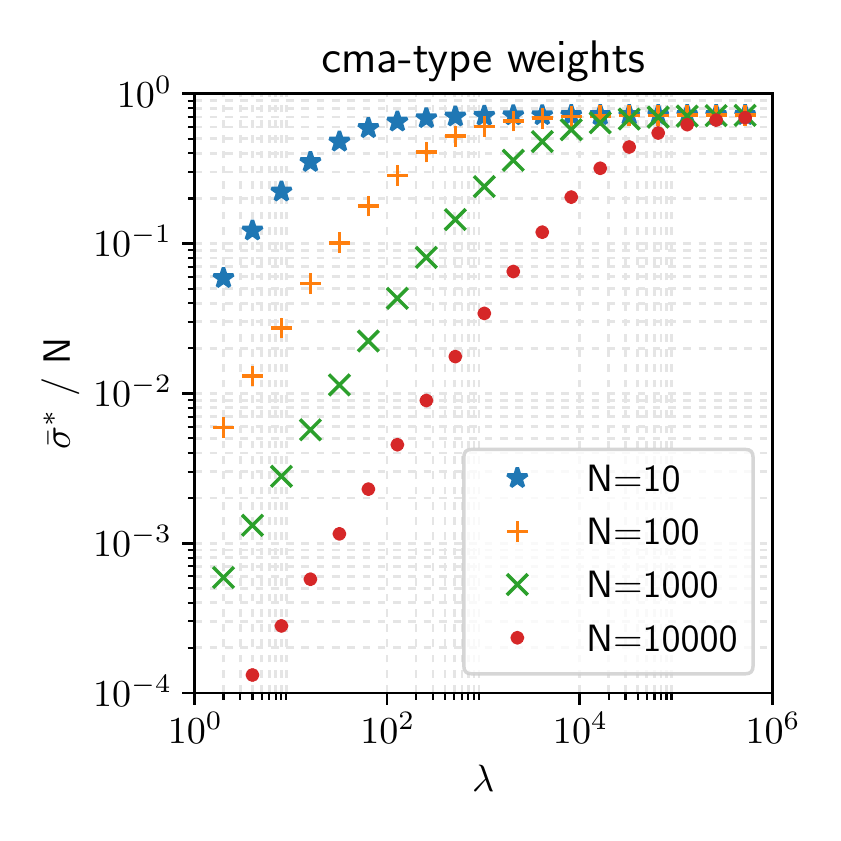}
  \end{subfigure}%
  \\
  \begin{subfigure}{0.5\linewidth}
      \centering
    \includegraphics[width=0.95\hsize]{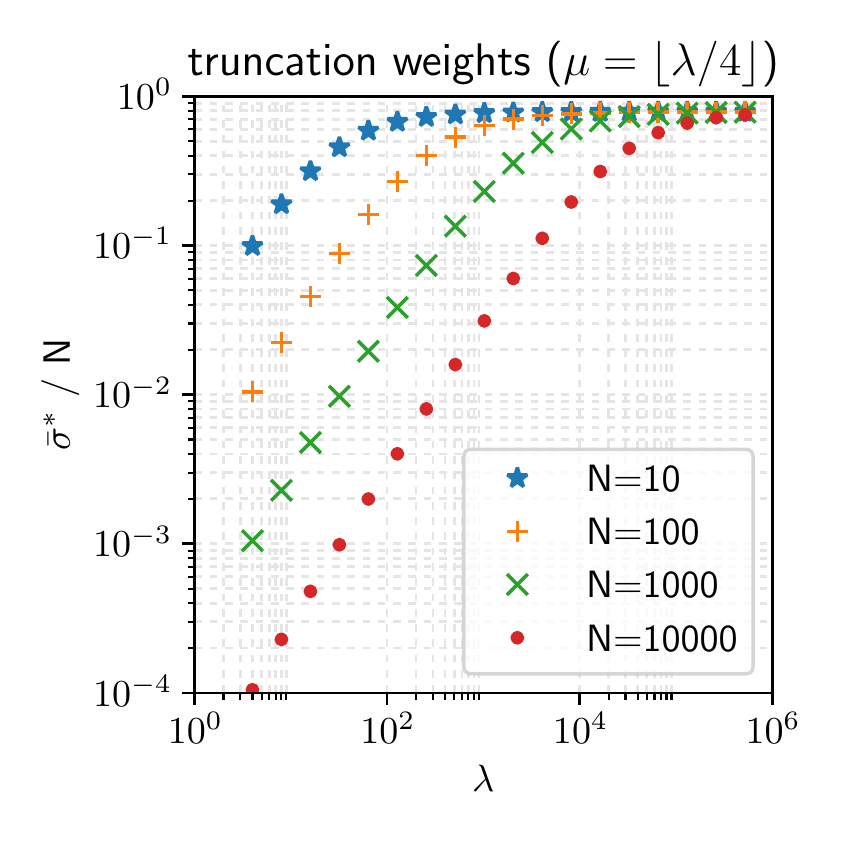}
  \end{subfigure}%
  \begin{subfigure}{0.5\linewidth}
      \centering
    \includegraphics[width=0.95\hsize]{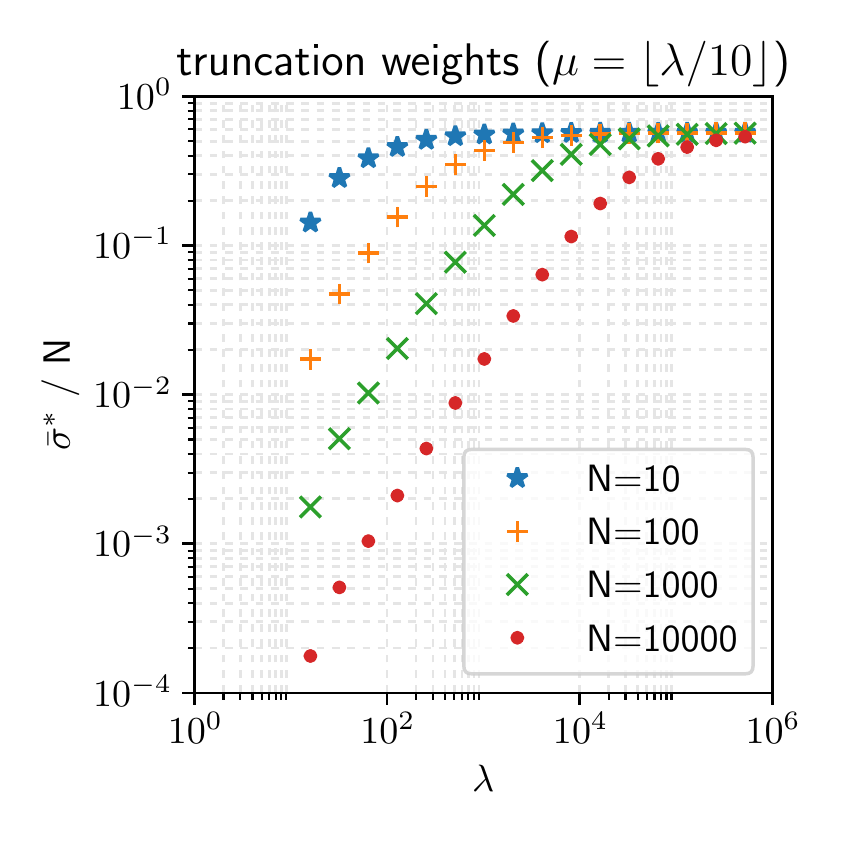}
  \end{subfigure}%
  \caption{Optimal normalized step-size on $N = 10$, $100$, $1000$, $10000$ dimensional sphere function for different weight schemes and different population size $\lambda$.}  
  \label{fig:cmlargeopts}
\end{figure}

\paragraph{Geometric Interpretation of the Optimal Situation}

On the infinite dimensional sphere function, we know that the optimal step-size puts the algorithm in the situation where $f(\xmean)$ improves twice as much by $\xmean$ moving towards the optimum as it deteriorates by $\xmean$ moving randomly in the subspace orthogonal to the gradient direction \cite{Rechenberg1994}. On a finite dimensional convex quadratic function, we find the analogous result. From \eqref{eq:main} and lemmas in Section~\ref{sec:proof}, the first term of the asymptotic normalized quality gain \eqref{eq:asynqg}, i.e.\ $- \ns \sum_{i=1}^{\lambda} w_i \E[\NN_{i:\lambda}]$, is due to the movement of $\xmean$ \rev{in}\del{ the} negative gradient direction, and the second and third terms are due to the random walk in the orthogonal subspaces\footnote{More precisely, the second and the third terms come from the quadratic term in \eqref{eq:fdiff} that contains the information in the gradient direction as well. However, the above statement is true in the limit $N \to \infty$ as long as $\Tr(\Hesst^2) \to 0$.}. The asymptotic normalized quality gain is maximized when the normalized step-size is set such that the first term is \rev{twice as large as} the absolute value of the sum of the second and the third terms. That is, the amount of the decrease of $f(\xmean)$ by $\xmean$ moving into the negative gradient direction is twice greater than the increase of $f(\xmean)$ by $\xmean$ moving in its orthogonal subspace.

\subsection{Infinitesimal Step-Size Case}\label{sec:cm}
The RHS of \eqref{eq:RHS}, the error bound between $\bar{\phi}$ and $\asynqg$, converges to zero when $\alpha \to 0$. One such situation is the limit of $\sigma / \norm{\xmean} \to 0$ while $\cm \sigma$ can remain positive, i.e., in \new{the} \emph{mutate small, but inherit large} situation, which is formally stated in the next corollary. 
\begin{corollary}\label{cor:sig}
For any positive constant $\mathrm{C} > 0$, 
\begin{equation}
  \lim_{\sigma/\norm{\xmean} \to 0} \sup_{\ns \in (0, \mathrm{C}]}\sup_{\xmean \in \R^N\setminus\{\bm{0}\}}\frac{1}{\ns}\abs*{\bar{\phi}(\xmean, \ns) - \asynqg(\ns, (w_k), \ee_{\xmean}, \Hesst) }
  = 0
  \label{eq:cor:sig}
\end{equation}
\end{corollary}
\begin{proof}
Note that the function $G(\alpha) \in \mathcal{O}(\alpha\ln(1/\alpha))$ as $\alpha \to 0$. Then, \eqref{eq:RHS} reads
\begin{multline}
\sup_{\xmean \in \R^N\setminus\{\bm{0}\}}\abs*{\bar{\phi}(\xmean, \ns) - \asynqg(\ns, (w_k), \ee_{\xmean}, \Hesst) }
\\\in \ns \lambda \mathcal{O}(\alpha\ln(1/\alpha)) \bigg[ L_1 + \cm \alpha L_2 + \cm \alpha (\lambda - 1) L_3 \bigg]
\enspace.
\label{eq:RHS-o}
\end{multline}%
Note also that
\begin{equation*}
\alpha = \frac{\ns}{\cm} \Tr(\Hesst^2)^\frac12 = \frac{\sigma \Tr(\Hess^2)^\frac12}{\norm{\nabla f(\xmean)}} \leq \frac{\sigma}{\norm{\xmean}}\frac{\Tr(\Hesst^2)^\frac12}{d_N(\Hess)}
\end{equation*}%
and $\alpha \cm = \ns \Tr(\Hesst^2)^\frac12 \leq C \Tr(\Hesst^2)^\frac12$. It implies that the RHS of \eqref{eq:RHS-o} divided by $\ns$ is in $O(\alpha\ln(1/\alpha)) \subseteq o(\alpha^{1-\epsilon})$ for any $\epsilon > 0$ under the condition $\ns \leq \mathrm{C}$. Since $\alpha \to 0$ as $\sigma / \norm{\xmean} \to 0$, \eqref{eq:RHS-o} implies \eqref{eq:cor:sig}.
\end{proof}

If $\cm$ is fixed, we have $\ns \to 0$ as $\sigma/\norm{\xmean} \to 0$. Then, the asymptotic normalized quality gain \eqref{eq:asynqg} converges towards zero as the bound on the RHS of \eqref{eq:RHS} goes to zero. The above corollary tells that the bound converges faster than $\ns$ does, while the asymptotic normalized quality gain decreases linearly in $\ns$. As a consequence, we find that the normalized quality gain approaches $- \ns \sum_{i=1}^{\lambda} w_i\E[\mathcal{N}_{i:\lambda}]$ as $\sigma/\norm{\xmean} \to 0$.

Consider the case that $\ns$ is fixed, i.e., $\cm \sigma$ is fixed. Then, from the corollary we find that the normalized quality gain converges towards $\asynqg$ in \eqref{eq:asynqg} as $\sigma/\norm{\xmean} \to 0$.

Taking $\cm \to \infty$ we obtain $\bar{\phi} \to \asynqg$. \new{Though resembling a numerical gradient estimation when $\sigma\to0$,}\del{However} it is not quite practical to take a large $\cm$. Indeed we \new{usually rather} do the opposite. For noisy optimization, the idea of \emph{rescaled mutations} (corresponding to a large $\sigma$ and a small $\cm$) that was proposed by A.~Ostermeier in 1993 (according to \cite{Rechenberg1994}) and analyzed in \cite{Bey98c,Arnold2006} is introduced to reduce the noise-to-signal ratio. If neither $\cm \not \gg 1$ nor $N \not \gg 1$, the RHS of \eqref{eq:RHS} will not be small enough to approximate the normalized quality gain by $\asynqg$ in \eqref{eq:asynqg}. Then the normalized step-size defined in \eqref{eq:optns-gen} is not guaranteed to provide an approximation of the optimal normalized step-size. 
However, in practice, we observe that the normalized step-size defined in \eqref{eq:optns-gen} provides a reasonable approximation of the optimal normalized step-size for $\cm \geq 1$. We will see it in Figure~\ref{fig:opt}.


\subsection{Infinite Dimensional Case}\label{sec:infn}

The other situation \new{when}\del{that} $\alpha \to 0$ occurs is the limit $N \to \infty$ under the condition $\lim_{N\to\infty}\new{\Tr(\Hesst^2)}\del{\Tr(\Hesst^2)^\frac12}{} = 0$.\todo{I am not sure I understand why we have the power of $\frac12$ here. At least it seems neither necessary nor notationally consistent with previous claims.} In this case, since $\ee^\T\Hesst\ee \leq \Tr(\Hesst^2)^\frac12 \to 0$, the limit expression $\asynqg$ converges to $\bar{\phi}_{\infty}$, i.e., the same limit on the sphere function. It is stated in the following corollary, which is a generalization of the result obtained in \cite{Arnold2005foga} from the sphere function to a general convex quadratic function. 

\begin{corollary}\label{cor:hess}
Let $(\Hess_N)_{N=1}^{\infty}$ be the sequence of nonnegative definite matrices satisfying $\lim_{N\to\infty} \Tr(\Hess^2)^\frac12 = 0$. Then, 
\begin{equation}
\lim_{N\to\infty} \sup_{\xmean\in\R^N\setminus\{\bm{0}\}} \abs*{\asynqg(\ns, (w_k)_{k=1}^\lambda, \ee, \Hesst_N) - \bar{\phi}_\infty(\ns, (w_k)_{k=1}^\lambda)} = 0 
\enspace,
\label{eq:cor:hess:1}
\end{equation}%
where $\bar{\phi}_\infty(\ns, (w_k)_{k=1}^\lambda) = - \ns \sum_{i=1}^\lambda w_i \E[\mathcal{N}_{i:\lambda}] - \ns^2 / (2 \muw)$ as defined in \eqref{eq:nqlim}. 
Moreover,
\begin{equation}
\lim_{N\to\infty}\sup_{\ns \in (0, \mathrm{C}]}\sup_{\xmean \in \R^N\setminus\{\bm{0}\}}\abs*{\bar{\phi}(\xmean, \ns) - \bar{\phi}_\infty(\ns, (w_k)_{k=1}^\lambda) }
= 0
\enspace.
\label{eq:cor:hess:2}
\end{equation}%
\end{corollary}

Corollary~\ref{cor:hess} shows that the normalized quality gain on a convex quadratic function converges towards the asymptotic normalized quality gain derived on the infinite dimensional sphere function as $\Tr(\Hess^2)^\frac12\to 0$. It implies that the optimal values of the recombination weights and the normalized step-size are independent of the Hessian of the objective function, and given by \eqref{eq:optw}. It is a nice feature since we do not need to tune the weight values depending on the function. Since any twice continuously differentiable function is locally approximated by a quadratic function, the optimal weights derived here are expected to be locally optimal on any twice continuously differentiable function. 

In the above corollary, the population size $\lambda$ is a constant over the dimension $N$. However, in the default setting of the CMA-ES, the population size is $\lambda = 4 + \lfloor 3 \ln(N) \rfloor$, meaning that the population size is unbounded. %
\rev{If the population size increases to infinity as $N \to \infty$, it is not guaranteed that the per-evaluation progress $\bar{\phi} / \lambda$ converges to $\bar{\phi}_\infty / \lambda$ as $N \to \infty$.} %
The following proposition provides a sufficient condition on the recombination weights and the population size such that the per-evaluation progress $\bar{\phi} / \lambda$ converges to $\bar{\phi}_\infty / \lambda$ when $\lambda$ increases as $N$ increases.
\begin{proposition}\label{prop:hesslam}
  Let $(\Hess_N)_{N=1}^{\infty}$ be the sequence of the Hessian matrix that satisfies $\lim_{N\to\infty} \Tr(\Hess_N^2)^\frac12 = 0$. Let $(\lambda_N)_{N=1}^{\infty}$ be the sequence of the population size and $(w_k^{N})_{k=1}^{\lambda_N}$ be the sequence of the weights for the population size $\lambda_N$. Suppose for an arbitrarily small positive $\epsilon$, 
  \begin{gather*}
  \lambda_N^2 \in o\left(\frac{1}{d_1(\Hesst_N)}\right)
  \quad \text{and} \\
  \max\left(\lambda_N, L_1^\frac{1}{1-\epsilon} \lambda_N^\frac{2 - \epsilon}{1-\epsilon},
  L_2^\frac{1}{2-\epsilon} \lambda_N^\frac{3 - \epsilon}{2-\epsilon},
  L_3^\frac{1}{2-\epsilon} \lambda_N^\frac{4 - \epsilon}{2-\epsilon}\right) \in O\left(\frac{1}{ \Tr(\Hesst_N^2)^\frac12}\right)
  \enspace.
\end{gather*}%
Then, 
\begin{equation}
\lim_{N\to\infty}\sup_{\ns \in (0, 2 \ns^*)}\sup_{\xmean \in \R^N\setminus\{\bm{0}\}} \frac{1}{\lambda_N}\abs*{\bar{\phi}(\xmean, \ns) - \bar{\phi}_\infty(\ns, (w_k^N)_{k=1}^{\lambda_N}) }
= 0
\enspace,
\label{eq:cor:hesslam:1}
\end{equation}%
  where $\ns^*$ is the normalized step-size optimal for $\bar{\phi}_\infty(\ns, (w_k^N)_{k=1}^{\lambda_N})$, which is formulated in \eqref{eq:optns}.
\end{proposition}

\begin{proof}
A sufficient condition for $\asynqg$ to converge to $\bar{\phi}_\infty$ for $\ns \in (0, 2 \ns^*)$ is that the third term on the RHS of \eqref{eq:asynqg} converges to zero as $N \to \infty$. As we know from \ref{apdx:nos} that $\bm{w}^\T \etwo_{(\lambda)} \bm{w} \leq d_1(\etwo_{(\lambda)}) \norm{\bm{w}}^2 \leq \Tr(\etwo_{(\lambda)}) \norm{\bm{w}}^2 = \lambda / \muw$. On the other hand, $\ee_{\xmean}^\T \Hesst \ee_{\xmean}$ is no greater than the greatest eigen value $d_1(\Hesst)$ of $\Hesst$. The third term on the RHS of \eqref{eq:asynqg} is maximized when $\ns = \ns^* = \muw (-\bm{w}^\T \eone_{(\lambda)})$, where, $\muw (\bm{w}^\T\eone_{(\lambda)}) \leq \sum_{i=1}^\lambda \abs{\E[\mathcal{N}_{i:\lambda}]}$ and $\frac1\lambda\sum_{i=1}^\lambda \abs{\E[\mathcal{N}_{i:\lambda}]} \to (2 / \pi)^\frac12$. From these arguments derives that the third term on the RHS of \eqref{eq:asynqg} converges to zero as $N \to \infty$ provided that $\lambda^2 d_1(\Hesst) \to 0$ as $N \to 0$.

Next we consider the convergence of the bound (RHS of \eqref{eq:RHS-o}).
Remember that $\alpha = (\ns / \cm) \Tr(\Hesst^2)^\frac12$ and $G(\alpha) \in O(\alpha \ln(1/\alpha))$. For $\ns \in (0, 2 \ns^*)$, we have  $(\ns/\cm) \Tr(\Hesst^2)^\frac12 \leq 2(\ns^*/\cm) \Tr(\Hesst^2)^\frac12$. Since $\ns^* \in O(\lambda)$, we have $\ns \in O(\lambda)$ and $\alpha \in O(\lambda \Tr(\Hess^2)^\frac12)$. Then, the RHS of \eqref{eq:RHS-o} divided by $\lambda$,
\begin{align*}
  \MoveEqLeft[2]O\left(
  \ns \alpha \ln(1/\alpha) [L_1 + L_2 \alpha  + L_3 \lambda \alpha]
  \right)
  \\
  &\subseteq
  o\left(
  \ns \alpha^{1-\epsilon} [L_1 + L_2 \alpha  + L_3 \lambda \alpha]
  \right)
  \\
  &\subseteq
  o\big(
  \lambda^{2-\epsilon} \Tr(\Hesst^2)^\frac{1-\epsilon}{2} [L_1 + L_2 \lambda \Tr(\Hesst^2)^\frac12 + L_3 \lambda^2 \Tr(\Hesst^2)^\frac12]
    \big),
\end{align*}%
where the convergence of each term is supposed in the proposition. 
\end{proof}

Consider the truncation weights with a fixed selection ratio $\lambda = \rho \mu$ for some $\rho > 1$ and the sequence of the Hessian matrices such that the condition number is bounded. From \ref{apdx:Lbound:deriv}, we have that $L_1 \in O(\lambda^{-1/2})$, $L_2 \in O(\lambda^{-3/2})$, and $L_3 \in O(\lambda^{-3/2})$. Moreover, we have $1/d_1(\Hesst) \in O(N)$ and $1/\Tr(\Hesst^2)^\frac12 \in O(N^\frac12)$. Then, the condition of Proposition~\ref{prop:hesslam} reduces to $\lambda \in o(N^\frac13)$. This condition is a rigorous (but probably not tight) bound for the scaling of $\lambda$ such that the per-iteration convergence rate of a $(\mu\slash \mu,\lambda)$-ES with a fixed $\lambda / \mu$ on the sphere function scales like $O(\lambda/N)$ \cite[Equation 6.140]{BeyerBOOK2001}. One can also deduce\todo{deduce?} the condition for the optimal weights and the CMA-type positive weights (positive half of the optimal weights).

\subsection{Effect of the Eigenvalue Distribution of the Hessian Matrix}

\begin{figure}[t]
  \centering
  \begin{subfigure}{0.45\hsize}
    \centering
    \includegraphics[width=0.8\hsize]{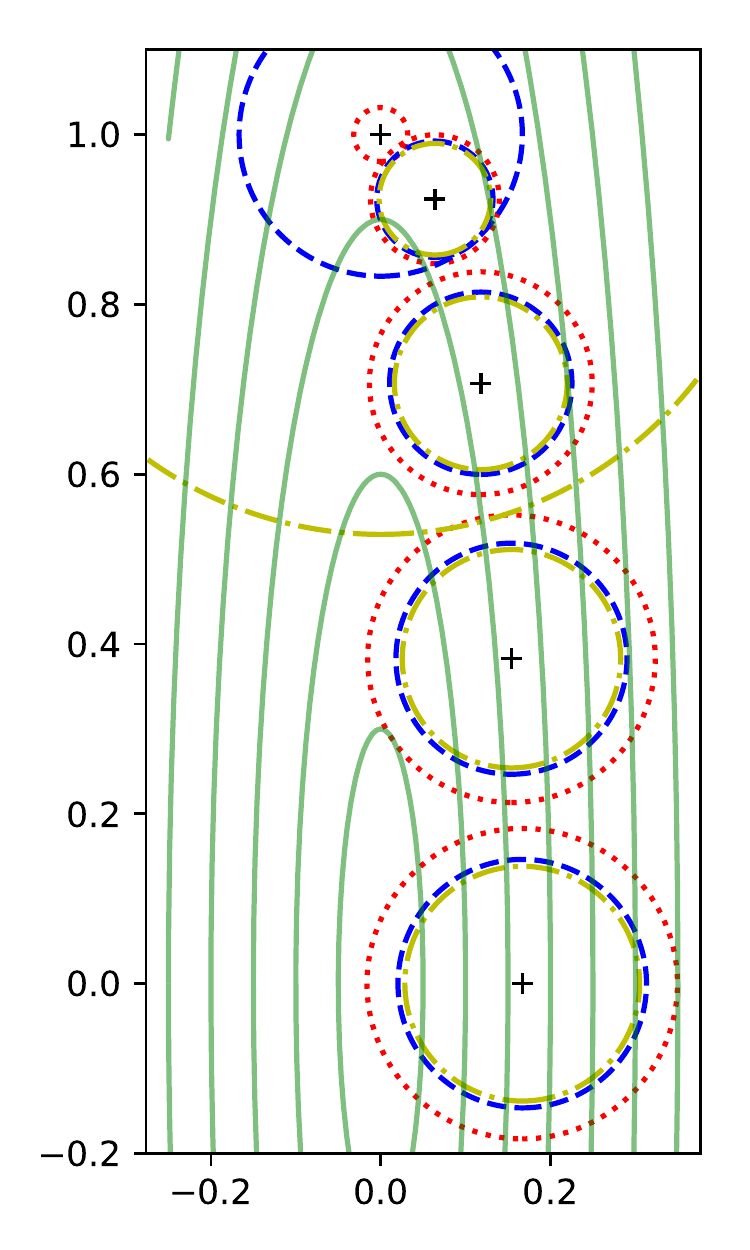}%
    \caption{Optimal weights \eqref{eq:optw}}\label{fig:a}
  \end{subfigure}%
  \begin{subfigure}{0.45\hsize}
    \centering
    \includegraphics[width=0.8\hsize]{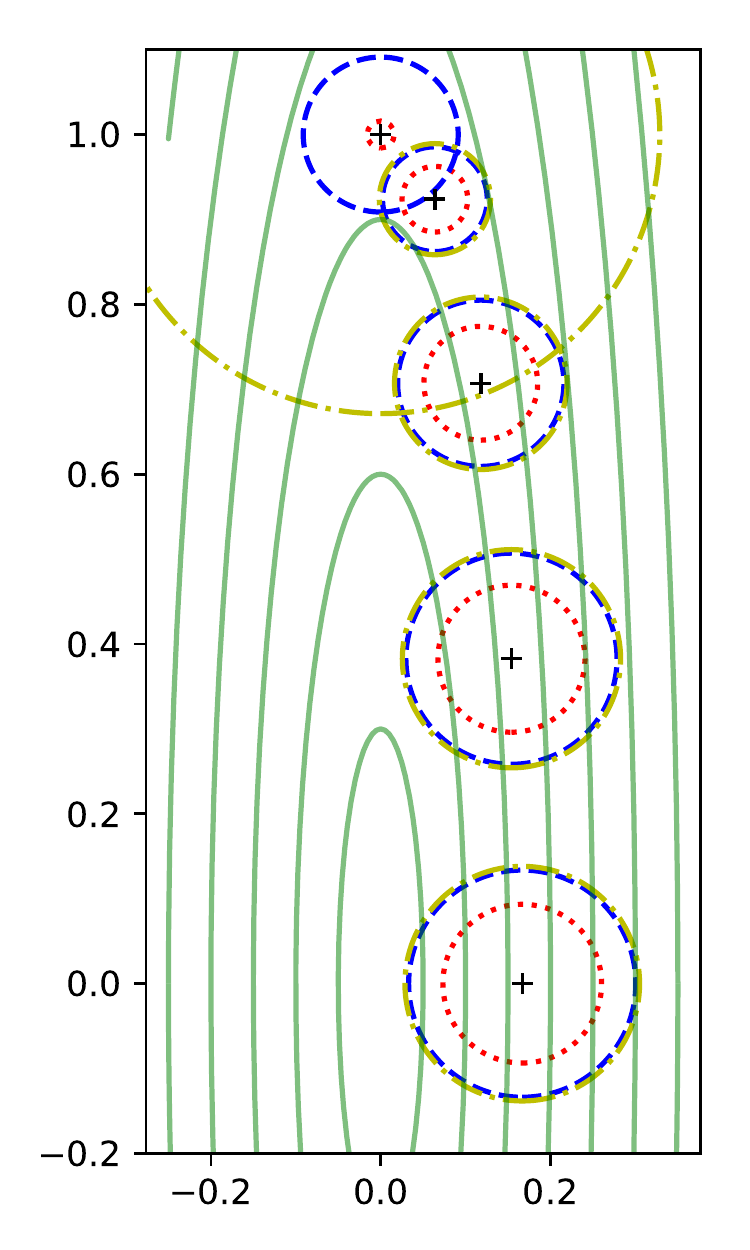}%
    \caption{Optimal positive weights}\label{fig:b}
  \end{subfigure}%
\caption{The \new{asymptotically} optimal step-size $\sigma^* = \ns^* \norm{\nabla f(\xmean)}$ on $f(x) = x^\T \Hess x / 2$ with $\Hess = \diag(1, 36)$. The circles with radius $\ns^* \norm{\nabla f(\xmean)}$ centered at $\xmean = 2\Hess^{-\frac12}(\cos(\theta), \sin(\theta))$ with $\theta = \pi/2, 3\pi/8, \pi/4, \pi/8, 0$ are displayed, where the \new{asymptotically} optimal normalized step-size $\ns^*$ is computed using \eqref{eq:optns-gen} with the optimal weights \eqref{eq:optw} \new{(left) and with the optimal positive weights (right)}. Red dotted: $\lambda =2$, Blue dashed: $\lambda = 10$, Yellow dot-dashed: $\lambda = 50$.\todo{To my understanding we do not see the \emph{actually} optimal step-size, but the "limit-optimal" step-size value which is actually not optimal here. If so, we should probably make this clear to prevent misunderstandings. I am also very puzzled by the observation that the optimal step-size is in most cases larger for the smaller $\lambda$. What am I missing? } 
}
\label{fig:visualize}
\end{figure}

Corollary~\ref{cor:hess} implies \new{that} the optimal recombination weights are independent of the Hessian in the limit $N \to \infty$ as long as $\lim_{N\to\infty}\Tr(\Hess^2) = 0$. Moreover, Proposition~\ref{prop:optw-gen} and Corollary~\ref{cor:sig} together impl\new{y that}\del{ies} the same values approximate the optimal recombination weights for a sufficiently large $\lambda$ in the limit of $\sigma / \norm{\xmean} \to 0$. On the other hand, the step-size and the progress rate depend on the Hessian. In the following we discuss the effect of the Hessian followed by a simulation. To make the discussion more intuitive, we remove the condition $\Tr(\Hess) = 1$ and consider an arbitrary non-negative definite symmetric $\Hess$. All the statements above still hold by replacing $\Hess$ with $\Hess / \Tr(\Hess)$.

Given $(w_k)_{k=1}^{\lambda}$, the optimal normalized step-size $\ns^*$ and the normalized quality gain $\bar{\phi}_{\infty}(\ns, (w_k))$ are independent of the Hessian $\Hess$ and the distribution mean $\xmean$. However, the step-size $\sigma = (\ns / \cm) \norm{\nabla f(\xmean)}$ and the quality gain $\phi(\xmean, \sigma) = g(\xmean) \bar{\phi}_{\infty}(\ns, (w_k)_{k=1}^{\lambda}) $ depend on them through $\norm{\nabla f(\xmean)}$ and $g(\xmean) = \norm{\nabla f(\xmean)}^2 / f(\xmean)$. If $\xmean$ is on a contour ellipsoid ($f(\xmean) = 1$ for example), $g(\xmean)$ increases as $\norm{\nabla f(\xmean)}$. In other words, the greater the optimal step-size is, the greater quality gain we achieve. These quantities are bounded as
\begin{gather*}
\frac{\eig_{N}(\Hess)}{\Tr(\Hess)}\norm{\xmean - x^*} \leq \norm{\nabla f(\xmean)} \leq \frac{\eig_{1}(\Hess)}{\Tr(\Hess)} \norm{\xmean - x^*} 
\text{ and } \\
\frac{\eig_{N}(\Hess)}{\Tr(\Hess)} \leq \frac{g(\xmean)}{2} \leq \frac{\eig_{1}(\Hess)}{\Tr(\Hess)} 
\enspace.
\end{gather*}%
The lower and upper equalities for both of the above inequalities hold if and only if $\xmean - x^*$, or equivalently $\ee_{\xmean}$, is parallel to the eigenspace corresponding to the smallest and largest eigenvalues of $\Hess$, respectively. 
Therefore, the optimal step-size and the quality gain can be different by the factor of at most $\Cond(\Hess) = \eig_{1}(\Hess) / \eig_{N}(\Hess)$. Figure~\ref{fig:visualize} visualizes example cases. The \new{asymptotic} optimal step-size heavily depends on the location of $\xmean$ if $\Hess$ is ill-conditioned. If we focus on the area around each circle, the function landscape looks like a parabolic ridge function. Note that a relatively large step-size displayed at $\xmean = (0, 1)$ for $\lambda > 2$ is because $\ee_\xmean^\T \Hess \ee_\xmean \ll 1$ in \eqref{eq:optns-gen}, resulting in $\ns^* \propto \muw$.
The asymptotic normalized quality gain is derived for the limit $\sigma / \norm{\xmean} \to 0$, and the update of the mean vector results in an approximation of the negative gradient direction. If the mean vector is exactly on the longest axis of the hyper-ellipsoid, the gradient points to the optimal solution and a large normalized step-size is desired. However, this never happens in practice, since the mean vector will not be exactly in such a situation \rev{with probability one}.
\new{We also remark that the asymptotically optimal normalized step-size \new{\eqref{eq:optns-gen}} is not monotonically increasing w.r.t.\ $\lambda$. Indeed, we see in Figure~\ref{fig:a} smaller step-sizes for greater $\lambda$ values, whereas they are monotonic in Figure~\ref{fig:b}. The main difference is that the step-sizes with the optimal weights for $\lambda = 2$ can be greater than those with the optimal positive weights. Nonetheless, there is no guarantee that these figures reflect the \emph{actually} optimal step-size precisely since displayed are the step-size optimal in the limit of $\cm$ to infinity. Further investigation needs to be conducted.}

Table~\ref{tbl:hess} summarizes $\eig_{N}(\Hess) / \Tr(\Hess)$, $\eig_{1}(\Hess) / \Tr(\Hess)$ and $\Tr(\Hess^2) / \Tr(\Hess)^2$ for different types of $\Hess$. The greater the first two quantities are, the greater the optimal step-size and hence the quality gain are. The smaller the last quantity is, the more reliable it is to approximate $\bar{\phi}$ with $\bar{\phi}_\infty$. If the condition number $\alpha = \Cond(\Hess)$ is fixed, the worst case ($\eig_{N}(\Hess) / \Tr(\Hess)$) is maximized when the function has a discus type structure and is minimized when the function has a cigar type structure. The value of $\eig_{N}(\Hess) / \Tr(\Hess)$ will be close to $1/N$ as $N \to \infty$ for the discus type function, whereas it will be close to $1 / (N\alpha)$ for the cigar. Therefore, the discus type function is as easy to solve as the sphere function if $N \gg \alpha$, while the cigar type function takes roughly $1/\alpha$ times more iterations to reach the same target function value. On the other hand, the inequality $\Tr(\Hess^2) / \Tr(\Hess)^2 < 1 / (N - 1)$ holds independently of $\alpha$ on the cigar type function, while $\Tr(\Hess^2) / \Tr(\Hess)^2$ depends heavily on $\alpha$ on the discus type function. The fraction will not be sufficiently small and we can not approximate the normalized quality gain by $\bar{\phi}_\infty$ unless $\alpha \ll N$ holds\footnote{However, the worst case scenario on the discus type function, $1 / (\alpha + (N-1))$, describes an empirical observation \cite{Hansen2001ec} that the convergence speed of evolution strategy with isotropic distribution does not scale down with $N$ for $N \ll \alpha$.}. 

\rev{%
  The condition $\lim_{N\to\infty} \Tr(\Hess^2)/\Tr(\Hess)^2 = 0$ also hold for some positive semi-definite $\Hess$, where only $M < N$ eigenvalues of $\Hess$ are positive and the others are zero. That is, $\eig_1(\Hess) \geq \dots \geq \eig_{M}(\Hess) > 0$ and $\eig_{M+1}(\Hess) = \dots = \eig_{N}(\Hess)$. In this case, the condition $\Tr(\Hess^2) / \Tr(\Hess)^2 \to 0$ holds only if the dimension $M$ of the effective search space tends to infinity as $N \to \infty$. The above inequalities are refined as follows. Let $\xmean^+$ and $\xmean^-$ be the decomposition of $\xmean$ such that $\xmean^{-}$ is the projection of $\xmean$ onto the hyper-plane through $x^*$ spanned by the eigenvectors of $\Hess$ corresponding to the zero eigenvalue, and $\xmean^{+} = \xmean - \xmean^{-}$.
Then,   
\begin{gather*}
\frac{\eig_{M}(\Hess)}{\Tr(\Hess)} \leq \frac{g(\xmean)}{2} \leq \frac{\eig_{1}(\Hess)}{\Tr(\Hess)} 
\\
\frac{\eig_{M}(\Hess)}{\Tr(\Hess)}\norm{\xmean^{+}} \leq \frac{\norm{\nabla f(\xmean)}}{\Tr(\Hess)} \leq \frac{\eig_{1}(\Hess)}{\Tr(\Hess)} \norm{\xmean^{+}} \enspace.
\end{gather*}
In this case, $g(\xmean)$ can be $2/M$ if $\eig_1(\mathbf{A}) = \cdots = \eig_M(\mathbf{A}) > 0$ and $\eig_i(\mathbf{A}) = 0$ for $i \in \llbracket M+1, N \rrbracket$. The quality gain is then proportional to $2/M$, instead of $2/N$. That is, the evolution strategy with the optimal step-size solves the quadratic function with the effective rank $M$ defined on the $N$ dimensional search space as efficiently as it solves its projection onto the effective search space.
}%

\begin{table}[t]
\centering
\caption{Different types of the eigenvalue distributions of $\Hess$. The second to fourth types (discus: $\eig_1(\Hess) = \alpha$ and $\eig_2(\Hess) = \dots = \eig_N(\Hess) = 1$, ellipsoid: $\eig_{i}(\Hess) = \alpha^\frac{i-1}{N-1}$, cigar: $\eig_1(\Hess) = \dots = \eig_{N-1}(\Hess) = \alpha$ and $\eig_N(\Hess) = 1$) have the condition number $\Cond(\Hess) = \eig_1(\Hess)/\eig_N(\Hess) = \alpha$, while the last type has the condition number $N$.}
\label{tbl:hess}
\begin{tabular}{lccc}
\toprule 
Type & $\frac{\eig_{N}(\Hess)}{\Tr(\Hess)}$ & $\frac{\eig_{1}(\Hess)}{\Tr(\Hess)}$ & $\frac{\Tr(\Hess^2)}{\Tr(\Hess)^2}$ 
\\
\midrule
Sphere & 
	$\frac{1}{N}$ & 
	$\frac{1}{N}$ & 
	$\frac{1}{N}$
\\
Discus & 
	$\frac{1}{(N-1) + \alpha}$ & 
	$\frac{\alpha}{(N-1) + \alpha}$ & 
	$\frac{(N-1) + \alpha^2}{((N-1) + \alpha)^2}$ 
\\
Ellipsoid & 
	$\frac{\alpha^\frac{1}{N-1} - 1}{\alpha^\frac{N}{N-1} - 1}$ & 
	$\frac{\alpha^\frac{N}{N-1} - \alpha}{\alpha^\frac{N}{N-1} - 1}$ & 
	$\frac{\big(\alpha^\frac{2N}{N-1} - 1\big) / \big(\alpha^\frac{2}{N-1} - 1\big)}{\big(\alpha^\frac{N}{N-1} - 1\big)^2 / \big(\alpha^\frac{1}{N-1} - 1\big)^2}$ 
\\
Cigar  & 
	$\frac{1}{(N-1)\alpha + 1}$ & 
	$\frac{\alpha}{(N-1)\alpha + 1}$ & 
	$\frac{(N-1)\alpha^2 + 1}{((N-1)\alpha + 1)^2}$ 
\\
\midrule
$\eig_{i}(\Hess) = i$ & 
	$\frac{1}{N(N+1)/2}$ & 
	$\frac{1}{(N+1)/2}$ & 
	$\frac{\frac16 N(N+1)(2N+1)}{\big(N(N+1)/2\big)^2}$ 
\\
\bottomrule
\end{tabular}
\end{table}

\paragraph{Comment on the algorithm dynamics}

The asymptotic quality gain depends on the distribution mean $\xmean$ through $g(\xmean)$. In practice, we observe near worst case performance with $g(\xmean) \approx 2 \eig_{N}(\Hess) / \Tr(\Hess)$, which implies that $\xmean - x^*$ is almost parallel to the eigenspace corresponding to the smallest eigenvalue $d_{N}(\Hess)$ of the Hessian matrix. We provide an intuition to explain this behavior, which will be useful to understand the algorithm, even though the argument is not fully rigorous.

Consider \rev{Algorithm~\ref{alg:es}} with scale-invariant step-size (Definition~\ref{def:ns}). Lemma~\ref{lem:u1} implies that the order of the function values $f(X_i)$ coincide with the order of $[\NN_{i}]_1 = \ee^\T (X_i - \xmean^{(t)}) / \sigma^{(t)}$, where $\ee = \nabla f(\xmean^{(t)}) / \norm{\nabla f(\xmean^{(t)})}$. This is because if $Z \sim \mathcal{N}(\bm{0}, \eye)$, then $Z^\T \Hess Z / \Tr(\Hess)$ in \eqref{eq:fdiff} almost surely converges to one by the strong law of large numbers as $N \to \infty$ under $\Tr(\Hess^2) / \Tr(\Hess)^2 \to 0$. It means that the function value of a \rev{candidate} solution is determined solely by the first component on the right-hand side of \eqref{eq:fdiff}, that is, $\ee^\T (X_i - \xmean^{(t)}) / \sigma^{(t)}$. 
Since the ranking of the function value only depends on $\ee^\T (X_i - \xmean^{(t)}) / \sigma^{(t)}$, one may rewrite the update of the mean vector as
\begin{equation}
\xmean^{(t+1)} = \xmean^{(t)} 
	+ \cm\sigma^{(t)}\sum_{i=1}^{\lambda} w_i \NN_{i:\lambda}(0, 1) \cdot \ee\\
	+ \cm\sigma^{(t)}\mu_{w}^{-\frac12}\NN(\bm{0}, \eye - \ee\ee^\T) \enspace,
	\label{eq:m-limit}
      \end{equation}%
where $\NN_{i:\lambda}(0, 1)$ are the $i$-th order statistics from $\lambda$ population of $\NN(0, 1)$, and $\NN(\bm{0}, \eye - \ee\ee^\T)$ is the normally distributed random vector with mean vector $\bm{0}$ and the degenerated covariance matrix $\eye - \ee \ee^\T$. It indicates that the mean vector moves along the gradient direction with the distribution $\cm\sigma^{(t)}\sum_{i=1}^{\lambda} w_i \NN_{i:\lambda}(0, 1)$, while it moves randomly in the subspace orthogonal to the gradient with the distribution $\cm\sigma^{(t)}\mu_w^{-\frac12}\NN(\bm{0}, \eye - \ee\ee^\T)$. 

If the function is spherical, i.e.\ $\Hess \propto \eye$, the mean vector does a symmetric, unbiased random walk on the surface of a hypersphere while the radius of the hypersphere gradually decreases due to the second term on \eqref{eq:m-limit}. If the function is a general convex quadratic function, $\Hess \not \propto \eye$, the
corresponding random walk on the surface of a hyperellipsoid becomes biased. Then, $\xmean - x^*$ tends to be parallel to the eigenspace corresponding to the smallest eigenvalue $d_{N}(\Hess)$, which means that the quality gain is close to the worst case of $\eig_N(\Hess) / \Tr(\Hess)$. The reason may be explained as follows. The progress in one step is the largest in the short axis direction (parallel to the eigenvector corresponding to the largest eigenvalue of $\Hess$), and the smallest in the long axis direction (parallel \rev{to} the eigenvector corresponding to the largest eigenvalue of $\Hess$). The \new{short axis direction is quickly optimized and the} situation \new{gets close}\del{quickly becomes closer} to the worst case, \new{where}\del{while} it takes \new{many}\del{longer} iterations to escape from\del{near worst situation}. Therefore, we observe \new{the} near worst situation in practice. Further theoretical investigation on the distribution of $\ee = \nabla f(\xmean^{(t)}) / \norm{\nabla f(\xmean^{(t)})}$ should be done in the future work.


\newcommand{\figsize}{0.333\linewidth}
\begin{figure}[!t]
\centering
\begin{subfigure}{\figsize}
\includegraphics[width=\hsize]{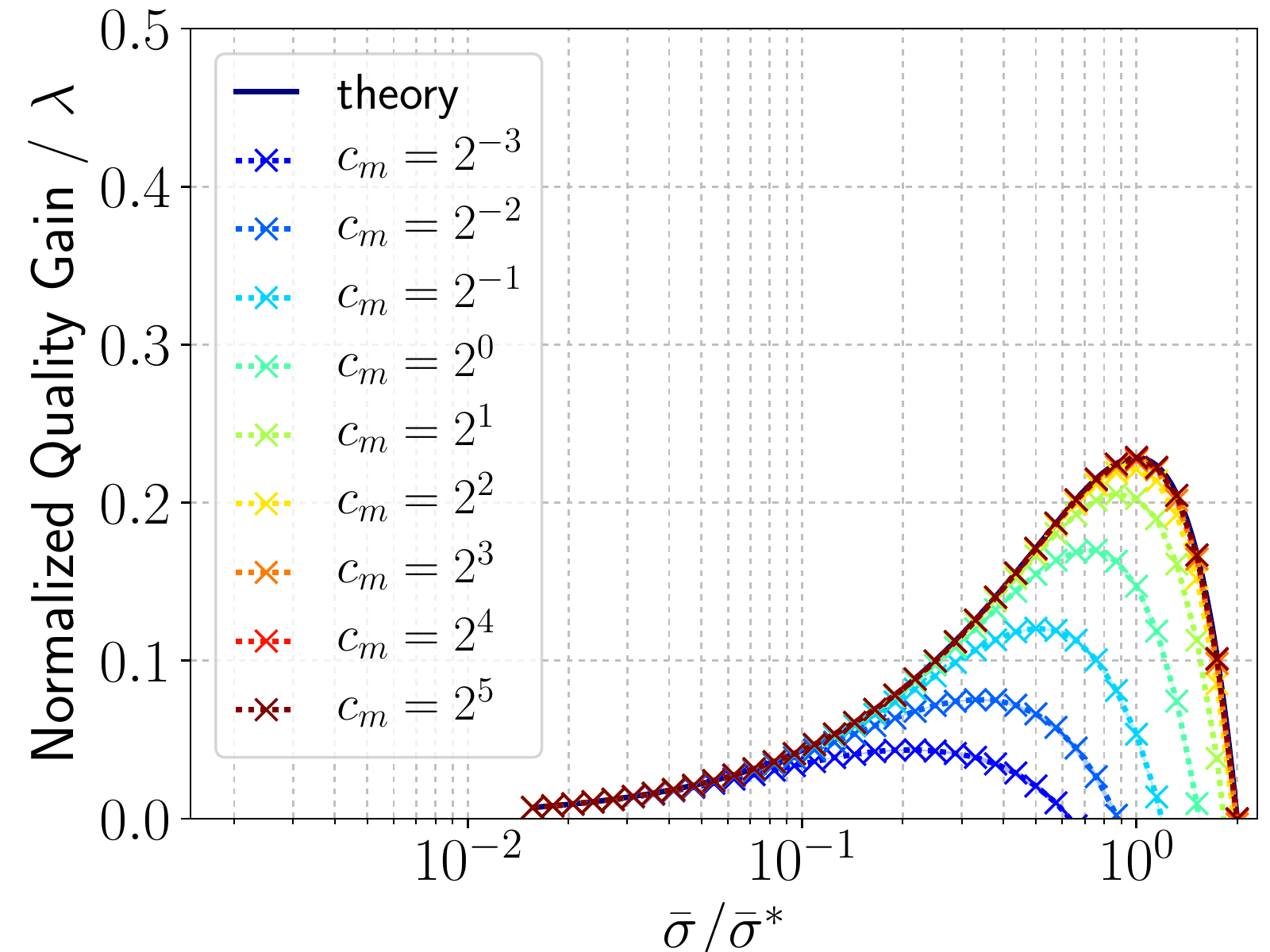}
\end{subfigure}%
\begin{subfigure}{\figsize}
\includegraphics[width=\hsize]{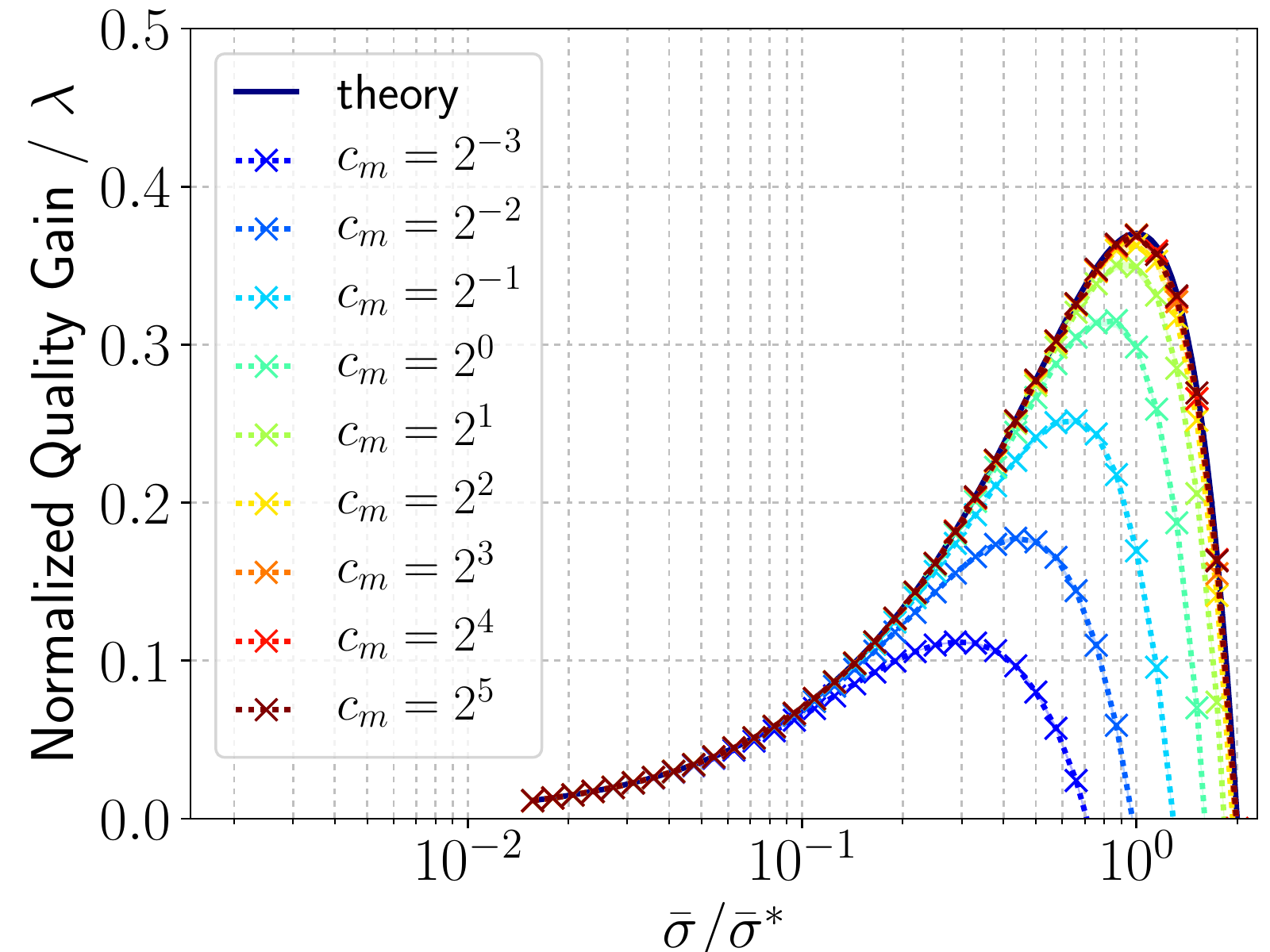}
\end{subfigure}%
\begin{subfigure}{\figsize}
\includegraphics[width=\hsize]{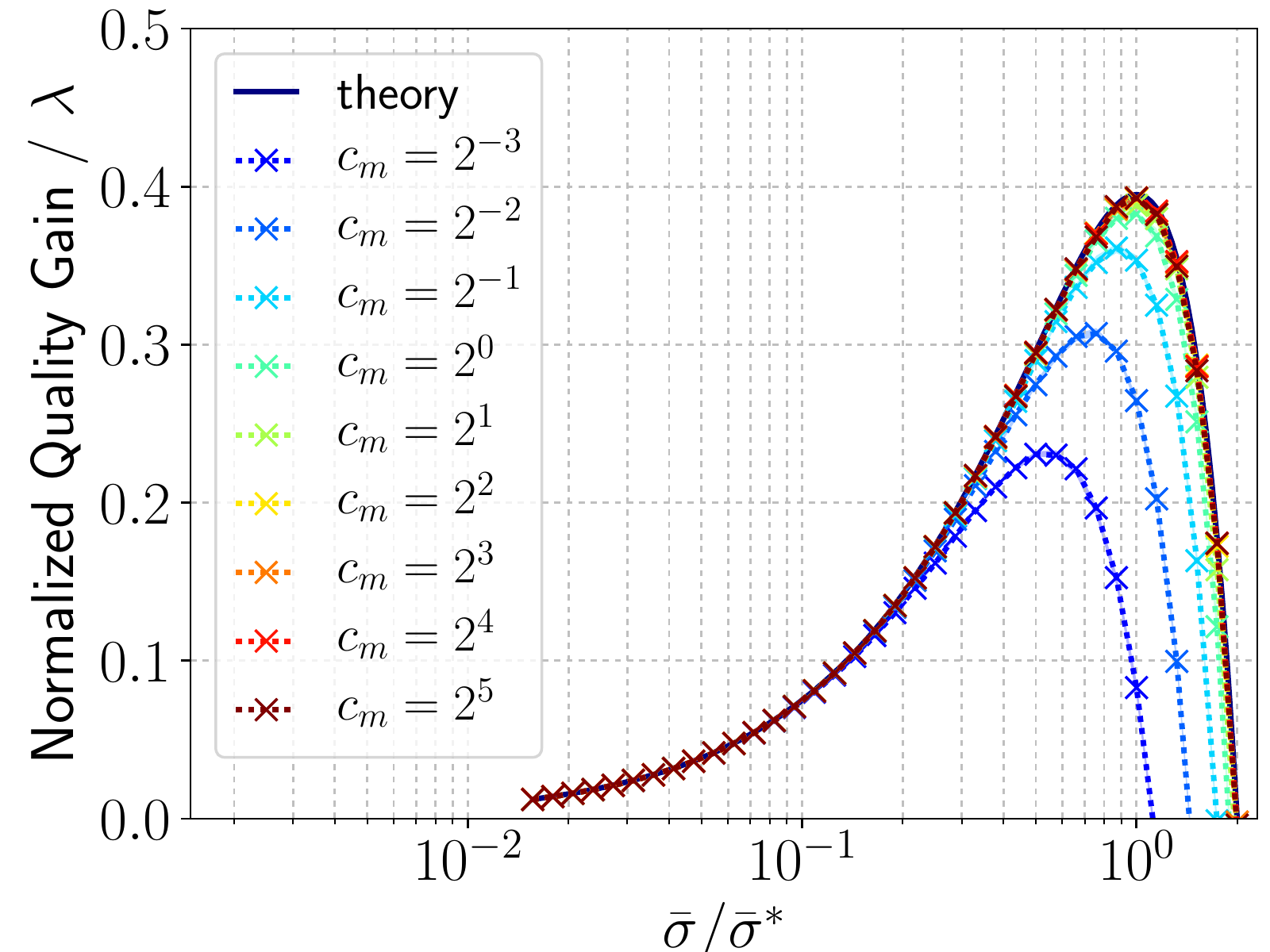}
\end{subfigure}%
\\
\begin{subfigure}{\figsize}
\includegraphics[width=\hsize]{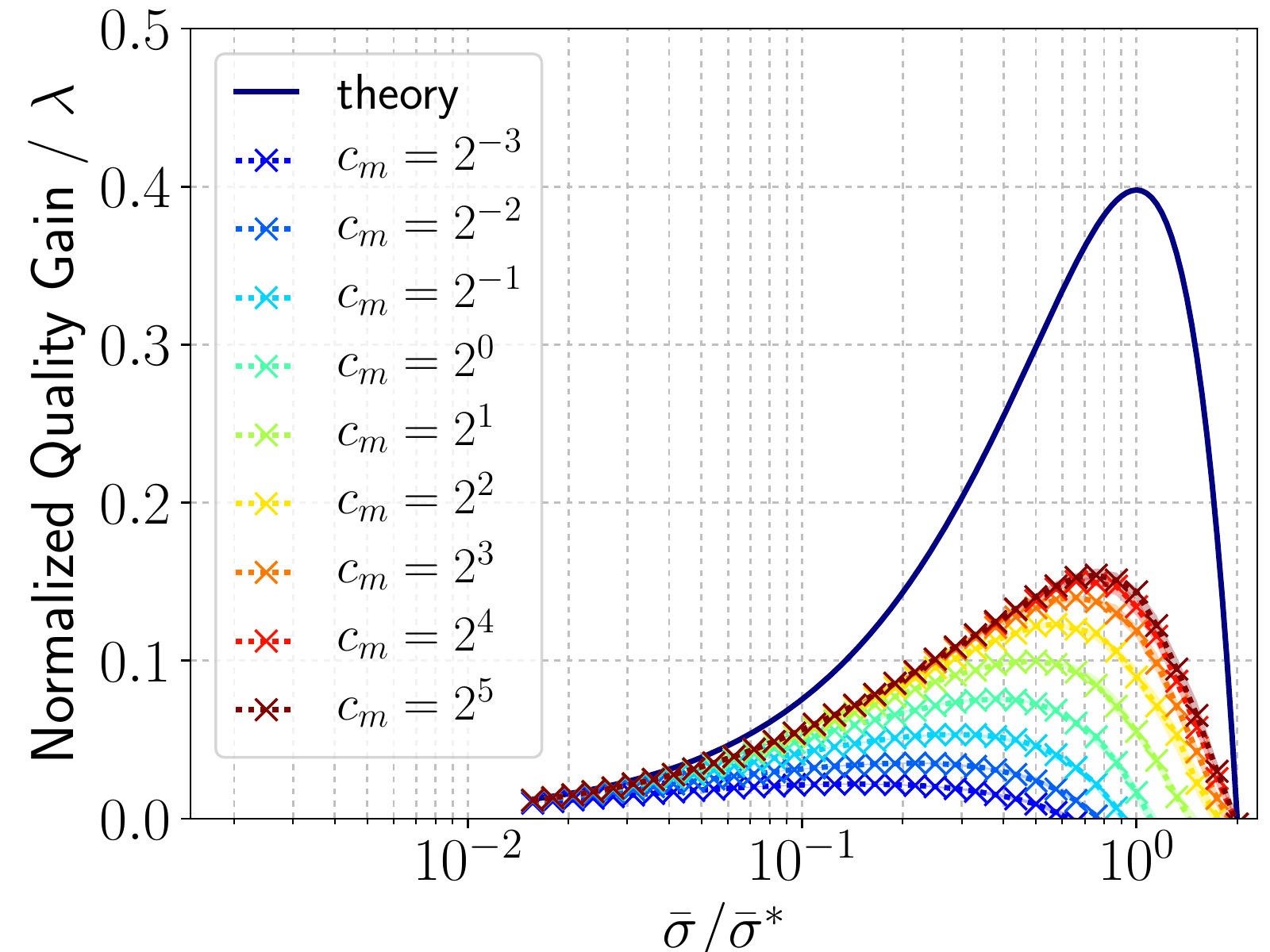}
\end{subfigure}%
\begin{subfigure}{\figsize}
\includegraphics[width=\hsize]{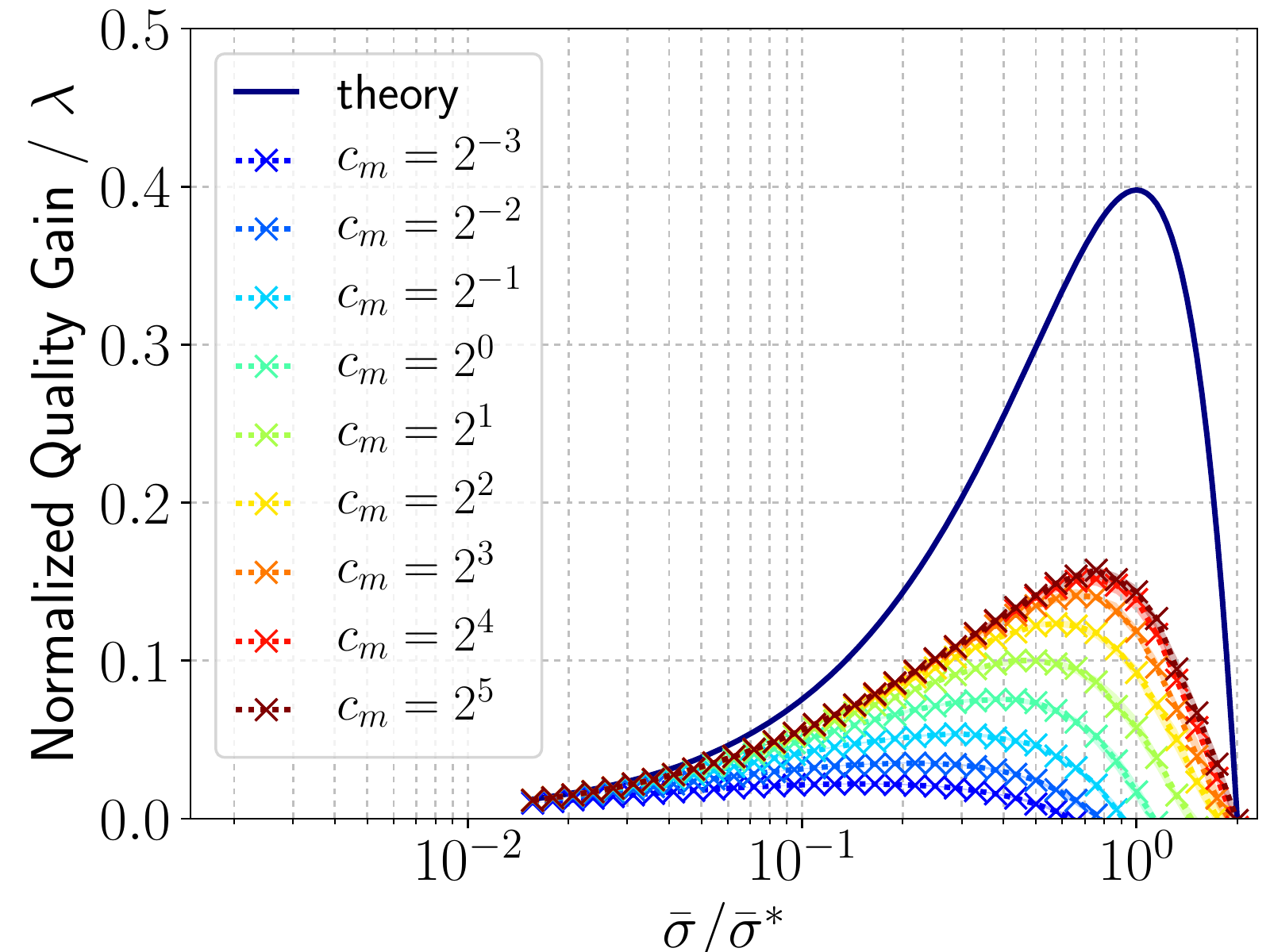}
\end{subfigure}%
\begin{subfigure}{\figsize}
\includegraphics[width=\hsize]{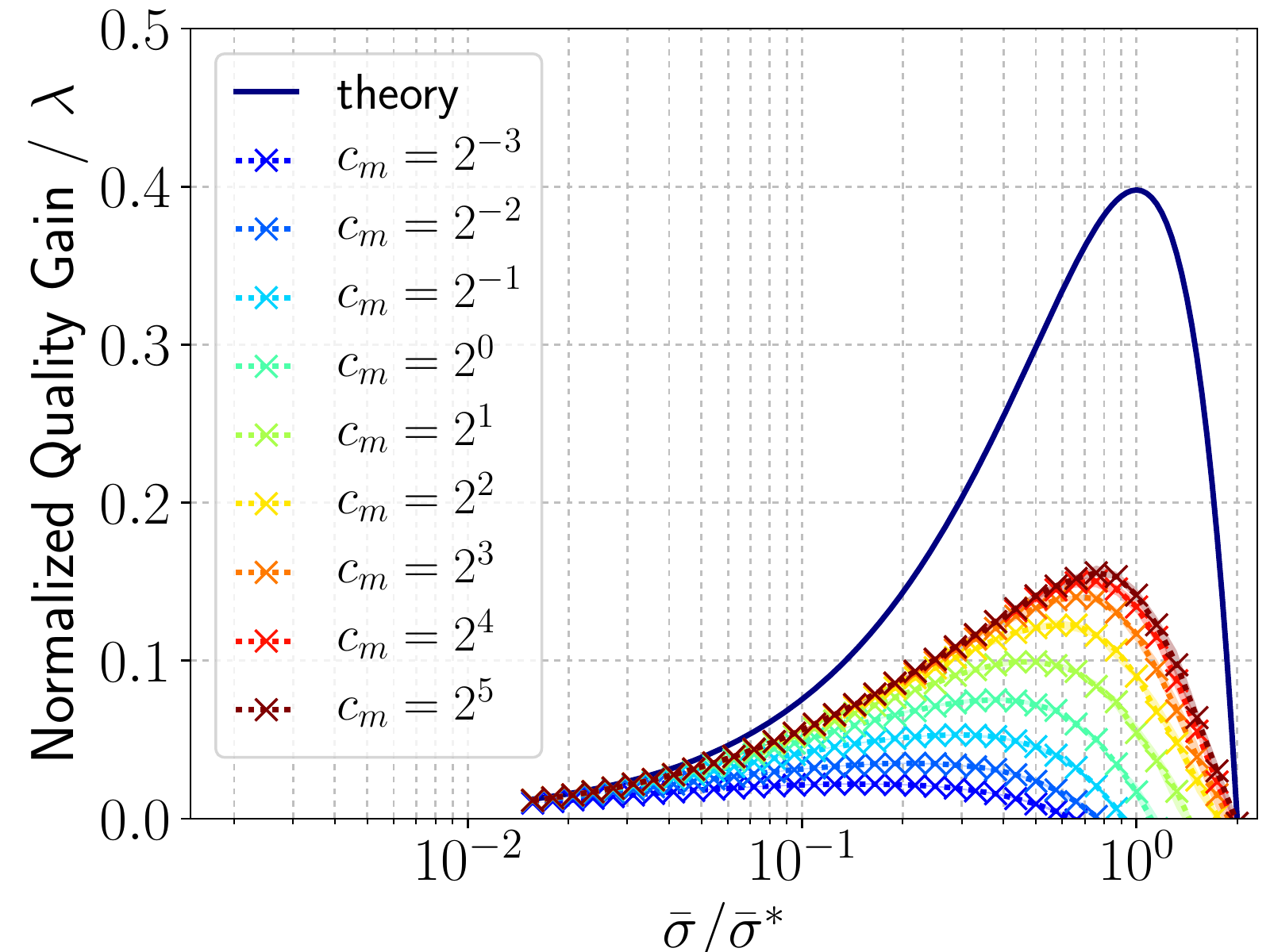}
\end{subfigure}%
\\
\begin{subfigure}{\figsize}
\includegraphics[width=\hsize]{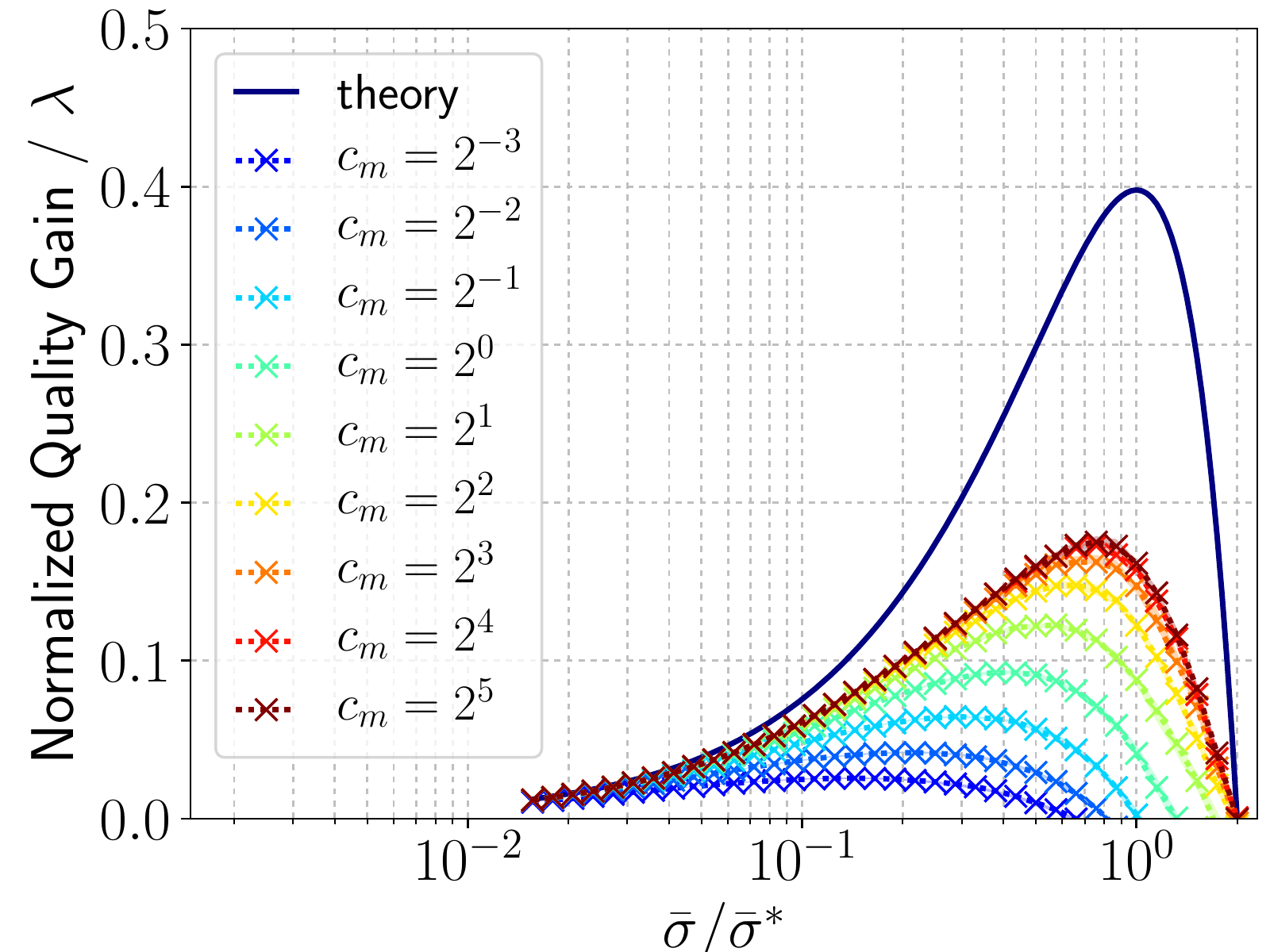}
\end{subfigure}%
\begin{subfigure}{\figsize}
\includegraphics[width=\hsize]{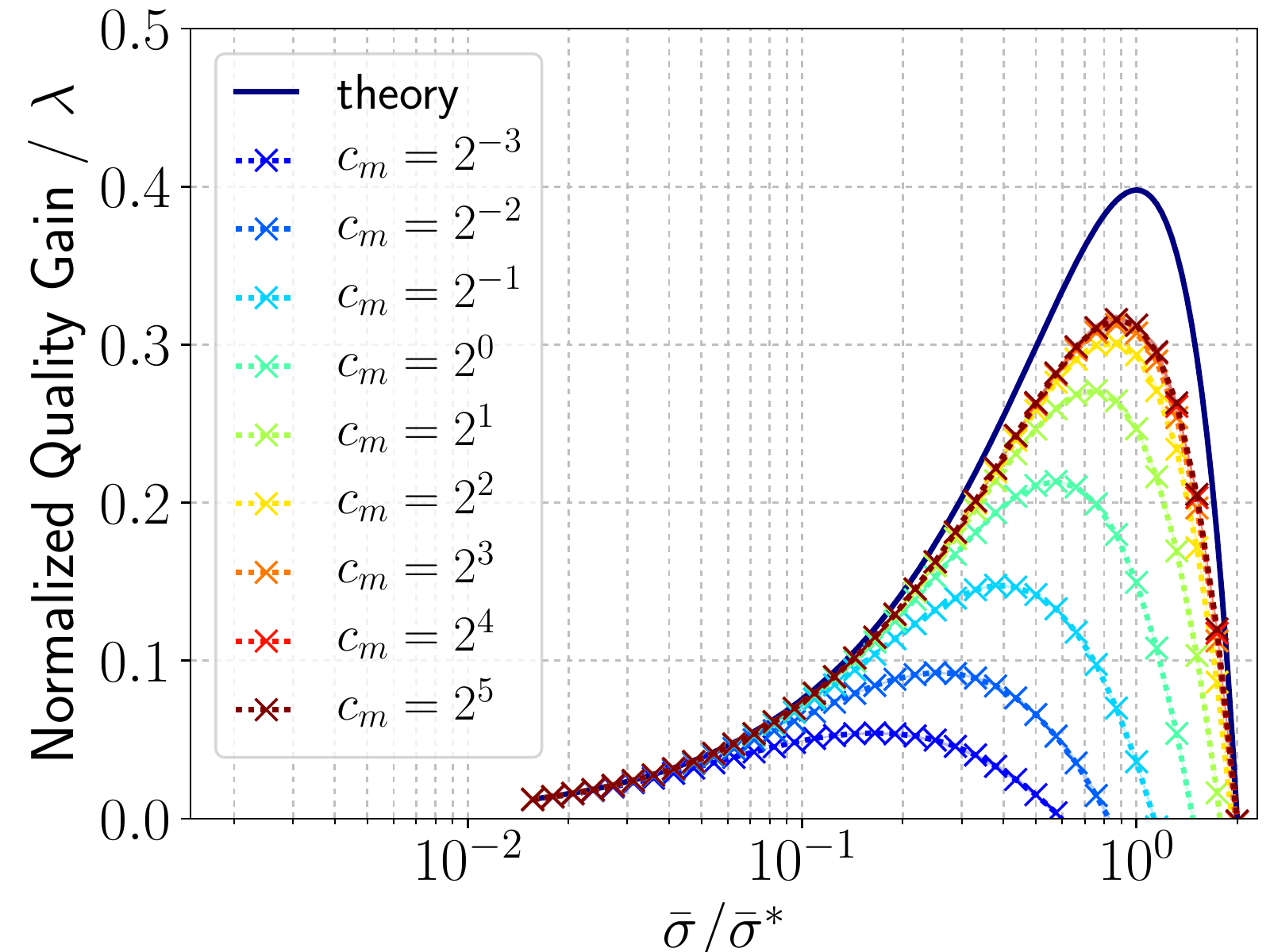}
\end{subfigure}%
\begin{subfigure}{\figsize}
\includegraphics[width=\hsize]{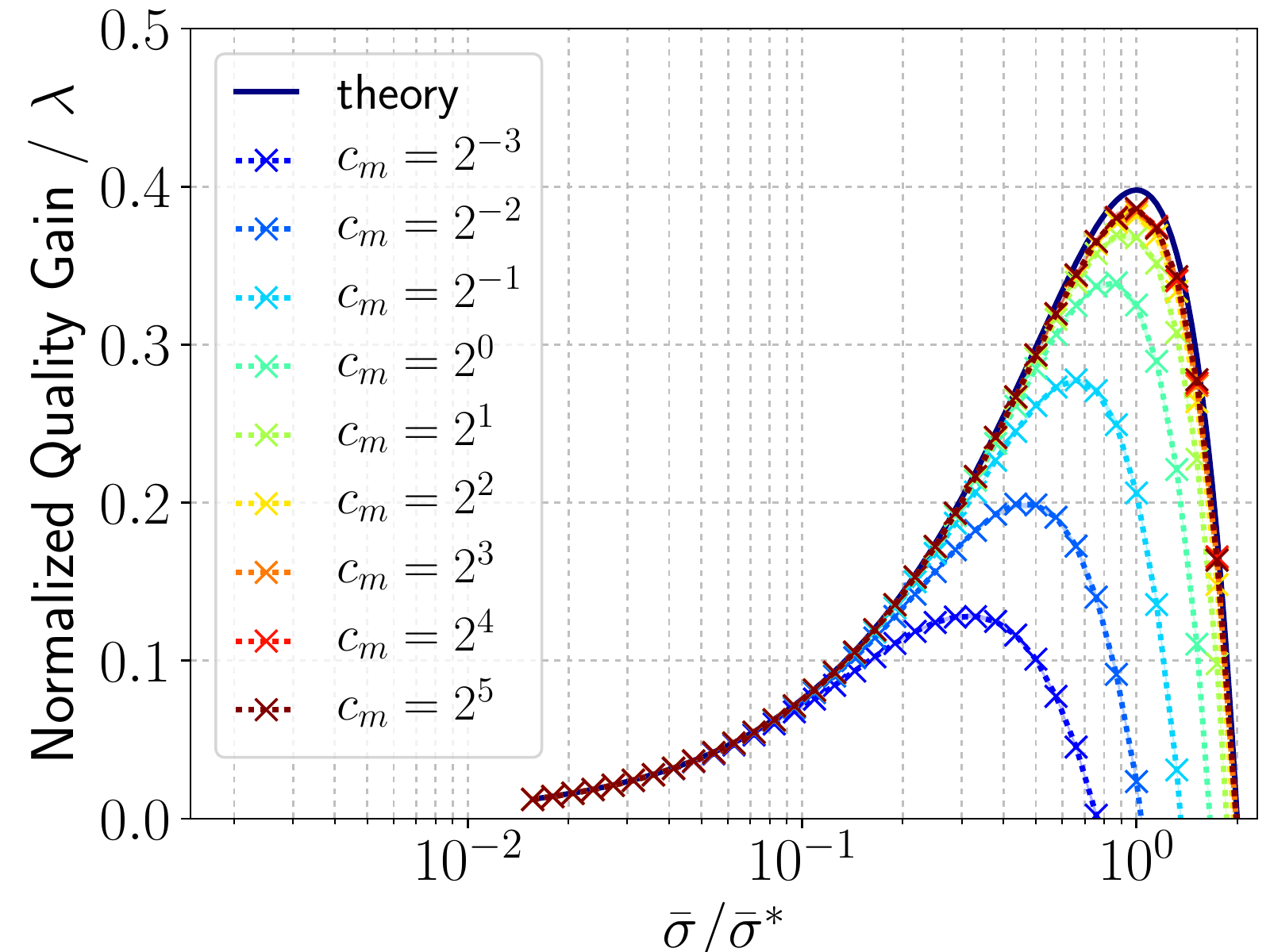}
\end{subfigure}%
\\
\begin{subfigure}{\figsize}
\includegraphics[width=\hsize]{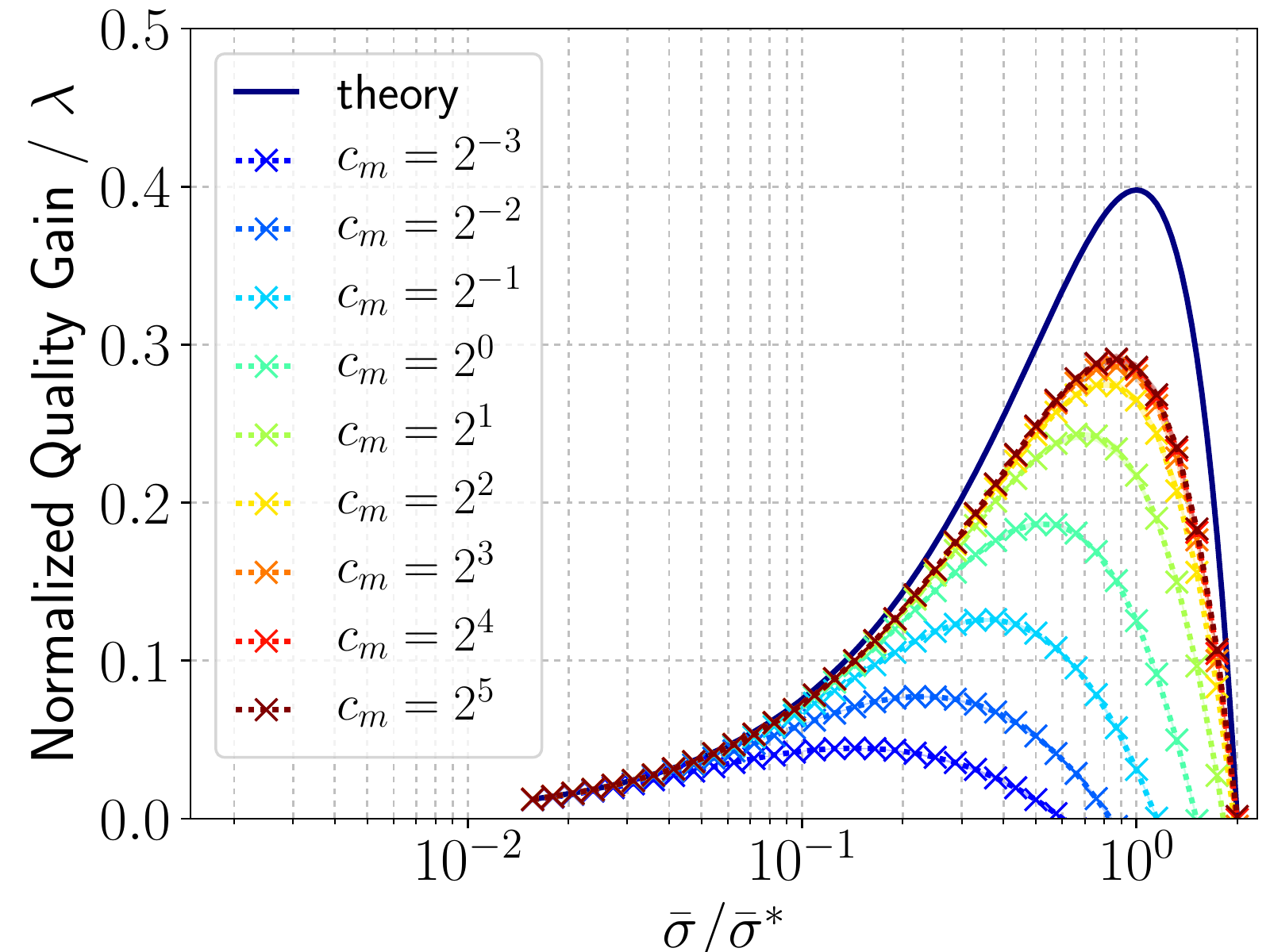}
\end{subfigure}%
\begin{subfigure}{\figsize}
\includegraphics[width=\hsize]{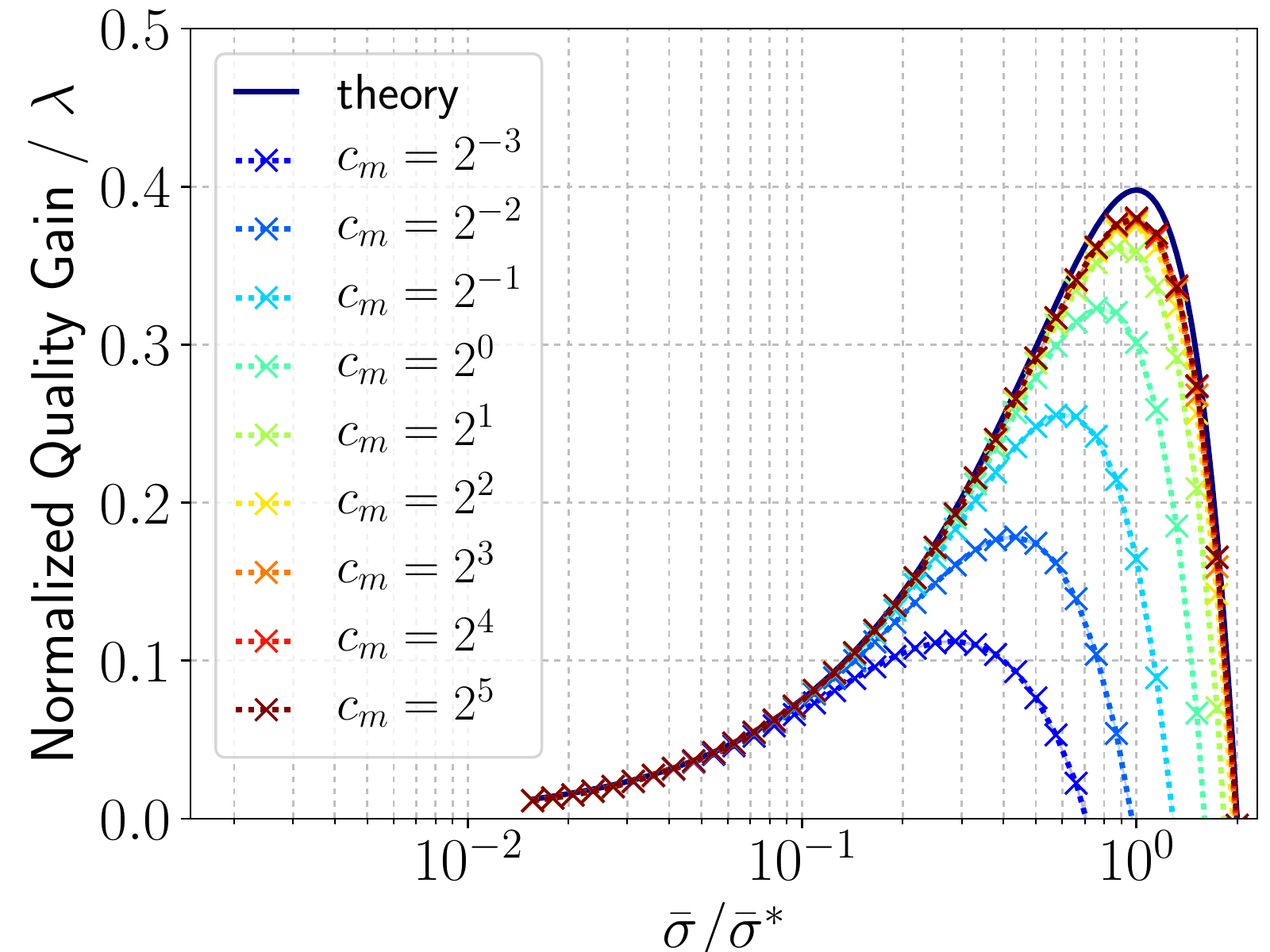}
\end{subfigure}%
\begin{subfigure}{\figsize}
\includegraphics[width=\hsize]{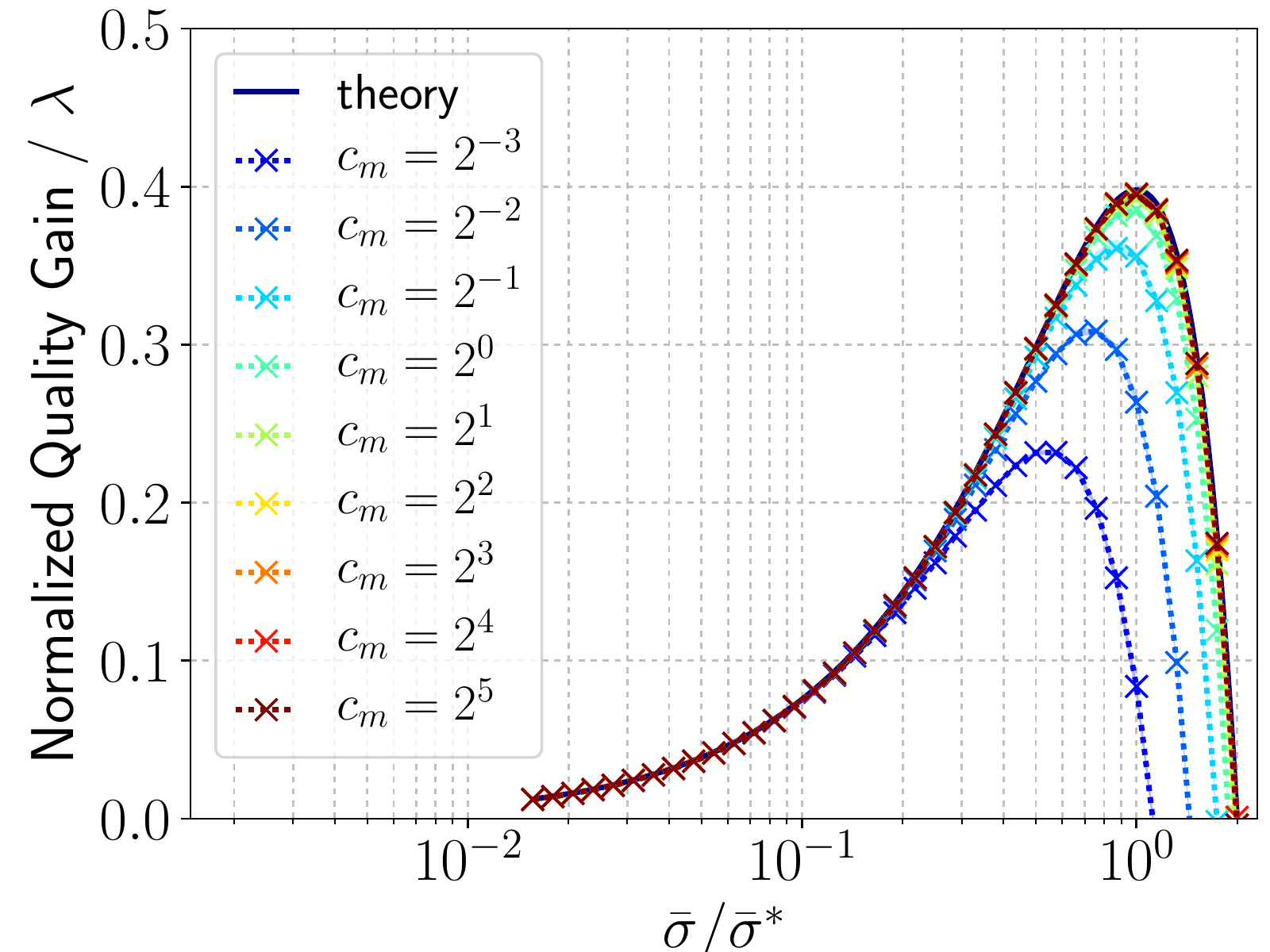}
\end{subfigure}%
\caption{Empirical normalized quality gain on four convex quadratic functions, Sphere, Discus, Ellipsoid and Cigar (from top to bottom) of dimension $N = 10$, $100$ and $1000$ (from left to right). The optimal weights \eqref{eq:optw} \rev{are} used and $\lambda = 10$.}
\label{fig:opt}
\end{figure}

\begin{figure}[!t]
\centering
\begin{subfigure}{\figsize}
\includegraphics[width=\hsize]{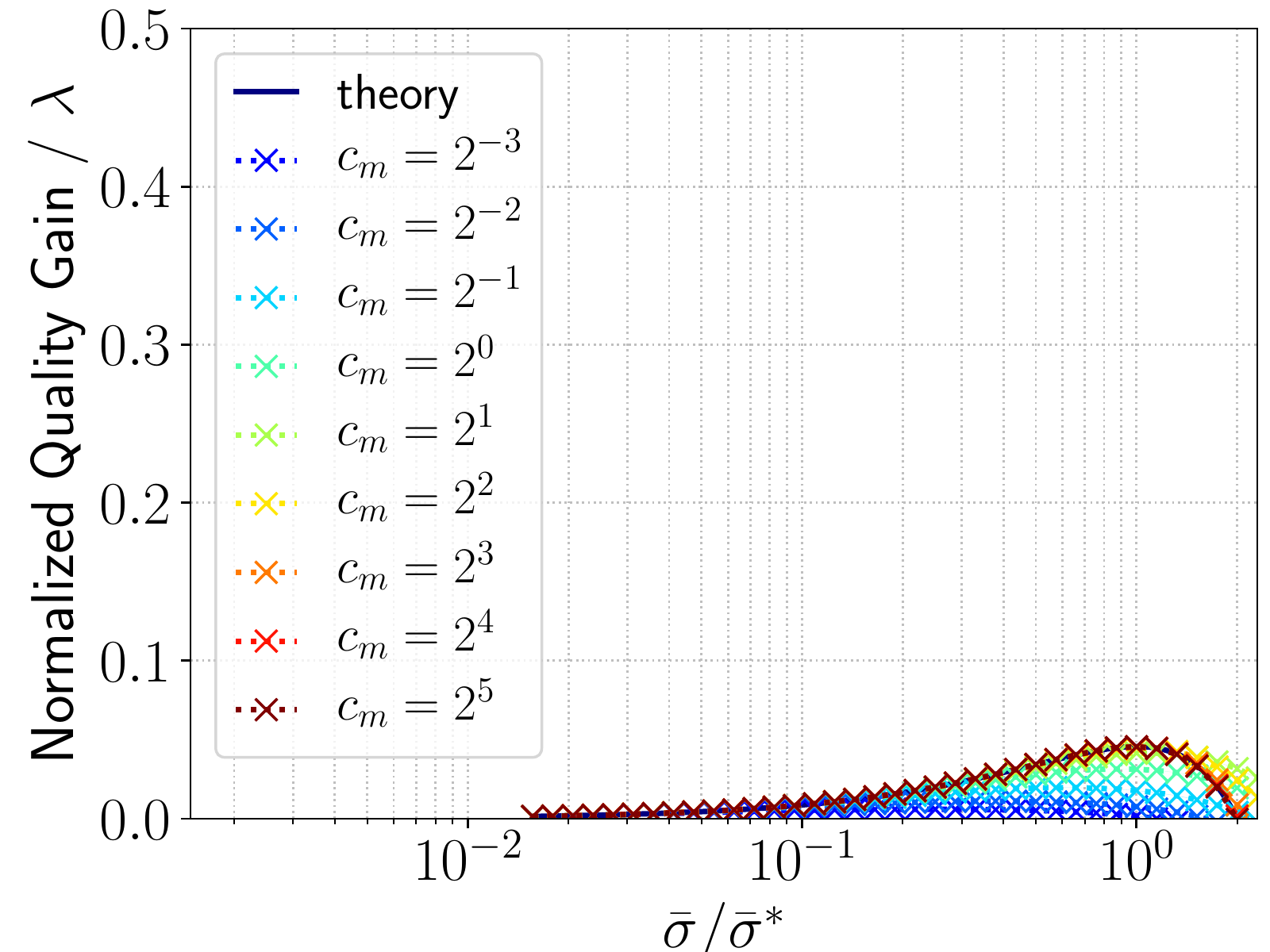}
\end{subfigure}%
\begin{subfigure}{\figsize}
\includegraphics[width=\hsize]{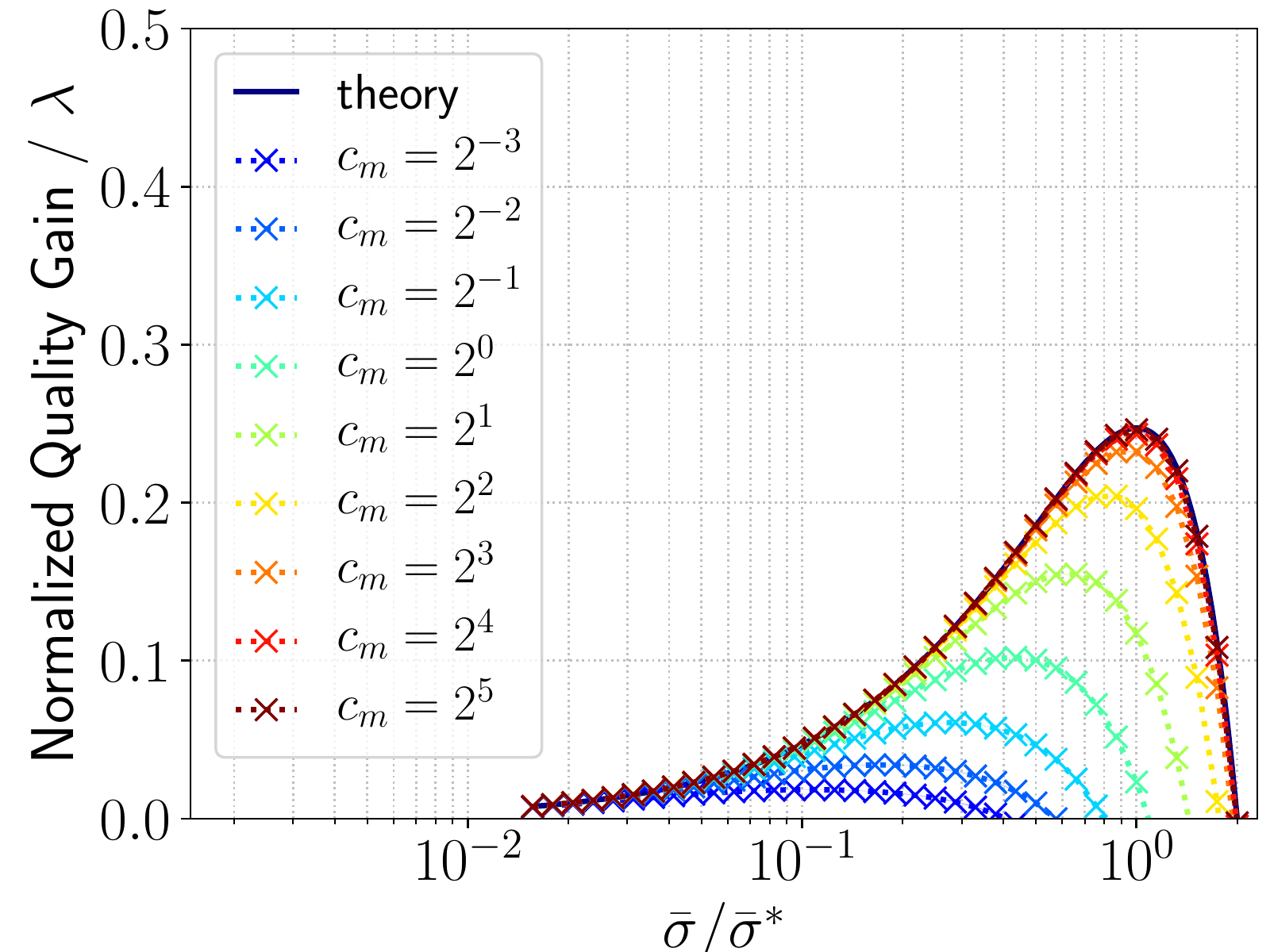}
\end{subfigure}%
\begin{subfigure}{\figsize}
\includegraphics[width=\hsize]{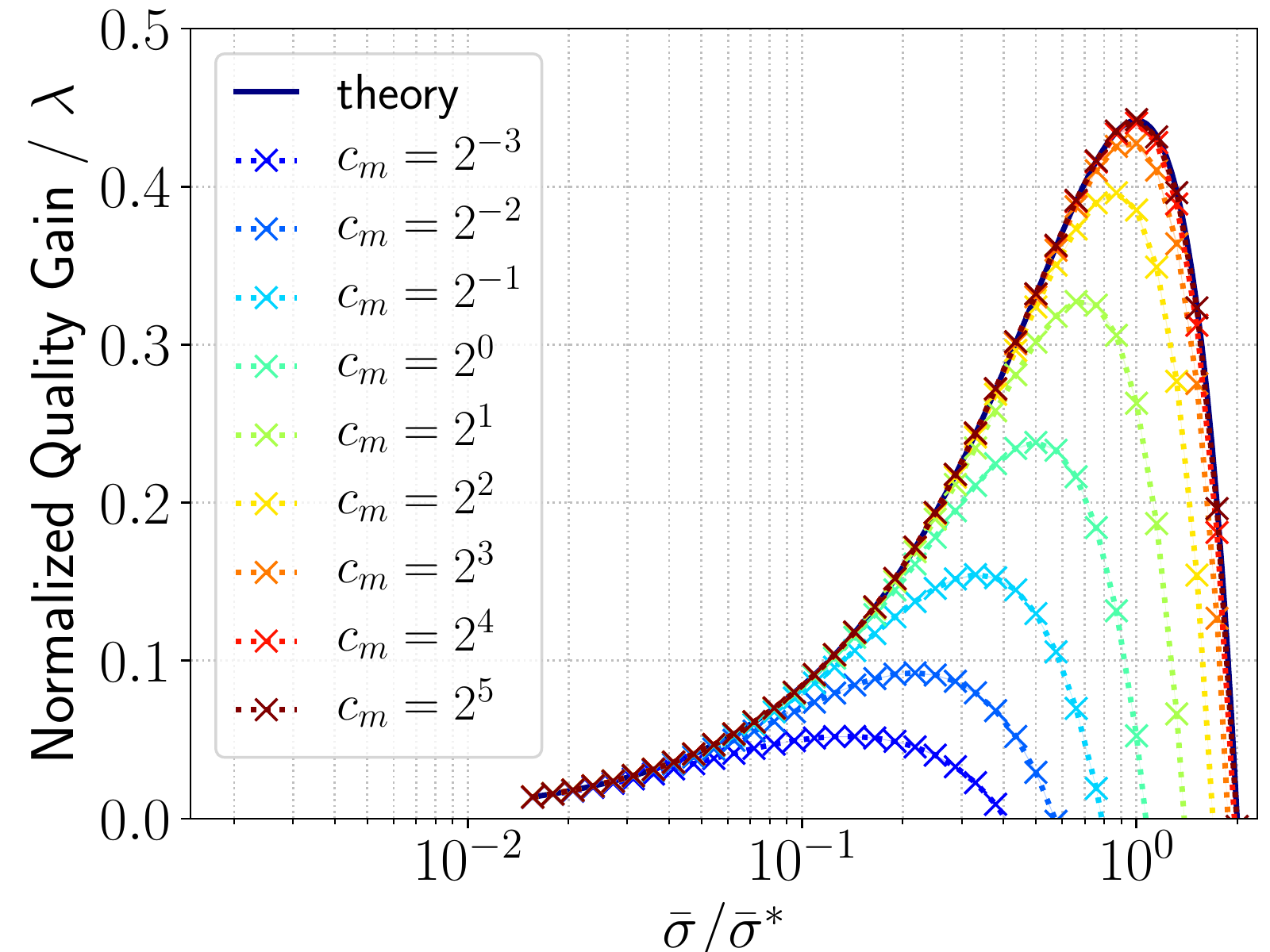}
\end{subfigure}%
\\
\begin{subfigure}{\figsize}
\includegraphics[width=\hsize]{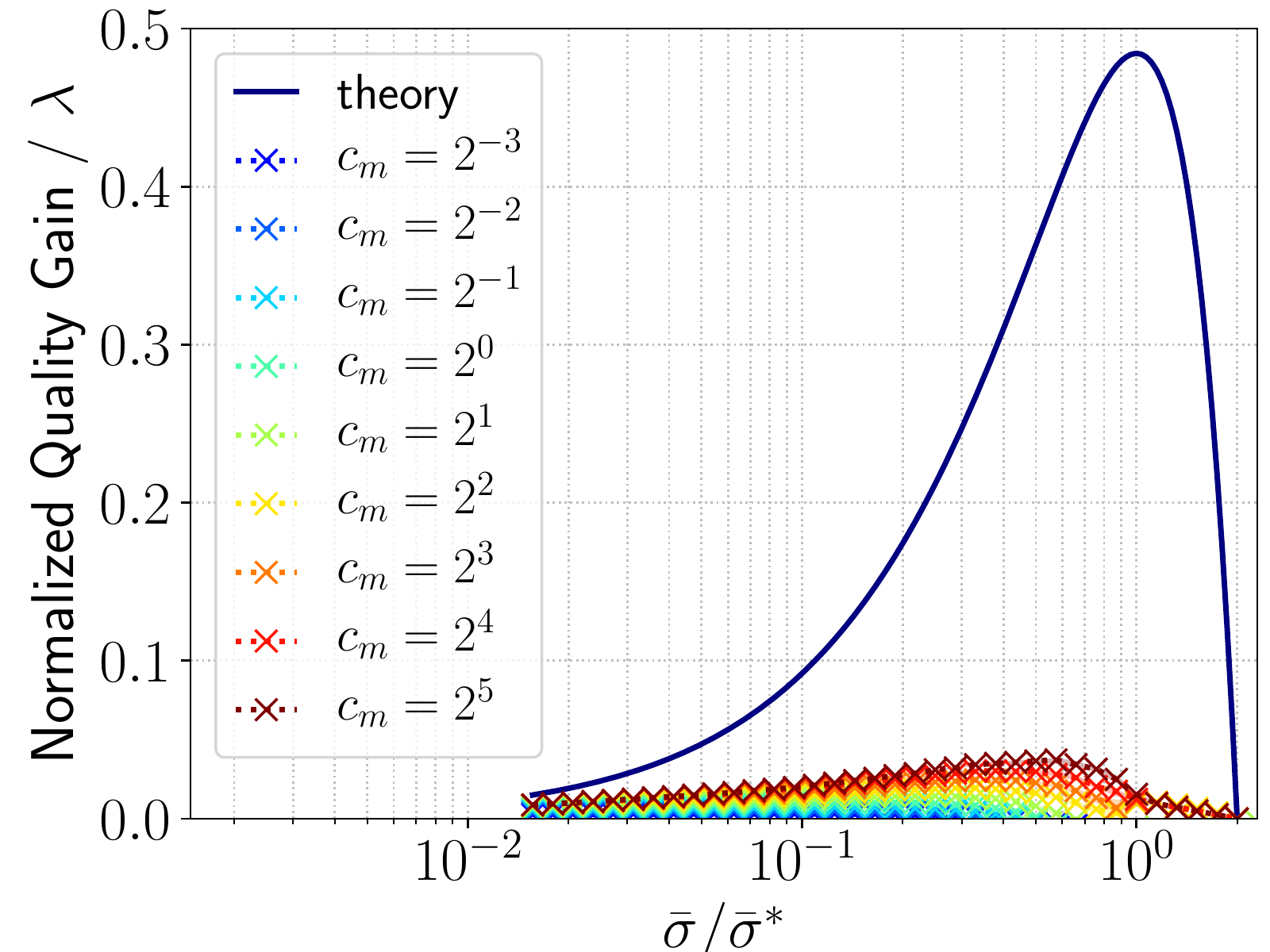}
\end{subfigure}%
\begin{subfigure}{\figsize}
\includegraphics[width=\hsize]{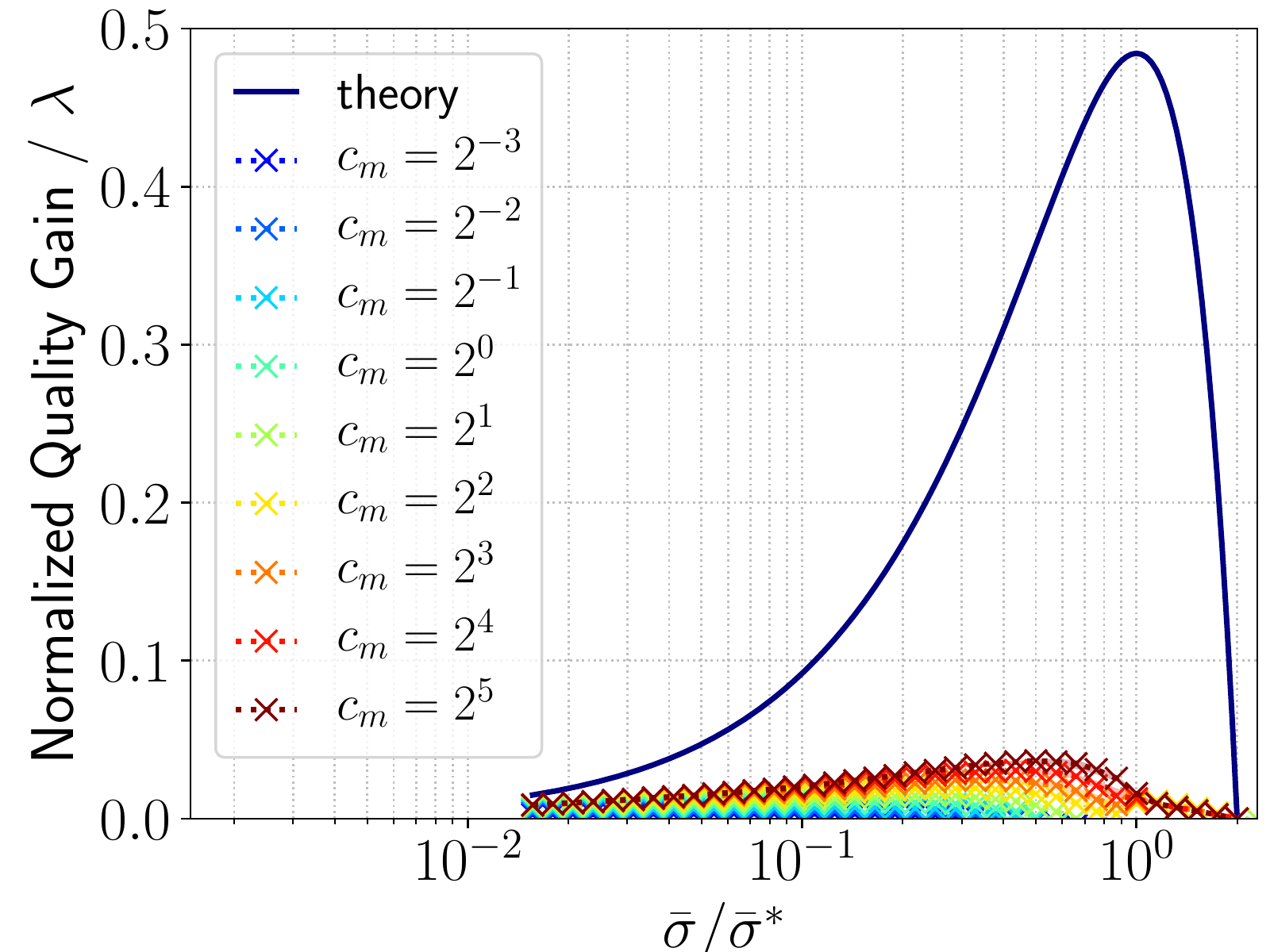}
\end{subfigure}%
\begin{subfigure}{\figsize}
\includegraphics[width=\hsize]{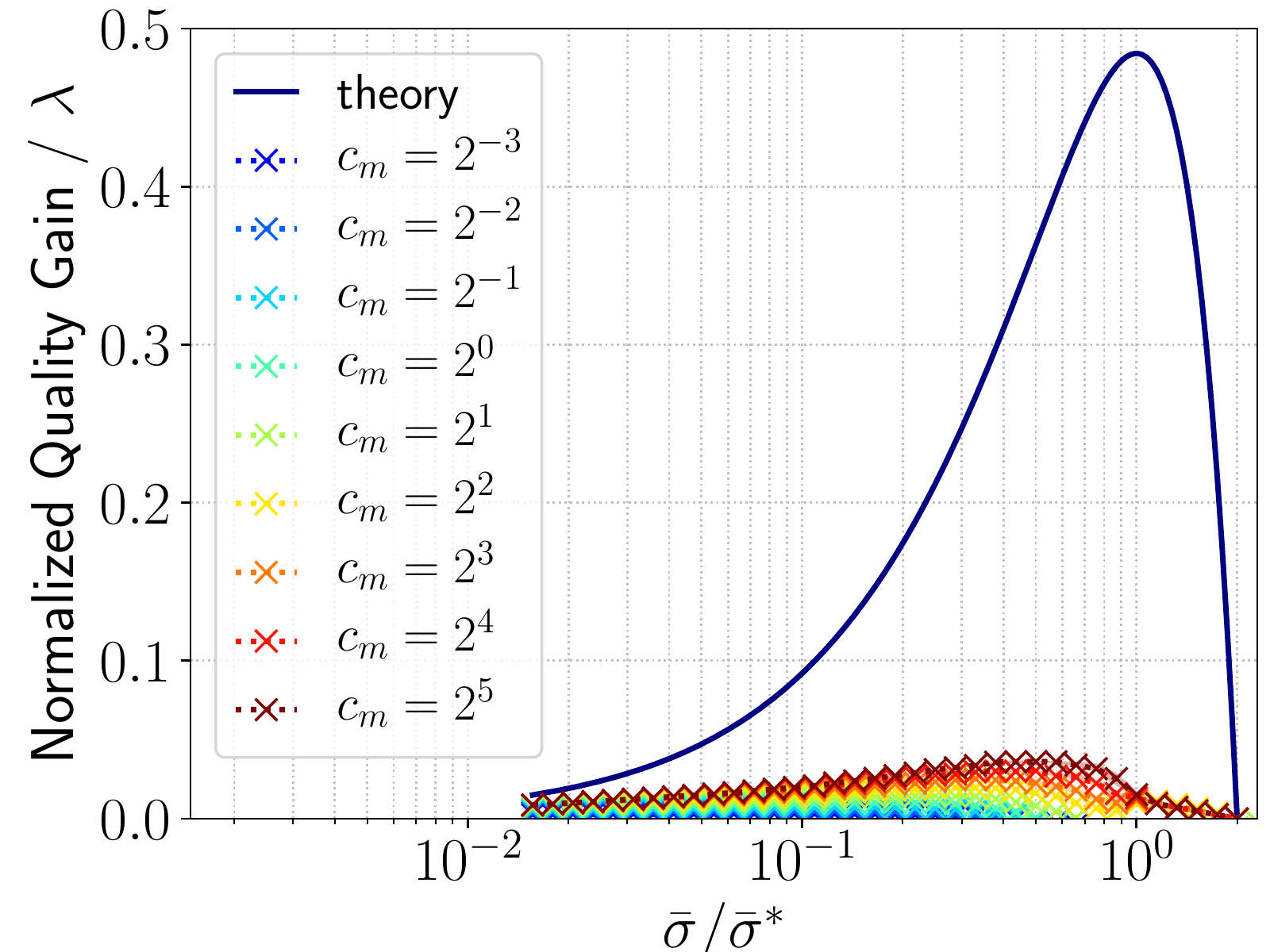}
\end{subfigure}%
\\
\begin{subfigure}{\figsize}
\includegraphics[width=\hsize]{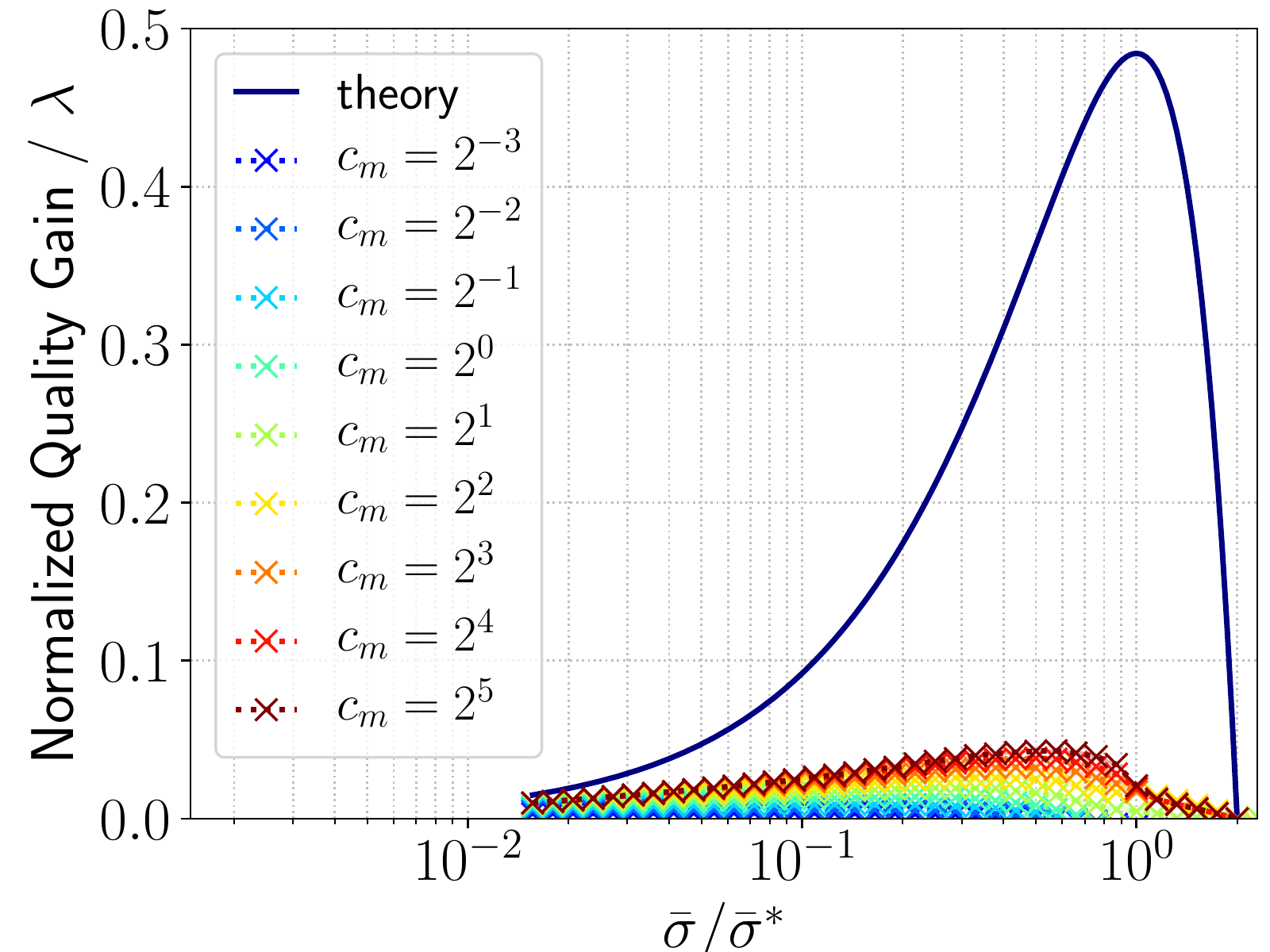}
\end{subfigure}%
\begin{subfigure}{\figsize}
\includegraphics[width=\hsize]{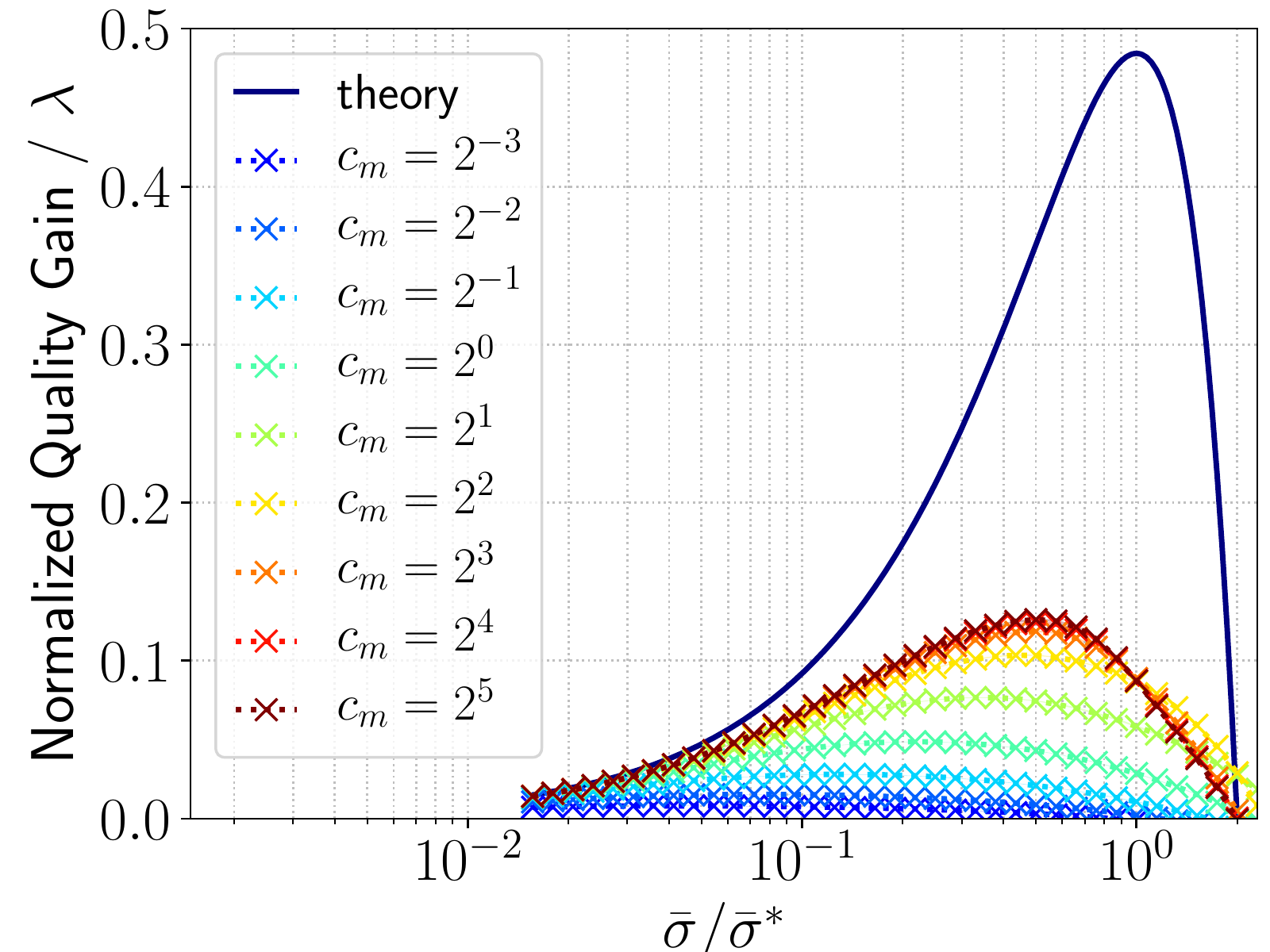}
\end{subfigure}%
\begin{subfigure}{\figsize}
\includegraphics[width=\hsize]{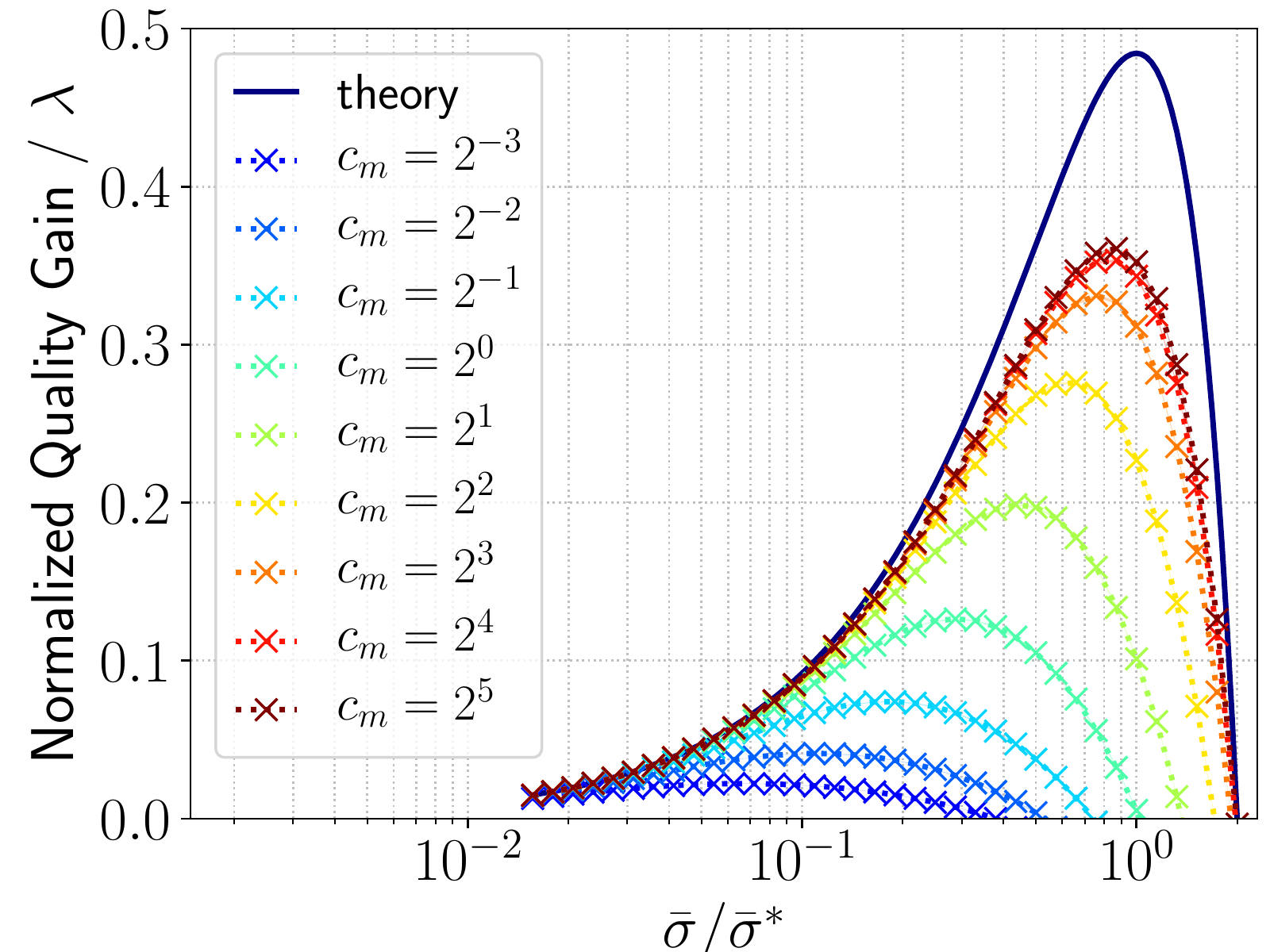}
\end{subfigure}%
\\
\begin{subfigure}{\figsize}
\includegraphics[width=\hsize]{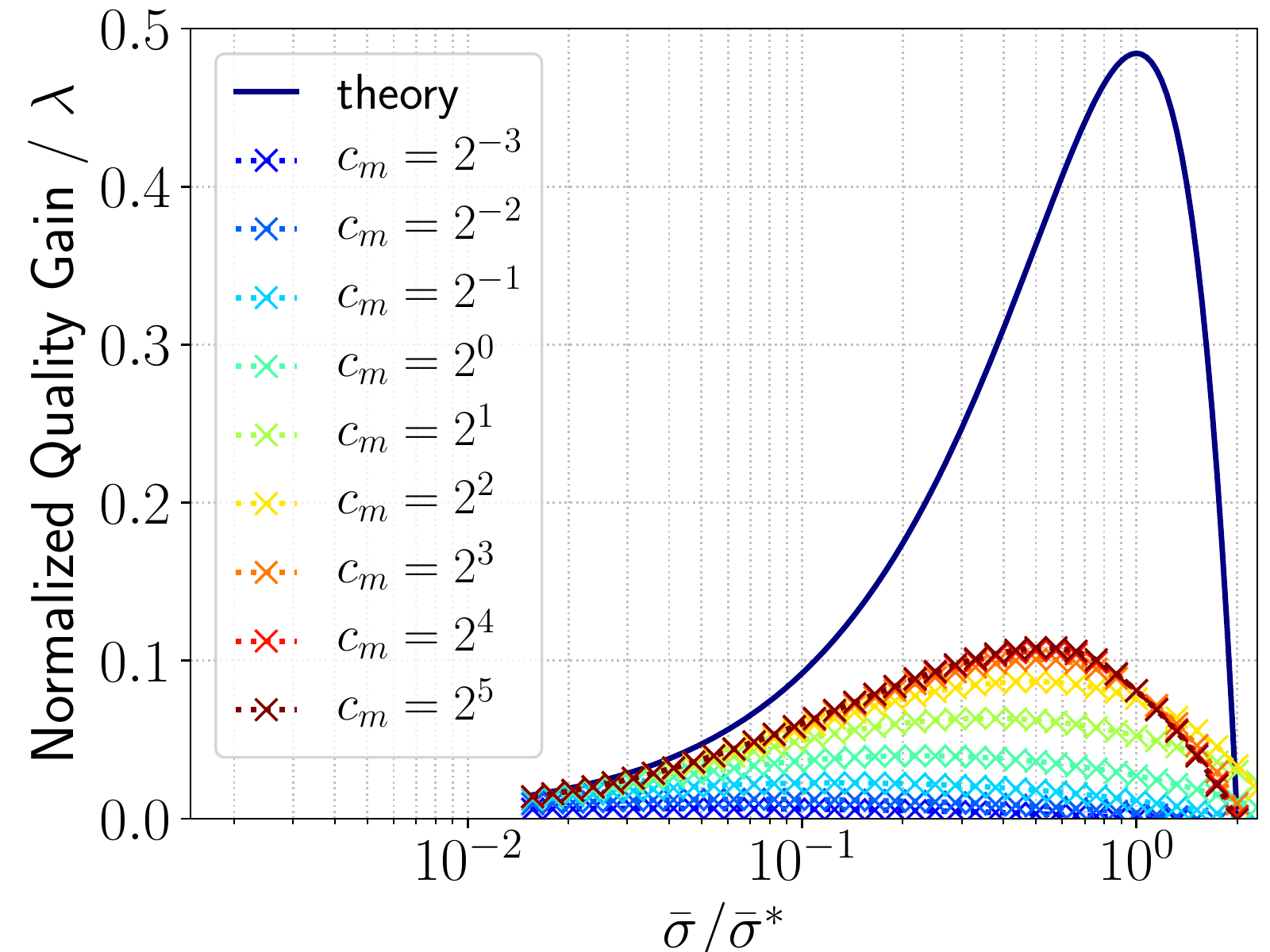}
\end{subfigure}%
\begin{subfigure}{\figsize}
\includegraphics[width=\hsize]{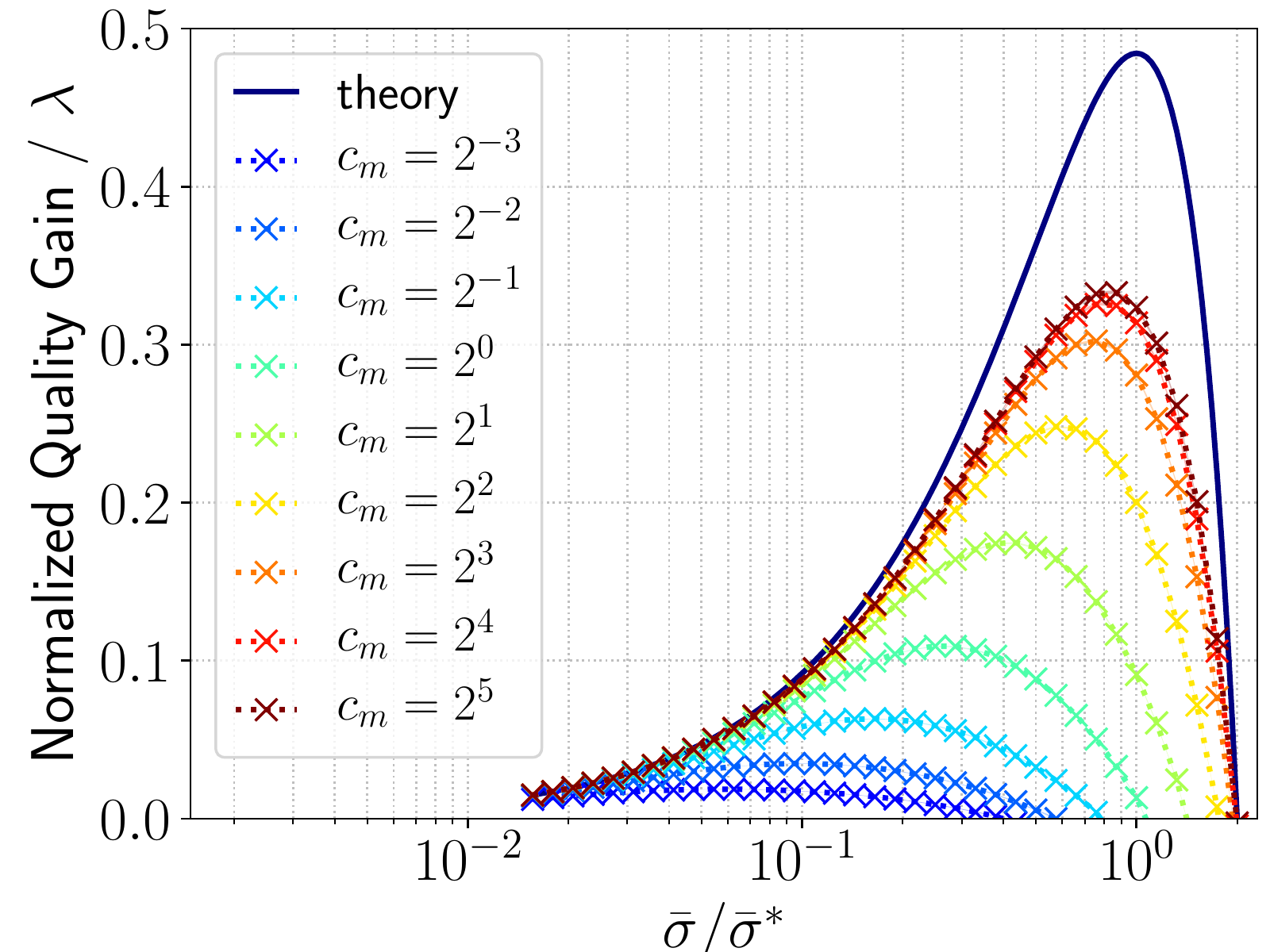}
\end{subfigure}%
\begin{subfigure}{\figsize}
\includegraphics[width=\hsize]{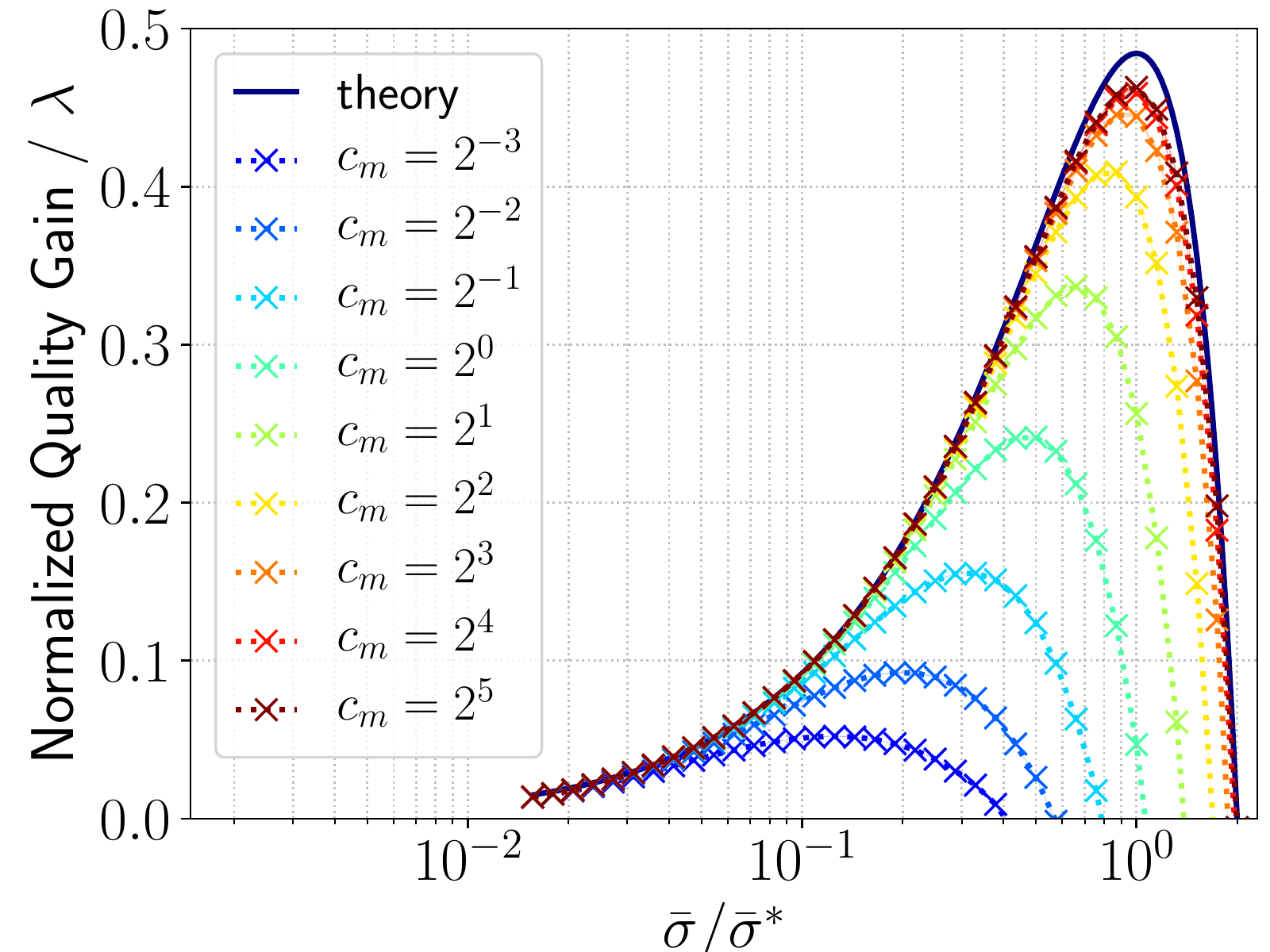}
\end{subfigure}%
\caption{Empirical normalized quality gain on four convex quadratic functions, Sphere, Discus, Ellipsoid and Cigar (from top to bottom) of dimension $N = 10$, $100$ and $1000$ (from left to right). The optimal weights \eqref{eq:optw} \rev{are} used and $\lambda = 100$.}
\label{fig:opt100}
\end{figure}

\subsection{Experiments}

To see the effect of the eigenvalue distribution of $\Hess$, we run the experiments. Four quadratic functions are considered: Sphere, Discus, Ellipsoid, Cigar functions of $N = 10$, $100$, $1000$ dimensions. The ES with the weights optimal for the infinite dimensional sphere, \eqref{eq:optw}, and the optimal normalized step-size $\ns^*$ derived for $\cm \to \infty$, \eqref{eq:optns-gen}, times a constant factor is run for $T = 10000$ iterations. The empirical normalized quality gain is estimated as $(2 / T) \sum_{t=T/2}^{T-1} \big[ f(\xmean^{(t)}) - f(\xmean^{(t+1)}) \big] / \big[ f(\xmean^{(t)}) g(\xmean^{(t)}) \big]$.
The mean vector is initialized randomly by the normal distribution $\mathcal{N}(\bm{0}, \eye)$. Eleven independent runs are conduced for each setting. The results are compared with $\asynqg$, which is supposed to approximate the empirical normalized quality gain for $\cm \gg 1$ and $N \gg 1$. Note that $\ns^*$ in \eqref{eq:optns-gen} and $\asynqg$ in \eqref{eq:asynqg} depend on $\xmean$ through $\ee^\T \Hess \ee  / \Tr(\Hess)$. We replace $\ee^\T \Hess \ee  / \Tr(\Hess)$ with $d_N(\Hess) / \Tr(\Hess)$ based on the observation and the above discussion that the mean vector tends to be parallel to the eigenspace corresponding to the smallest eigenvalue of $\Hess$.\niko{Do we have any idea how much of an error this may introduce?} Figures~\ref{fig:opt} and \ref{fig:opt100} show the median \rev{(marker)} and the $10\%$-$90\%$ interval \rev{(shaded area)} of the empirical normalized quality gain for each $\cm$ and the theoretically derived normalized quality gain formula discussed above. \rev{Note that the shaded area is almost invisible, implying that the number of runs and the number of iterations are sufficient to get accurate estimates\del{see the trends}.}

We first focus on the results with $\cm = 1$ (the default setting). The empirical normalized quality gain gets closer to the normalized quality gain derived for the infinite dimensional quadratic function as $N$ increases. The approach of the empirical normalized quality gain to the theory is the fastest for the sphere function ($\Hess = \eye$). For convex quadratic functions with the same condition number of $\alpha = 10^6$, the speed of the convergence of the normalized quality gain to $\asynqg$ as $N \to \infty$ is the fastest for the cigar function, and the slowest for the discus function. This reflects the upper bound derived in Theorem~\ref{thm:main} that depends on the ratio $\Tr(\Hess^2) / \Tr(\Hess)^2$, whose value is summarized in Table~\ref{tbl:hess}. For the cigar function $\Tr(\Hess^2) / \Tr(\Hess)^2$ is close to $1/(N-1)$, while for the discus function it is very close to $1$ for $N \ll \alpha$ and we do not observe significant difference between results on different $N$.

A larger $\cm$ led to a better empirical normalized quality gain for all cases, i.e., the empirical normalized quality gains became monotonically closer to the theoretical curve\footnote{Figure~4 in the previous work \cite{Akimoto2017foga} shows non-monotonic change of empirical normalized quality gain over $\cm$, whereas in Figures~\ref{fig:opt} and \ref{fig:opt100} of this paper shows a monotonic behavior. In Figure~4 in \cite{Akimoto2017foga} $\ns^*$ is approximated with \eqref{eq:optns}, whereas in the figures of this paper $\ns^*$ is computed with \eqref{eq:optns-gen}. The difference between these two quantities is less pronounced as $N$ increases. The monotonic changes of the graphs \rev{are} because $\ns^*$ in \eqref{eq:optns-gen} approximates the optimal normalized step-size better than \eqref{eq:optns} on a finite dimensional quadratic function.}. As $\cm$ becomes greater while the normalized step-size is fixed, the ratio $\sigma / \norm{\xmean}$ becomes smaller and tends to zero in the limit $\cm \to \infty$. As Corollary~\ref{cor:sig} implies, the normalized quality gain converges to $\asynqg$ in the limit $\sigma / \norm{\xmean} \to 0$. Therefore, the results reflect the theory. Moreover, the theoretically optimal normalized step-size $\ns^*$ well approximate\rev{s} the empirically optimal normalized step-size $\ns$ that maximize the normalized quality gain for all cases when $\cm \geq 1$. As $\cm$ becomes smaller, the empirically optimal normalized step-size $\ns$ becomes smaller compared to $\ns^*$. \new{Note that the difference of the empirical normalized quality gain curves on the sphere function comes only from the randomness of the length of each step $Z$. If we replace $Z$ with $(\E[\norm{Z}]/\norm{Z}) Z$ in the algorithm, the selection is independent of $\cm$ values and is determined by the inner product of the step and the gradient of the objective function at the mean vector. Then, the effect of $\cm$ goes away. }

Comparing Figure~\ref{fig:opt} and Figure~\ref{fig:opt100}, the empirical curves are closer to the theoretical curves in Figure~\ref{fig:opt}. It reflects the fact that the bound between the normalized quality gain and the asymptotic normalized quality gain derived in Theorem~\ref{thm:main} typically increases as $\lambda$ increases. To approximate the theoretical curve, a larger $\cm$ value is required when $\lambda$ is greater. The peaks of the empirical curves tend to be achieved at a smaller normalized step-size as \rev{$\lambda$ or $\cm$ becomes greater or smaller, respectively}.

\section{Conclusion}\label{sec:conc}

We perform the quality gain analysis of the weighted recombination evolution strategy (ES) on a convex quadratic function. Differently from the previous works, where the limit for the search space dimension $N$ to infinity is considered, we derive the error bound between the so-called normalized quality gain and its limit expression for the finite dimension. We show that the bound converges to zero when (I) $N \to \infty$ as long as the Hessian $\Hess$ of the objective function satisfies $\Tr(\Hess^2) / \Tr(\Hess)^2 \to 0$, \new{or}\del{and} when (II) $\sigma / \norm{\xmean} \to 0$. The limit expression of the normalized quality gain reveals that the optimal recombination weights are independent of the Hessian matrix in the limit (I). Moreover, if the effective variance selection mass $\muw$ is sufficiently large, \todo{this sentence should be revised}the optimal recombination weights for the limit (II) admits the same optimal recombination weights. The optimal normalized step-size for given recombination weights is derived. In the limit (I) the optimal normalized step-size is independent of $\Hess$, while the optimal step-size is proportional to the length of the gradient at the distribution mean. The limit (II) reveals the dependencies of the normalized step-size on $N$ and $\muw$.

The quality gain analysis provides a useful insight into the algorithmic behavior, even though it does not take into account the adaptation of the step-size. Knowing the optimal recombination weights $(w_k^*)$ directly contributes to the optimal parameter setting. \new{On the contrary}\del{Compared to it}, knowing the optimal normalized step-size $\ns^*$ does not lead to the optimal step-size control. \new{This}\del{It} is because the optimal scale-invariant step-size $\sigma^*$ in Definition~\ref{def:ns} where $\ns$ is replaced with its optimal value $\ns^*$ is proportional to $\norm{\nabla f(\xmean)}$, which is unknown to the algorithm. The optimal step-size, however, is useful to evaluate step-size control mechanisms and to see how close to the optimal situation the step-size control mechanism is. \rev{Some theoretical insights into the adaptation mechanism of practical step-size adaptive methods is provided by the approached referred to as ``dynamical system" approach by its authors.}\del{The dynamical system approach further provides a theoretical insight into the adaptation mechanism of the practical step-size adaptation.} We refer to \cite{beyer2014dynamics,Beyer2016ec} for the recent development in the dynamical system approach. An important remaining question is: what is the optimal parameter update? \new{Neither} the quality gain analysis \new{nor}\del{and} the dynamical system approach will\del{ not} answer this question. The optimal step-size on a quadratic function is revealed in this paper, however, it depends \rev{on} the norm of the gradient, which is unknown to the real algorithm. A methodology to analyze the optimal update, rather than the optimal parameter value, hopefully including the covariance matrix update is desired in\del{ the} future work.

\paragraph*{\bf Acknowledgement}
The authors thank \emph{Dagstuhl seminar 17191: Theory of Randomized Optimization Heuristics} for providing the opportunity to present and discuss this work. This work is partially supported by JSPS KAKENHI Grant Number 15K16063.

\section*{References}
%

\appendix

\section{Normal Order Statistics}
\label{apdx:nos}
Here we summarize some important properties of the moments of normal order statistics that are useful to understand the results in the paper. 

The first moments of the normal order statistics have the properties: $\E[\NN_{i:\lambda}] \leq \E[\NN_{i+1:\lambda}]$, $\E[\NN_{i:\lambda}] = - \E[\NN_{\lambda + 1 - i:\lambda}]$, and $\sum_{i=1}^{\lambda} \E[\NN_{i:\lambda}] = 0$. The second (product) moments of the normal order statistics have the following properties: $\sum_{j=1}^{\lambda} \E[\NN_{i:\lambda}\NN_{j:\lambda}] = 1$, $\sum_{i=1}^{\lambda} \E[\NN_{i:\lambda}^2] = \sum_{i=1}^{\lambda}\sum_{j=1}^{\lambda} \E[\NN_{i:\lambda}\NN_{j:\lambda}] = \lambda$, and $\E[\NN_{i:\lambda}\NN_{j:\lambda}] = \E[\NN_{j:\lambda}\NN_{i:\lambda}] = \E[\NN_{\lambda+1-i:\lambda}\NN_{\lambda+1-j:\lambda}] = \E[\NN_{\lambda+1-j:\lambda}\NN_{\lambda+1-i:\lambda}]$.

Here we summarize useful inequalities about order statistics that are all listed in Section~35.1.6 of \cite{dasgupta2008asymptotic}. The positive dependency inequality tells that the order statistics are non-negatively correlated, $\Cov(\NN_{i:\lambda}, \NN_{j:\lambda}) = \E[\NN_{i:\lambda}\NN_{j:\lambda}] - \E[\NN_{i:\lambda}]\E[\NN_{j:\lambda}] \geq 0$. Together with $\sum_{j=1}^{\lambda} \Cov(\NN_{i:\lambda}, \NN_{j:\lambda}) = \sum_{j=1}^{\lambda} \E[\NN_{i:\lambda}\NN_{j:\lambda}] = 1$, we have $0 \leq \Cov(\NN_{i:\lambda}, \NN_{j:\lambda}) \leq 1$. It implies $\E[\NN_{i:\lambda}]\E[\NN_{j:\lambda}] \leq \E[\NN_{i:\lambda}\NN_{j:\lambda}] \leq \E[\NN_{i:\lambda}]\E[\NN_{j:\lambda}] + 1$.

Another important inequality is David inequality for normal distribution. It tells that $\Phi^{-1}\big(i / (\lambda + 1)\big) \leq \E[\NN_{i:\lambda}] \leq \min\big\{\Phi^{-1}\big(i / (\lambda + 0.5)\big) ,\ \Phi^{-1}\big((i-0.5) / \lambda\big) \big\}$, where $\Phi$ is the c.d.f.~of $\mathcal{N}(0, 1)$. 
It proves an asymptotically tight approximation (Blom's approximation) $\E[\mathcal{N}_{i:\lambda}] \approx \Phi^{-1}\big(\frac{i - \alpha}{\lambda - 2 \alpha + 1}\big)$ with $\alpha = 0.375$ for $i \leq \lceil \lambda / 2 \rceil$.
The following asymptotic equalities are also used (see Example 8.1.1 in \cite{dasgupta2008asymptotic})%
{\small\begin{equation}%
\lim_{\lambda \to \infty} \frac{\E[\NN_{\lambda:\lambda}] - \E[\NN_{1:\lambda}]}{2 (2 \ln(\lambda))^{\frac12}} = 1,\quad
\lim_{\lambda\to\infty}\frac{1}{\lambda}\sum_{i=1}^{\lambda} \abs{\E[\NN_{i:\lambda}]} = \frac{2^{\frac12} }{\pi^{\frac12} },\quad 
\lim_{\lambda\to\infty}\frac{1}{\lambda}\sum_{i=1}^{\lambda} \E[\NN_{i:\lambda}]^2 = 1 \enspace.
\label{eq:noslim}
\end{equation}}%

Let $\eone_{(\lambda)}$ be the $\lambda$ dimensional column vector whose $i$-th component is $\E[\NN_{i:\lambda}]$ and $\etwo_{(\lambda)}$ be the $\lambda$ dimensional symmetric matrix whose $(i, j)$-th element is $\E[\NN_{i:\lambda} \NN_{j:\lambda}]$. The covariance matrix $\etwo_{(\lambda)} - \eone_{(\lambda)}\eone_{(\lambda)}^\T$ is by definition nonnegative definite. It implies the eigenvalues of $\etwo_{(\lambda)}$ are all nonnegative. Moreover, from the above mentioned fact derives that the sum of the eigenvalues is $\Tr(\etwo_{(\lambda)}) = \sum_{i=1}^{\lambda}\sum_{j=1}^{\lambda} \Cov(\NN_{i:\lambda}, \NN_{j:\lambda}) = \lambda$. Furthermore, the third asymptotic relation of \eqref{eq:noslim} reads $\lim_{\lambda \to \infty} \Tr(\eone_{(\lambda)}\eone_{(\lambda)}^\T)/\lambda = \lim_{\lambda \to \infty} \norm{\eone_{(\lambda)}}^2/\lambda = 1$. It implies, for any $\bm{x} \in \R^{\lambda}\setminus\{\bm{0}\}$, we have
{\small\begin{equation}
  \lim_{\lambda \to \infty} \frac{\bm{x}^\T \etwo_{(\lambda)}\bm{x}}{\lambda \norm{\bm{x}}^2} = \lim_{\lambda \to \infty} \frac{(\bm{x}^\T \eone_{(\lambda)})^2}{\lambda \norm{\bm{x}}^2} = \lim_{\lambda \to \infty} \frac{ (\bm{x}^\T\eone_{(\lambda)})^2 }{ \norm{\bm{x}}^2 \norm{\eone_{(\lambda)}}^2 }\enspace.
  \label{eq:etwoapprox}
\end{equation}}%

\section{Proofs and Derivations}
\label{apdx:proofs}

\subsection{Proof of Proposition~\ref{prop:optsig}}
\label{apdx:prop:optsig}
\begin{proof}
\newcommand{\vv}{\Delta}
Let $\vv = \cm \sum_{i=1}^{\lambda} W(i; (\xmean + \sigma Z_k)_{k=1}^{\lambda}) Z_{i}$, where $(Z_i)_{i=1}^{\lambda}$ are independent and $N$-variate standard normally distributed random vectors. Then,
{\small\begin{align*}
  \phi(\xmean, \sigma)
  &= 1 - \E[f^*(\xmean + \sigma \vv)] / f^*(\xmean) \\
  &= 1 - \E[f(\xmean + \sigma \vv  - x^*)] / f(\xmean - x^*) 
  \\
  &= 1 - \alpha^{-n} \E[f(\alpha\cdot(\xmean + \sigma \vv  - x^*)) / \alpha^{-n}f(\alpha\cdot (\xmean  - x^*)) 
  \\
  &= 1 - \E[f(\alpha\cdot(\xmean + \sigma \vv  - x^*)) ] / f(\alpha\cdot (\xmean  - x^*)) 
  \\
  &= 1 - \E[f^*(x^* + \alpha\cdot(\xmean  - x^*) + \alpha \sigma \vv ) / f^*(x^* + \alpha\cdot (\xmean  - x^*)) 
  \\
  &= \phi(x^* + \alpha (\xmean - x^*), \alpha \sigma) 
\enspace.
\end{align*}}%
Note that $\phi(x^* + (\xmean - x^*), \sigma) = \phi(\xmean, \sigma)$. That is, the quality gain is scale invariant around $(x^*, 0)$. Moreover, the above equality implies that $\argmax_{\sigma} \phi(x^* + (\xmean - x^*), \sigma) =  \argmax_{\sigma} \phi(x^* + \alpha (\xmean - x^*), \alpha \sigma)$, i.e., the optimal step-size at $x^* + \alpha (\xmean - x^*)$ is $\alpha$ times greater than the optimal step-size at $x^* + (\xmean - x^*)$. Therefore, the optimal step-size as a function of $\xmean - x^*$ is homogeneous of degree $1$, i.e., $\sigma^*(\alpha\cdot(\xmean - x^*)) = \alpha \sigma^*(\xmean - x^*)$. 
\end{proof}

\subsection{Proof of Lemma~\ref{lem:0}}
\label{apdx:lem:0}
\begin{proof}
Since $(X_{k})_{k=1}^{\lambda}$ are independent and normally distributed, the \rev{conditional probability of $\ind{f(X_{k}) < f(X_{i})} = 1$ given $X_{i}$ for any $k \neq i$} is $F_{f}(f(X_{i}))$. Then, the probability of $\sum_{k=1}^{\lambda} \ind{f(X_{k}) \leq f(X_{i})}$ being $a$ for $a \in \llbracket 1, \lambda \rrbracket$ is given by $\pb(a-1; \lambda-1, p)$
with $p = F_{f}(f(X_i))$. 
Then, for any $\alpha \geq 0$,
{\small\begin{align*}
\E_{i}[W(i; (X_k)_{k=1}^{\lambda})^{\alpha}] 
\textstyle= \sum_{k = 1}^{\lambda} w_{k}^{\alpha} \pb(k-1; \lambda-1, p)
\enspace.
\end{align*}}%

Similarly, the joint \rev{distribution} of $\sum_{k=1}^{\lambda} \ind{f(X_{k}) \leq f(X_{i})}$ and $\sum_{k=1}^{\lambda} \ind{f(X_{k}) \leq f(X_{j})}$ is derived. Due to the symmetry between $i$ and $j$, we can assume w.l.o.g.\ that $f(X_i) \leq f(X_j)$. Then, the joint probability of $\sum_{k=1}^{\lambda} \ind{f(X_{k}) \leq f(X_{i})} = a$ and $\sum_{k=1}^{\lambda} \ind{f(X_{k}) \leq f(X_{j})} = b$ for $a, b \in \llbracket 1, \lambda \rrbracket$ is given by $\pt(a-1, b-a-1; \lambda-2, p, q-p)$
with $p = F_{f}(f(X_i))$ and $q = F_{f}(f(X_j))$ if $a < b$, and zero otherwise. 
Then,
{\small\begin{equation*}
\E_{i, j}[W(i; (X_k)_{k=1}^{\lambda})W(j; (X_k)_{k=1}^{\lambda}) ] \\
= \sum_{m = 1}^{\lambda-1}\sum_{l = m+1}^{\lambda}w_{m}w_{l} \pt(m-1, l-m-1; \lambda-2, p, q-p)
\enspace.
\end{equation*}}%
This ends the proof.
\end{proof}

\subsection{Proof of Lemma~\ref{lem:u1lip}}
\label{apdx:lem:u1lip}
\begin{proof}
The derivative of $u_1$ is $\sum_{k=1}^{\lambda} w_{k} \binom{\lambda-1}{k-1} \frac{\rmd }{\rmd p}[ p^{k-1}(1-p)^{\lambda-k} ]$, where
{\small\begin{equation*}
\textstyle \frac{\rmd }{\rmd p}[ p^{k-1}(1-p)^{\lambda-k} ] =
   (k-1) p^{k-2} (1 - p)^{\lambda - k} - (\lambda - k) p^{k-1} (1 - p)^{\lambda - k - 1}
\enspace.
\end{equation*}}%
Substituting the derivatives and rearranging the terms, we obtain
{\small\begin{align*}
         \textstyle \frac{\rmd u_1(p)}{\rmd p}
 &\textstyle= (\lambda - 1) \sum_{k = 1}^{\lambda-1} (w_{k+1} - w_{k}) \binom{\lambda - 2}{k - 1} p^{k - 1}(1 - p)^{\lambda - k - 1} 
 \enspace. 
\end{align*}}%
The Lipschitz constant $L_1$ is the supremum of the absolute value of the derivative derived above. It completes the proof for the $\ell_1$-Lipschitz continuity of $u_1$ and its Lipschitz constant.
Since $u_2$ is equivalent to $u_1$ if $w_i$ are replaced with $w_i^2$ in the definition of $u_1$, we have the $\ell_1$-Lipschitz continuity of $u_2$ and its Lipschitz constant by replacing $w_i$ with $w_i^2$ in the above argument.

The partial derivative of $u_3$ with respect to $p$ is
{\small\begin{equation*}
  \textstyle  \sum_{k=1}^{\lambda-1}\sum_{l=k+1}^{\lambda} w_{k} w_{l} \binom{\lambda-2}{l-2}\binom{l-2}{k-1} \frac{\partial }{\partial p}[ \min(p, q)^{\rev{k-1}}\abs{q - p}^{l - k - 1}(1-\min(p, q))^{\lambda-l} ] \enspace,
\end{equation*}}
where
{\small\begin{multline*}
\frac{\partial }{\partial p} \min(p, q)^{k-1}\abs{q - p}^{l - k - 1}(1-\min(p, q))^{\lambda-l}  \\
=
  \begin{cases}
	   [(k-1) (q - p)  - (l - k - 1) p  ] p^{k - 2}(q - p)^{l - k - 2}(1-q)^{\lambda-l} & (p < q)\\
	   [(l - k - 1)(1 - p)  - (\lambda - l) (p - q)  ] q^{k - 1}(p - q)^{l - k - 2}(1-p)^{\lambda-l-1} & (p > q)\\
  \end{cases}
\enspace.
\end{multline*}}%
Substituting the derivatives and rearranging the terms, we obtain
{\small\begin{multline*}
  \textstyle \rev{\frac{1}{\lambda - 2}} \frac{\partial u_3(p, q)}{\partial p} \\
  = 
  \begin{cases}
	   \sum_{k=1}^{\lambda-2}\sum_{l=k+2}^{\lambda} w_l (w_{k+1} - w_{k}) \rev{\binom{\lambda-3}{l-3}\binom{l-3}{k-1}} p^{k - 1}(q - p)^{l - k - 2}(1-q)^{\lambda-l} & (p < q) \\
	   \sum_{k=1}^{\lambda-2}\sum_{l=k+2}^{\lambda} w_k (w_{l} - w_{l-1}) \rev{\binom{\lambda-3}{l-3}\binom{l-3}{k-1}} q^{k - 1}(p - q)^{l - k - 2}(1-p)^{\lambda-l} & (p > q) \\
  \end{cases}
\end{multline*}}%
Since $u_3$ is differentiable with respect to $p$ almost everywhere in $(0, 1)$, it is Lipschitz continuous with respect to $p$. Its Lipschitz constant is $\sup_{q \in (0, 1)}\sup_{p \in (0, q) \cup (q, 1)}\abs[\big]{\frac{\partial u_3(p, q)}{\partial p}}$. Due to the symmetry, $u_{3}(p, q)$ is $\ell_1$-Lipschitz continuous on $[0, 1]^2$ with the Lipschitz constant $L_3 = \sup_{q \in (0, 1)}\sup_{p \in (0, q) \cup (q, 1)}\abs[\big]{\frac{\partial u_3(p, q)}{\partial p}}$. This completes the proof.
\end{proof}

\subsection{Upper bounds of Lipschitz constants}
\label{apdx:Lbound:deriv}

For a general weight scheme, we have the following trivial upper bounds for the Lipschitz constants derived in Lemma~\ref{lem:u1lip}, 
{\small\begin{align}
L_1 &\leq (\lambda - 1)\max_{k \in \llbracket 1, \lambda - 1\rrbracket} \abs{w_{k+1} - w_k}
\enspace,
\label{eq:l1bound}
\\
L_2 &\leq (\lambda - 1)\max_{k \in \llbracket 1, \lambda - 1\rrbracket} \abs{w_{k+1}^2 - w_k^2}
\enspace,
\label{eq:l2bound}
\\
L_3 &\leq (\lambda - 2)\max_{k \in \llbracket 1, \lambda\rrbracket} \max_{l \in \llbracket 1, k - 2\rrbracket \cup \llbracket k +1, \lambda - 1\rrbracket} \abs{w_k}\cdot\abs{w_{l+1} - w_l} 
\enspace.
\label{eq:l3bound}
\end{align}}%
These upper bounds are straight-forward from the facts $\sum_{k=1}^{\lambda-1}\pb(k-1; \lambda-2, p) = 1$ and $\sum_{k=1}^{\lambda-2} \sum_{l=k+2}^{\lambda} \pt(k-1, l-k-2; \lambda-3, \min(p, q), \abs{q - p}) = 1$.

For the truncation weights with $3 \leq \mu \leq \lambda - 2$, we can obtain better bounds.
The bounds of the factorial of $n \geq 1$ known by Robbins \cite{Robbins1955}, namely,
{\small\begin{equation*}
(2 \pi n)^\frac12 \left(\frac{n}{e}\right)^n \exp\left(\frac{1}{12n + 1}\right)
< n! <
(2 \pi n)^\frac12 \left(\frac{n}{e}\right)^n \exp\left(\frac{1}{12n} \right)
\end{equation*}}%
gives us an upper bound of $\binom{n}{k}$ for $0< k < n$
{\small\begin{equation}
\textstyle   \binom{n}{k} < \left( \frac{n}{2\pi k (n - k)}\right)^\frac12 \left(\frac{n}{k}\right)^k\left(\frac{n}{n - k}\right)^{n - k} \enspace.
    \label{eq:binomupper}
  \end{equation}}%
\rev{Here we used $ \exp\big(\frac{1}{12n} - \frac{1}{12k + 1} - \frac{1}{12(n-k) + 1}\big) < 1$.}
On the other hand, we have for $0< k < n$
{\small\begin{equation}
\textstyle  \sup_{0\leq p \leq 1} p^k(1 - p)^{n - k} = \left(\frac{k}{n}\right)^k\left(\frac{n - k}{n}\right)^{n - k}
    \label{eq:binomsup}    
\enspace.
\end{equation}}%
Since $w_{k+1} - w_k = - 1/\mu$ for $k = \mu$ and $w_{k+1} - w_k = 0$ for $k \neq \mu$, we have for $3 \leq \mu \leq \lambda - 2$, 
{\small\begin{align*}
L_1 &\textstyle= \sup_{0 < p < 1} \abs*{(\lambda - 1) \frac1\mu \binom{\lambda - 2}{\mu - 1} p^{\mu - 1}(1 - p)^{\lambda - \mu - 1}}
\notag\\
    &\textstyle= \frac{\lambda - 1}{\mu} \binom{\lambda - 2}{\mu - 1} \left(\frac{\mu - 1}{\lambda-2}\right)^{\mu - 1}\left(\frac{\lambda - \mu - 1}{\lambda - 2}\right)^{\lambda - \mu - 1}
\notag\\
    &\textstyle\leq \frac{\lambda - 1}{\mu} \left( \frac{\lambda - 2}{2\pi (\mu - 1) (\lambda - \mu - 1)}\right)^\frac12 
      \enspace.
       \end{align*}}%
Analogously, since $w_{k+1}^2 - w_k^2 = - 1/\mu^2$ for $k = \mu$ and $w_{k+1}^2 - w_k^2 = 0$ for $k \neq \mu$, we obtain the bound of $L_2$: $L_2 < \big[(\lambda - 1 ) / \mu^2\big]\cdot \big[ (\lambda - 2) / (2\pi (\mu - 1) (\lambda - \mu - 1))\big]^{1/2}$.
     
Moreover, since $w_k(w_l - w_{l-1}) = -1/\mu^2$ for $l = \mu+1$ and $w_k(w_l - w_{l-1}) = 0$ otherwise, we have
{\small\begin{align*}
L_3 
&\textstyle= \sup_{q \in (0, 1)} \sup_{p \in (q, 1)} \sum_{k=1}^{\mu-1}\frac{\lambda - 2}{\mu^2} 
\binom{\lambda - 3}{\mu-2}\binom{\mu-2}{k - 1}
 q^{k - 1}(p - q)^{\mu - k - 1}(1-p)^{\lambda-\mu-1}
 \notag
 \\
 &\textstyle= \sup_{p \in (0, 1)} \frac{(\lambda - 2)}{\mu^2} 
\binom{\lambda - 3}{\mu-2}(1-p)^{\lambda-\mu-1} \sup_{q \in (0, p)} \sum_{k=1}^{\mu-1}\binom{\mu-2}{k - 1}
 q^{k - 1}(p - q)^{\mu - k - 1}
 \notag
  \\
 &\textstyle=\frac{(\lambda - 2)}{\mu^2} 
 \binom{\lambda - 3}{\mu-2}\sup_{p \in (0, 1)} (1-p)^{\lambda-\mu-1}p^{\mu - 2}
 \notag
 \\
 &\textstyle=\frac{(\lambda - 2)}{\mu^2} 
 \binom{\lambda - 3}{\mu-2}
 \left(\frac{\mu - 2}{\lambda - 3}\right)^{\mu - 2}\left(\frac{\lambda-\mu-1}{\lambda - 3}\right)^{\lambda-\mu-1}
 \notag
 \\
 &\textstyle\leq\frac{(\lambda - 2)}{\mu^2} 
 \left(\frac{\lambda - 3}{2 \pi (\mu - 2)(\lambda - \mu - 1)}\right)^\frac{1}{2}
\enspace.
       \end{align*}}%
Here we used \eqref{eq:binomupper}, \eqref{eq:binomsup}, and the binomial relation $\sum_{k=1}^{\mu-1} \binom{\mu-2}{k - 1}
 q^{k - 1}(p - q)^{\mu - k - 1} = p^{\mu - 2}$.

\subsection{Proof of Lemma~\ref{lem:weak}}
\label{apdx:lem:weak}
\begin{proof}
If $\alpha = 1$, then $G(\alpha) = 1$, and the inequality is trivial. Hence, we assume $\alpha < 1$ in the following. 

Remember that $H_N = \Ze + h(Z)$. The absolute difference between $F_{N}(t)$ and $\Phi(t)$ is rewritten as follows
{\small\begin{align*}
\MoveEqLeft[1] \abs{F_{N}(t) - \Phi(t)} \\
&= \abs{\Pr[H_N  \leq t] - \Pr[\Ze \leq t]}
\\
&= \abs{\Pr[\Ze + h(Z) \leq t ] - \Pr[\Ze \leq t]}
\\
&= \Pr[h(Z) \geq 0 ~\text{and}~ t - h(Z) \leq \Ze \leq t ] + \Pr[h(Z) \leq 0 ~\text{and}~ t \leq \Ze \leq t - h(Z)]
  \enspace.
\end{align*}}%
With an arbitrary $\epsilon_+ > 0$, the first term on the RMS is upper bounded as
{\small\begin{align*}
\MoveEqLeft[1]\Pr[h(Z) \geq 0 ~\text{and}~ t - h(Z) \leq \Ze \leq t ] \\
&\leq \Pr[h(Z) \geq \epsilon_+] + \Pr[h(Z) < \epsilon_+ ~\text{and}~ t - h(Z) \leq \Ze \leq t ] 
\\
&\leq \Pr[h(Z) \geq \epsilon_+] + \Pr[h(Z) < \epsilon_+ ~\text{and}~ t - \epsilon_+ \leq \Ze \leq t ] 
\\
&\leq \Pr[h(Z) \geq \epsilon_+] + \Pr[t - \epsilon_+ \leq \Ze \leq t ] 
\\
&\leq \Pr[h(Z) \geq \epsilon_+] + (2\pi)^{-\frac12}\epsilon_+
  \enspace.  
\end{align*}}%
For the last inequality, we used that the density of the one-dimensional standard normal distribution is at most $(2\pi)^{-\frac12}$ and $\Ze$ is of the one-dimensional standard normal distribution. Analogously, we have for any $\epsilon_- > 0$
{\small\begin{align*}
\Pr[h(Z) \leq 0 ~\text{and}~ t \leq \Ze \leq t - h(Z)] 
&\leq \Pr[h(Z) \leq - \epsilon_-] + \Pr[t \leq \Ze \leq t + \epsilon_-] 
\\
&\leq \Pr[h(Z) \leq - \epsilon_+] + (2\pi)^{-\frac12}\epsilon_-
  \enspace.  
\end{align*}}%
Let $\tilde{h}(Z) = 2 (\cm / \ns) h(Z) = Z^\T \Hesst Z - 1$, $\tilde{\epsilon}_+ = 2 (\cm / \ns) \epsilon_+$ and $\tilde{\epsilon} _- = 2 (\cm / \ns) \epsilon_-$. Then, $\Pr[h(Z) \geq \epsilon_+] = \Pr[\tilde{h}(Z) \geq \tilde{\epsilon}_+]$ and $\Pr[h(Z) \leq - \epsilon_-] = \Pr[\tilde{h}(Z) \leq - \tilde{\epsilon}_-]$. From Lemma~1 in \cite{laurent2000} knows that for any $x \geq 0$
{\small\begin{equation}
    \begin{split}
\Pr\big[\tilde{h}(Z) \geq 2 \Tr(\Hesst^2)^\frac12 x^\frac12 + 2 \eig_1(\Hesst) x\big] &\leq \exp(-x) \enspace,\\
\Pr\big[\tilde{h}(Z) \leq - 2 \Tr(\Hesst^2)^\frac12 x^\frac12\big] &\leq \exp(-x) \enspace.
\end{split}
\label{eq:laurent}
\end{equation}}%
Let $x = \ln(1/\alpha)$ and let $\epsilon_+$ and $\epsilon_-$ such that
{\small\begin{align*}
\tilde{\epsilon}_{+} &= 2 \Tr(\Hesst^2)^\frac12 x^\frac12 + 2 \eig_1(\Hesst) x = 2 \Tr(\Hesst^2)^\frac12 \big((\ln(1/\alpha))^\frac12 + (\eig_1(\Hesst) / \Tr(\Hesst^2)^\frac12) \ln(1/\alpha)\big) 
\\
\tilde{\epsilon}_{-} &= 2 \Tr(\Hesst^2)^\frac12 x^\frac12 = 2 \Tr(\Hesst^2)^\frac12 (\ln(1/\alpha))^\frac12
\enspace.
\end{align*}}%
Then, from \eqref{eq:laurent} derives that
{\small\begin{align*}
\MoveEqLeft[2] \textstyle \Pr[h(Z) \geq \epsilon_+] + (2\pi)^{-\frac12}\epsilon_+
         \\
& \textstyle = \Pr[\tilde{h}(Z) \geq \tilde{\epsilon}_+] + (\ns \tilde{\epsilon}_+) / (2(2\pi)^\frac12 \cm)
\\
& \textstyle \leq \alpha + (\ns \tilde{\epsilon}_+) / (2(2\pi)^\frac12 \cm) 
\\
& \textstyle = \alpha + (2\pi)^{-\frac12} (\ns / \cm) \Tr(\Hesst^2)^\frac12 \big((\ln(1/\alpha))^\frac12 + \big(\eig_1(\Hesst)/ \Tr(\Hesst^2)^\frac12\big)\ln(1/\alpha)\big)
\\
& \textstyle = \alpha \big(1 + (2\pi)^{-\frac12} (\ln(1/\alpha))^\frac12 + (2\pi)^{-\frac12}\big(\eig_1(\Hesst)/ \Tr(\Hesst^2)^\frac12\big)\ln(1/\alpha) \big)
  \enspace.  
\end{align*}}%
Similarly, we have
$\Pr[h(Z) \leq - \epsilon_-] + (2\pi)^{-\frac12}\epsilon_- 
\leq \alpha \big(1 + (2\pi)^{-\frac12}(\ln(1/\alpha))^\frac12 \big)$.
Altogether, we obtain
{\small\begin{equation*}
\abs{F_{N}(t) - \Phi(t)} \leq \alpha \big(2 + (2 / \pi)^\frac12(\ln(1/\alpha))^\frac12 + (2\pi)^{-\frac12}\big(\eig_1(\Hesst) / \Tr(\Hesst^2)^\frac12\big) \ln(1/\alpha) \big)
\enspace.
\end{equation*}}%
Since the RHS of the above inequality is independent of $t$, taking the supremum of both sides over $t \in \R$, we obtain the desired inequality.
\end{proof}

\subsection{Proof of Lemma~\ref{lem:u1}}
\label{apdx:lem:u1}

\begin{proof}
First, note that $F_{f}(f(\xmean + \sigma Z)) = F_{N}(H_N)$ and $H_N = \Ze + h(Z)$. Using Lemma~\ref{lem:u1lip}, we have
{\small\begin{multline*}
\abs{\E[u_1(F_{f}(f(X))) \Ze] - \E[u_1(\Phi(\Ze)) \Ze]}
= \abs{\E[u_1(F_{N}(H_N)) \Ze] - \E[u_1(\Phi(\Ze)) \Ze]}
\\
\leq \E[\abs{u_1(F_{N}(H_N)) - u_1(\Phi(\Ze))} \cdot \abs{\Ze}]
\leq L_1 \E[\abs{F_{N}(H_N) - \Phi(\Ze)} \cdot \abs{\Ze}]
  \enspace.  
\end{multline*}}%
Noting that $\Phi$ is Lipschitz continuous with the Lipschitz constant $(2\pi)^{-\frac12}$, we have $\abs{\Phi(H_N) - \Phi(\Ze)} \leq (2\pi)^{-\frac12}\abs{H_N - \Ze} = (2\pi)^{-\frac12}\abs{h(Z)}$. On the other hand, Lemma~\ref{lem:weak} says that $\abs{F_{N}(H_N) - \Phi(H_N)} \leq G(\alpha)$. From these inequalities we obtain
{\small\begin{align}
\abs{F_{N}(H_N) - \Phi(\Ze)}
 = \abs{F_{N}(H_N) - \Phi(H_N) + \Phi(H_N) - \Phi(\Ze)} 
\leq G(\alpha) + (2\pi)^{-\frac12}\abs{h(Z)} \enspace. \label{eq:lem:u1:1}
\end{align}}%
Using the inequality \eqref{eq:lem:u1:1} and the Schwarz inequality and the identities $\E[\abs{\Ze}] = (2/\pi)^\frac12$, $\E[\Ze^2] = 1$, and
{\small\begin{align}
\E[\abs{h(Z)}^2] 
= \left(\frac{1}{2}\frac{\ns}{\cm}\right)^2 \E\left[\left(Z^\T\Hesst Z - 1\right)^2\right]
= \left(\frac{1}{2}\frac{\ns}{\cm}\right)^2 (2\Tr(\Hesst^2))
= \frac{\alpha^2}{2} \enspace,\label{eq:lem:u1:2}
\end{align}}%
we have
{\small\begin{align*}
\E[\abs{F_{N}(H_N) - \Phi(\Ze)} \cdot \abs{\Ze}]
&\leq G(\alpha) \E[\abs{\Ze}] + (2\pi)^{-\frac12} \E[\abs{h(Z)}\cdot\abs{\Ze}]
\\
&\leq G(\alpha) \E[\abs{\Ze}]  + (2\pi)^{-\frac12} \E[h(Z)^2]^\frac12\E[\Ze^2]^\frac12
\\
&= (2 / \pi)^\frac12 G(\alpha)  + (2\pi)^{-\frac12} \E[h(Z)^2]^\frac12
\\
&= (2 / \pi)^\frac12 G(\alpha) + (4\pi)^{-\frac12} \alpha
  \enspace.  
\end{align*}}%
Altogether, we obtain the inequality stated in the lemma. This completes the proof.
\end{proof}

\subsection{Proof of Lemma~\ref{lem:u2}}
\label{apdx:lem:u2}

\begin{proof}
Analogously to the proof of Lemma~\ref{lem:u1}, we have
{\small%
  \begin{align*}
\MoveEqLeft[2]\abs{\E[u_2(F_{f}(f(X))) (Z^\T\Hesst Z - 1) ] - \E[u_2(\Phi(\Ze)) (Z^\T\Hesst Z - 1) ]}
\\
&\leq L_2\E[\abs{F_{N}(H_N) - \Phi(\Ze)} \cdot \abs{Z^\T\Hesst Z - 1}]
\\
&\leq L_2\E[(G(\alpha) + (2\pi)^{-\frac12}\abs{h(Z)} ) \abs{Z^\T\Hesst Z - 1}]
\\
&= L_2(G(\alpha) \E[ \abs{Z^\T\Hesst Z - 1} ] + (2\pi)^{-\frac12} \E[\abs{h(Z)}\cdot \abs{Z^\T\Hesst Z - 1}])
  \enspace.  
  \end{align*}}%
Applying the inequalities $\E[ \abs{Z^\T\Hesst Z - 1} ] \leq \E[ (Z^\T\Hesst Z - 1)^2 ]^\frac{1}{2} = (2 \Tr(\Hesst^2))^\frac{1}{2}$ and
{\small\begin{equation*}
  \E[\abs{h(Z)}\cdot \abs{Z^\T\Hesst Z - 1}] = \frac12 (\ns/ \cm) \E[(Z^\T\Hesst Z - 1)^2] = (\ns/ \cm)\Tr(\Hesst^2) = \alpha \Tr(\Hesst^2)^\frac12 \enspace,
\end{equation*}}%
we obtain the inequality stated in the lemma. This completes the proof.
\end{proof}

\subsection{Proof of Lemma~\ref{lem:u3}}
\label{apdx:lem:u3}

\begin{proof}
Using Lemma~\ref{lem:u1lip}, we have
{\small\begin{align*}
\MoveEqLeft[2]\abs{\E[u_3(F_{f}(f(X)), F_{f}(f(\tilde{X}))) Z^\T\Hesst\tilde{Z}] - \E[u_3(\Phi(\Ze), \Phi(\Zet)) Z^\T\Hesst\tilde{Z}]}
\\
&\leq \E[\abs{(u_3(F_{N}(H_N), F_{N}(\tilde{H}_N)) - u_3(\Phi(\Ze), \Phi(\Zet))} \cdot \abs{Z^\T\Hesst\tilde{Z}}]
\\
&\leq L_3\E[(\abs{F_{N}(H_N) - \Phi(\Ze)} + \abs{F_{N}(\tilde{H}_N) - \Phi(\Zet)}) \cdot \abs{Z^\T\Hesst\tilde{Z}}]
  \enspace.
\end{align*}}%
Then, using the equality $\E[\abs{Z^\T\Hesst \tilde{Z}} \mid Z] = (2/\pi)^\frac12 \norm{\Hesst Z}$ (since $\abs{Z^\T\Hesst\tilde{Z}}$ given $Z$ is half-normally distributed), the symmetry of $Z$ and $\tilde{Z}$, the Schwarz inequality, and the inequality \eqref{eq:lem:u1:1}, we have
{\small\begin{align*}
\MoveEqLeft[2]\E[(\abs{F_{N}(H_N) - \Phi(\Ze)} + \abs{F_{N}(\tilde{H}_N) - \Phi(\Zet)}) \cdot \abs{Z^\T\Hesst\tilde{Z}}]
\\
&\leq 2 G(\alpha) \E[\abs{Z^\T\Hesst\tilde{Z}}] + 2(2\pi)^{-\frac12}\E[\abs{h(Z)}\cdot\abs{Z^\T\Hesst\tilde{Z}}]
    \enspace.
       \end{align*}}%
On one hand, we have
{\small%
  \begin{multline*}
    \E[\abs{Z^\T\Hesst\tilde{Z}}]
    = \E[\E[\abs{Z^\T\Hesst\tilde{Z}} \mid Z]]
    = (2/\pi)^\frac12 \E[\norm{\Hesst Z}]
    \\
    \leq (2/\pi)^\frac12 \E[\norm{\Hesst Z}^2]^\frac12
    = (2/\pi)^\frac12 \Tr(\Hesst^2)^\frac12
 \enspace,
\end{multline*}%
where we used $\E[\norm{\Hesst Z}^2] = \E[\Tr(\Hesst Z Z^\T \Hesst)] = \Tr(\Hesst \E[Z Z^\T] \Hesst) = \Tr(\Hesst^2)$. 
}%
On the other hand, we have
{\small%
  \begin{multline*}
    \E[\abs{h(Z)}\cdot\abs{Z^\T\Hesst\tilde{Z}}]
    = \E[\abs{h(Z)}\E[\abs{Z^\T\Hesst\tilde{Z}} \mid Z]]
    = (2/\pi)^\frac12\E[\abs{h(Z)}\cdot\norm{\Hesst Z}]\\
    \leq (2/\pi)^\frac12 \E[\abs{h(Z)}^2]^\frac12\E[\norm{\Hesst Z}^2]^\frac12
    = \pi^{-\frac12} \alpha \Tr(\Hesst^2)^\frac12 \enspace,
  \end{multline*}%
}%
where we used $\E[\abs{h(Z)}^2] = \alpha^2 / 2$ derived in \eqref{eq:lem:u1:2}.
Altogether, we obtain the inequality stated in the lemma. This completes the proof.
\end{proof}

\subsection{Proof of Lemma~\ref{lem:os}}
\label{apdx:lem:os}
\begin{proof}
Let $p$ be the probability density function of the one-dimensional standard normal distribution and $p_{i:\lambda}$ be the probability density function of $\NN_{i:\lambda}$ and $p_{i,j:\lambda}$ be the joint probability density function of $\NN_{i:\lambda}$ and $\NN_{j:\lambda}$. It is well known that $p_{i:\lambda}(x) = \lambda \binom{\lambda-1}{i-1} \Phi(x)^{i-1}(1 - \Phi(x))^{\lambda - i} p(x)$
and $p_{i,j:\lambda}(x, y) = \lambda (\lambda-1) \binom{\lambda - 2}{j - 2}\binom{j - 1}{i - 1} \Phi(x)^{i-1}(\Phi(y) - \Phi(x))^{(j - i - 1)}(1 - \Phi(x))^{\lambda - j} p(x)p(y)$
for $i < j$ and $x < y$, and $p_{i,j:\lambda}(x, y) = 0$ for $i < j$ and $x \geq y$. Note also that $p_{i,j:\lambda}(x, y) = p_{j,i:\lambda}(y, x)$. 

The functions $u_1$ and $u_2$ are then written using these p.d.f.s of the normal order statistics as
$\lambda u_1(\Phi(x)) p(x) = \sum_{k=1}^{\lambda} w_k p_{k:\lambda}(x)$ and $\lambda u_2(\Phi(x)) p(x) = \sum_{k=1}^{\lambda} w_k^2 p_{k:\lambda}(x)$. From these identities, we obtain \eqref{eq:os:1} and \eqref{eq:os:2}.
The identity \eqref{eq:os:3} is derived by using $\lambda \E[ u_2(\Phi(\Ze)) (\Ze^2 - 1) ] = \sum_{i=1}^\lambda w_i^2 (\E[\NN_{i:\lambda}^2] - 1)$ and $\E[ u_2(\Phi(\Ze)) (Z^\T \Hesst Z - 1) ] = \E[ u_2(\Phi(\Ze)) (\Ze^2 - 1) ] \ee^\T \Hesst \ee $, where the last equality is proved by using the expression $Z^\T \Hesst Z = \Ze^2 \ee^\T \Hesst \ee + \Ze \ee^\T\Hesst \Zee + \Zee^\T\Hesst\Zee$, the mutual independence between $\Ze$ and $\Zee$, and $\E[\Zee] = 0$ and $\E[\Zee^\T\Hess\Zee] = 1 - \ee^\T \Hesst \ee$.

Using $p_{i,j:\lambda}$, we can write
{\small$$\lambda (\lambda - 1) u_3(\Phi(x), \Phi(y)) p(x)p(y) = \sum_{k=1}^{\lambda-1}\sum_{l=k+1}^{\lambda} w_k w_l \max(p_{k,l:\lambda}(x, y), p_{l,k:\lambda}(x, y))\enspace.$$}%
The equality \eqref{eq:os:4} is obtained by substituting the equality
{\small\begin{align*}
\MoveEqLeft[0]\lambda (\lambda - 1)\E[ u_3(\Phi(\Ze), \Phi(\Zet)) \Ze\Zet ]
\\
&\textstyle= \sum_{k = 1}^{\lambda-1}\sum_{l = k + 1}^{\lambda} w_k w_l \iint \Ze\Zet \max(p_{k,l:\lambda}(x, y), p_{l,k:\lambda}(x, y)) dx dy
\\
&\textstyle= \sum_{k = 1}^{\lambda-1}\sum_{l = k + 1}^{\lambda} w_k w_l \left(\iint_{x<y} x y p_{k,l:\lambda}(x, y) dx dy + \iint_{x\geq y} x y p_{l,k:\lambda}(x, y) dx dy\right)
\\
&\textstyle= \sum_{k = 1}^{\lambda-1}\sum_{l = k + 1}^{\lambda} w_k w_l \left(\iint x y p_{k,l:\lambda}(x, y) dx dy + \iint x y p_{l,k:\lambda}(x, y) dx dy\right)
\\
&\textstyle= 2 \sum_{k = 1}^{\lambda-1}\sum_{l = k + 1}^{\lambda} w_k w_l \iint x y p_{k,l:\lambda}(x, y) dx dy
\\
&\textstyle= 2\sum_{k = 1}^{\lambda-1}\sum_{l = k + 1}^{\lambda} w_k w_l \E[\NN_{k:\lambda}\NN_{l:\lambda}]
       \end{align*}}%
into $\E[ u_3(\Phi(\Ze), \Phi(\Zet)) Z^\T \Hesst \tilde{Z} ] = \E[ u_3(\Phi(\Ze), \Phi(\Zet)) \Ze\Zet ] \ee^\T\Hesst\ee$. The last equality is obtained by using the expression $Z^\T \Hesst \tilde{Z} = \Ze\Zet \ee^\T\Hesst\ee + \Ze\ee^\T\Hesst\Zeet + \Zet\ee^\T\Hesst\Zee + \Zee^\T\Hesst\Zeet$, the mutual independence between $\Ze$, $\tilde\Ze$, $\Zee$, and $\tilde\Zee$, and the equalities $\E[\Zee] = \E[\tilde\Zee] = 0$. 
\end{proof}

\end{document}